\documentclass[twoside]{irmaems}
\usepackage{amssymb} 
\usepackage{amsmath} 
\usepackage{latexsym}
\usepackage{amsfonts}
\title{Brief Article}
\author{The Author}

\renewcommand\theequation{\arabic{section}.\arabic{equation}}

\theoremstyle{definition} 

 \newtheorem{definition}{Definition}[section]
 \newtheorem{remark}[definition]{Remark}

\theoremstyle{plain}      

 \newtheorem{proposition}[definition]{Proposition}
 \newtheorem{theorem}[definition]{Theorem}
 \newtheorem{corollary}[definition]{Corollary}
 \newtheorem{lemma}[definition]{Lemma}

\newcommand{\weil}{Weil-Petersson }
\newcommand {\teich}{Teichm\"{u}ller }
\newcommand{\T}{{\mathcal T}}
\newcommand{\Tbar}{\overline{\mathcal T}}

\newcommand{\g}{\ensuremath{\gamma}}
\newcommand{\del}{\partial}
\newcommand{\map}{\ensuremath{{\rm Map}(\Sigma)}}
\newcommand{\e}{\ensuremath{\varepsilon}}
\newcommand{\h}{\ensuremath{{\mathbb H}^2}}

\newcommand{\tr}{\operatorname{Tr}}

\begin{document}

\title{Local and Global Aspects of \\ Weil-Petersson Geometry}

\author{Sumio Yamada\thanks{
Work partially supported by Grant-in-aid-for-scientific-research 20540201, 23654061} }

\address{
Department of Mathematics, Gakushuin University\\
Toshima, Tokyo 171-8588\\
email:\,\tt{yamada@math.gakushuin.ac.jp}
}

\maketitle




\tableofcontents   

\section{Introduction}\label{s-1}

The study of moduli spaces of Riemann surfaces  holds a central position in the history of modern mathematics, as the subject is a crossroad among several areas such as complex function theory, algebraic geometry, topology, number theory, partial differential equations and differential geometry.  In those areas, the original motivations for the investigations differ from each other, which led to a seemingly distinct set of subfields all studying the moduli space of Riemann surfaces in essence.   

In this context, it seems that the field of differential geometry holds a curious position, as the very origin of surface theory was marked by Gauss in his study of  embedded surfaces in ${\mathbb R}^3$, leading to
the Gauss curvature and  the Theorema Egregium, which had become the starting point of  modern differential geometry.  Soon afterwards, however,  the theory of algebraic curves initiated by Riemann and that of Fuchsian groups took over the classical differential geometry in the study of surfaces, which in a hindsight is a natural development considering the effectiveness of algebraic equations and the ${\rm SL}(2, {\mathbb R})$ representation theory associated with the uniformaization theory of Poincar\'{e} and Koebe.  

In 1939, when \teich investigated  \teich maps in order to relate a pair of conformal structures, he was most likely aware of the lack of differential geometric approaches in the theory of  moduli space.
Ahlfors followed \teich in
rewriting the theory of Fuchsian groups by regarding it as a deformation theory of conformal structures, 
using the theory of Beltrami differentials. Andre Weil was also instrumental in recognizing the geometric importance of Teichm\"{u}ller's work.   Bers and Ahlfors pushed 
hard over the 1950s and 60s to make Teichm\"{u}ller's  theory complete.  At the same time, the theory of algebraic 
curves was developed independently by a set of algebraic geometers including Grothendieck, Serre, and 
Mumford.   In the meantime, it is fair to say that among the differential geometers of the second half of the 20th century, the moduli space of Riemann surfaces remained an esoteric topic compared to their scientific interests.  

This chapter, especially the first half of it,  is written with a second year graduate student in mind, who has taken a year long course in Riemannian geometry, but not necessarily well-versed in complex analysis, algebraic geometry, 
or \teich theory.  The author has ventured to write down the preceding paragraphs on the very abridged and very incomprehensive history of the subject, simply to make a point that  much of so-called \teich theory can be understood from a purely Riemannian geometric veiwpoint, despite the  non-differential-geometric development of the subject over time. We owe this approach to the following list of contributions:  1) Morrey's proof of the existence of solutions to the Beltrami equation~\cite{Mo}, 2) Earle-Eells's attempt~\cite{EE} to recoginize \teich space as a submanifold  of the space of smooth metrics, 3) the theory of harmonic maps into non-positively curved manifolds by Eells-Sampson \cite{ES} allowing to track varying metrics by smooth harmonic maps, 4) the body of work by Thurston who took the hyperbolic geometry of  surfaces to its 
full power to reinterpret the Ahlfors-Bers theory, and  5) Fischer-Tromba~\cite{FT1, FT2, Tr1} who in the 70's and 80's rewrote \teich theory, in particular the \weil geometry  from the deformation theory of hyperbolic metrics, partly initiated by the theory of traceless-transverse (TT) tensors  by Fischer and Marsden \cite{FM} originally developed with applications to general relativity in mind.      

In the  post-Ahlfors \weil geometry, Wolpert's contributions stand out for his singular pursuit wanting to understand the geometry of \teich spaces.  The chapter~\cite{Wohb} written by Wolpert that has appeared in this Handbook series covers most of the ground in the development of \weil geometry in the last few decades, particularly the last one, and the present article is meant to complement Wolpert's. One should also consult the book of Wolpert's~\cite{Wobook} 
which offers a more comprehensive exposition.  The reader will notice that many of  the results stated in his chapter and the present  chapter overlap, and the author uses several statements of Wolpert's at crucial steps in developing the theory.  However, also apparent should be that the languages used in describing the geometry and the techniques employed in the proofs of the theorems are distinct from each other in the two chapters. 

We now state several objectives of this exposition.
The first is  to write down the basics of the Riemannian approach to the \weil geometry, complementing Wolpert's expositions on the subject where much of the argument is made with Beltrami coefficients. 
This amounts to the {\it local} aspect of \weil geometry of the title of the chapter. In doing so, we add some new material to the existing literature such as \cite{FT1, FT2, Jo, Tr1}. Much of the exposition in this part of the chapter can be applied to settings of higher-dimensional moduli spaces, such as that of Calabi-Yau manifolds (for example \cite{CO}.)  Therefore insisting on the Riemannian geometric approach is meaningful in the sense that
the theories of Beltrami equations and complex analysis are specific to  dimension two, and most of it is not transferable to higher dimensional situations.  Some of the results in this exposition have been taken from the article \cite{Y1} which appeared in the Journal of Differential Geometry in 1999.  Unfortunately  while in press numerous typographical mistakes were introduced in that article. All the results and  proofs in the paper are valid, but over the years this situation has posed unnecessary challenges to the interested readers.   We have tried to rectify the situation, and present here a comprehensive version of the content of the paper \cite{Y1},  including a proof of the \weil convexity of energy of harmonic maps.  Incidentally there has been some dispute (see for example \cite{DW2}) whether the statement was first proved by Tromba \cite{Tr2} in 1996, where the domain of the harmonic maps is the surface itself.  We point out  in Section 3.6 of the present chapter, however reluctantly,  a difference between the two proofs explicitly, in particular the difference in the \weil geodesic equation, to let the matter rest. 

In utilizing harmonic maps for the \teich theory  we have included a section on the \weil geometry of the
\teich space of the torus, as  this has not been explicitly written down in the literature although it is  well-known to the experts.  Each conformal structure on it  is uniformized by a flat metric, which can be 
identified with a harmonic map from a fixed reference torus. Then we show here that the resulting \weil 
geometry is isometric to the homogeneous space ${\rm SL}(2, {\mathbb R})/{\rm SO}(2)$ with its left 
invariant metric, or equivalently the  Poincar\'e disc.  As the harmonic map is an affine map, it is  the \teich 
map as well (see for example  \cite{BPT} for  \teich geometry).   

The second goal of this exposition is to present a series of  developments the \weil geometry has gone through over the last decade, after a paper by the author~\cite{Y5} (later as \cite{Y2})  appeared where the CAT(0) geometry was introduced for augmented \teich space for the first time.  This part of the chapter corresponds to the {\it global} aspect in the title.  We take steps in discussing 1) \teich space, which is incomplete \weil metrically and \weil geodesically,   2) the \weil metric completion of  \teich space, which is identified with the augmented \teich space,  and 3) the \weil geodesic completion of  \teich space, which is realized as a Coxeter complex where the set which works as the simplex is the \weil metric completion. The successive enlargement of the original \teich space as traced through the author's work \cite{Y1,Y2, Y3, Y4} is motivated from the  point of view that there should be a  genus-varying  \teich theory, a direction already actively pursued in the Teichm\"{u}ller-Grothendieck theory from the algebro-geometric approach.  With this in mind, we demonstrate that the Coxeter complex can be embedded in the universal \teich space.

The last objective of this chapter is to write down the \weil geometry relevant to the universal \teich space. The orthodox deformation theory in the universal \teich space is written in terms of Beltrami differentials (see for example O. Lehto's book~\cite{Le}).  The Nag-Verjovsky's paper \cite{NV} has shed much light in  clarifying how the \weil metric can be considered as the $L^2$-pairing of  the linearized Beltrami differentials induced by a smooth vector fields on the unit circle.   Takhtajan-Teo \cite{TT} then took over the idea and further developed it so that they succeeded in generalizing a collection of results by Wolpert
 \cite{W4} concerning the second derivatives of the \weil metric tensor,
in particular the \weil curvature by regarding the universal \teich space as a Hilbert manifold.  We look at the Nag-Vejovsky paper \cite{NV}  again, and describe the tangent space of the universal \teich space at the identity 
as  the set of traceless transverse tensors constituting a component in an $L^2$-decomposition theorem of Hodge type.   Additionally  we demonstrate the \weil metric tensor as a Hessian of the $\overline{\del}$-energy of harmonic maps.  Those results have their counterparts in the compact surface cases, but were never  investigated in the universal context before \cite{Y1}.   The relation between a negatively curved complete manifold and its geometric boundary has been an active area of investigation in the last decade, partly due to the excitement from so-called AdS-CFT correspondence, which in turn has triggered much incentive to study the conformally compact/Einstein-Poincar\'{e} manifolds. As the Poincar\'{e} disc is the simplest example of conformally compact manifolds as well as Einstein-Poincar\'{e} manifolds, we believe that  the universal \teich space offers a prototype of the moduli space of such manifolds.  In doing so, it will become necessary to formulate the theory without the use of one-variable complex analysis, and the deformation theory of Riemannian metrics as explored in this chapter will be a basic model for higher dimensional analogues.   
 
The author thanks Athanase Papadopoulos for his encouragement to write down this exposition.

\section{Classical and Universal \teich Spaces}

\subsection{Classical \teich spaces\index{\teich space!classical} for closed surfaces}
Let $\Sigma$ be a compact surface without boundary of genus $g \geq 1$ (when
$g =0$ the situation is very simple.) 
By the existence theorem of an isothermal coordinate system by Korn and Lichtenstein, any
Riemannian metric $g$ can be identified with a Riemann surface, namely a Riemannian surface is a Riemann surface.  The universal covering space of the surface is either the whole plane or the upper half space, and thus the surface can be uniquely equipped with a Euclidean metric when $g=1$ or a hyperbolic metric when $g >1$.  This statement is the so-called Uniformization Theorem.  Hence we can think of the space ${\cal M}_{K}, \,\,\, (K \equiv 0,-1)$ 
of constant curvature metrics as a subset of the space of smooth metrics $\cal M$ 
on $\Sigma$, the latter space being fibered by the elements of ${\cal M}_{K}$ so that each fiber consist of  the 
metrics conformal to a constant curvature/uniformized metric $G \in {\cal M}$.    

The \teich space\index{\teich space} is then defined as the quotient space
\[
{\T}_g = {\cal M}_{K} / {\rm Diff}_0 \Sigma
\]
where the equivalence relation is given as 
\[
G_1 \sim G_2 \Leftrightarrow G_2 = \varphi^* G_1
\]
for some $\varphi$ in ${\rm Diff}_0 \Sigma$. Here ${\rm Diff}_0 \Sigma$ is
the identity component of the  orientation-preserving diffeomorphism group 
${\rm Diff} \Sigma$. Recall that the map $\varphi: (\Sigma, \varphi^* G_2) \rightarrow (\Sigma, G_2)$ is an isometry.  Note that in defining the identity element of ${\rm 
Diff}_0 \Sigma$ one requires a reference Riemann surface $(\Sigma_0, G_0)$ such 
that it acts as the domain of ${\rm Id}:\Sigma_0 \rightarrow \Sigma$. Namely 
$(\Sigma_0, G_0)$ gives homotopy markings on the target surface.  
 
By an important  theorem of Earle-Eells \cite{EE}, it is known that the identity component ${\rm Diff}_0 \Sigma \subset {\rm Diff} \Sigma$ consists of diffeomorphisms homotopic to the identity map. 

The moduli space\index{moduli space} $\mathfrak{M}_g$ is defined as 
\[
\mathfrak{M}_g =  {\cal M}_{K} / {\rm Diff} \Sigma
\]
where the equivalence relation is given as 
\[
G_1 \sim G_2 \Leftrightarrow G_2 = \varphi^* G_1
\]
for some $\varphi$ in ${\rm Diff} \Sigma$. Thus the \teich space projects down to the 
moduli space with the fibers identified with the discrete infinite group ${\rm Diff} 
\Sigma/{\rm Diff}_0 \Sigma$, called mapping class group\index{mapping class group}, or \teich modular 
group\index{\teich modular group}.  We denote this group by $\map$. We define now for a later use the full 
diffeomorphism group $\widehat{\rm Diff} \Sigma$ which, in addition to the elements 
of ${\rm Diff} \Sigma$, also contains the orientation-reversing diffeomorphisms 
of $\Sigma$.  Then the quotient group $\widehat{\rm Diff} \Sigma /{\rm Diff}_0 \Sigma$ is called the extended mapping class group $\widehat{\map}$. 

\subsection{The Universal \teich space\index{universal \teich space}\index{\teich space!universal}} 
The Uniformization Theorem, by Poincar\'{e} and Koebe (see \cite{Le})) says that given a closed surface, all the smooth 
metrics on it can be uniquely {\it uniformized} by a constant curvature metric.  When the surface is of genus 
greater than one, the constant curvature metrics are hyperbolic metrics, and the hyperbolic surface 
can then be written as ${\mathbb H}^2 /\Gamma$ for some Fuchsian group $\Gamma \subset {\rm SL}(2, {\mathbb R})$. Namely, given a Riemannian surface 
$(\Sigma, G)$,  the metric $G$ is conformal to a hyperbolic metric $G_0$ on $\Sigma$. 
The space of Fuchsian groups appearing in the Uniformization Theorem can be regarded as  
\[
{\rm QC}(\Gamma_0) / {\rm SL}(2, {\mathbb R}) = \{ w \in {\rm QC}({\bf D}): w \Gamma_0 w^{-1} \mbox{ is an element of ${\rm SL} (2, {\mathbb R})$} \}
\]
for some fixed reference Fuchsian group $\Gamma_0$, where 
the set is identified with the set of all  pulled-back metrics on the unit disc ${\bf D} := \{ z \in {\mathbb C} : |z| < 1 \}$,  obtained by 
pulling back the Poincar\'e disc metric 
\[
G_0 = \frac{4}{(1-|z|^2)^2} |dz|^2 
\] 
by $\Gamma_0$-equivariant quasi-conformal self-maps of  ${\bf D}$  up to an equivalence
via the M\"{o}bius transforms of the disc.  Note that the original hyperbolic surface $(\Sigma, G)$ in this context appears as $ {\rm Id}_{\bf D}  \in {\rm QC}({\bf D})$. 

Note that this set is the set of all the hyperbolic metrics 
on the surface $\Sigma$;
\[
{\cal M}_{-1} = {\rm QC}(\Gamma_0) / {\rm SL}(2, {\mathbb R}). 
\] 
On the other hand, recall that two 
metrics are deemed geometrically equivalent  if one can be obtained from the 
other by a diffeomorphism on the manifold. Here the manifold is a hyperbolic surface 
realized as a quotient manifold of the open disc.  As the hyperbolic metric $G_0$ of the Poincar\'{e} 
disc model blows up near the geometric boundary $S^1 = \partial {\bf D} = \{z: |z| =1\}$, 
the diffeomorphisms of the surface regarded  (by going to its universal covering space)  as diffeomorphisms of the open disc are the elements of the 
set ${\rm QC}_0 (\Gamma_0)$ of quasi-conformal self-maps of the disc which 
extend to the identity map of the geometric boundary $S^1 = \partial {\bf D}$.  This can be understood from the 
picture that those diffeormorphisms can be lifted to the universal covering, which are periodic on the tessellation by a fundamental region of $\Sigma$,
leaving the tessellation invariant, and that on the  Poincar\'{e} disc the tessellation pattern gets increasingly 
dense as one approaches the geometric boundary $\del {\bf D}$.  
   
Hence the \teich 
space in this context is  regarded as
\begin{eqnarray*}
{\cal T} &=&  {\cal M}_{-1} / {\rm Diff}_0 \Sigma \\
&=& \Big[ {\rm QC}(\Gamma_0) / {\rm SL}(2, {\mathbb R}) \Big] /
{\rm QC}_0(\Gamma_0) \\
&=&  
{\rm QS}(\Gamma_0)/ {\rm SL} (2, {\mathbb R})
\end{eqnarray*}
where ${\rm QS}(\Gamma_0) := {\rm QC}(\Gamma_0) / {\rm QC}_0(\Gamma_0) $ is the set of $\Gamma_0$-equivariant quasi-symmetric self-maps of 
$S^1 = \partial {\bf D}$, a subset of the space of quasi-symmetric self-maps of the circle ${\rm QS}(S^1)$.  The quotient space
is defined 

under the following equivalence relation :  a pair of elements $[\phi_1]$ and $[\phi_2]$ in ${\rm QC}(\Gamma_0)/{\rm SL}(2, {\mathbb R})$ are  
equivalent if they are related by $\phi_2 = \phi_1 \circ \psi$ for some $\psi$ in $ {\rm QC}_0(\Gamma_0)$.  
Note that we have used the fact that ${\rm SL}(2, {\mathbb R}) \cap {\rm QC}_0 (\Gamma_0) = \{ {\rm Id} \}$.
Here we think of points in $\T$
as left cosets of the form $ {\rm SL} (2, {\mathbb R}) \circ w = [w]$ where
each $w$ is a  quasi-symmetric homeomorphism of the circle.

We replace the closed surface $\Sigma$ above by the hyperbolic plane ${\mathbb H}^2$.  This 
can be regarded as replacing the Fuchsian group $\Gamma \subset {\rm SL}(2, {\mathbb R})$
by the trivial group $\Gamma_0 = {\rm Id}  \in {\rm SL}(2, {\mathbb R})$.  Namely in the 
 above construction, such a replacement results in a new space
\begin{eqnarray*}
{\cal UT}  &=& \Big[ {\rm QC}(\{ {\rm Id }_{\bf D} \}) / {\rm SL}(2, {\mathbb R}) \Big] /
{\rm QC}_0( \{ {\rm Id}_{\bf D} \}) \\ &=& \Big[ {\rm QC}({\bf D}) / {\rm SL}(2, {\mathbb R}) \Big] /
{\rm QC}_0( {\bf D}) \\
&=&  
{\rm QS}(\{ {\rm Id}_{\bf D} \})/ {\rm SL} (2, {\mathbb R}) = {\rm QS}(S^1) / {\rm SL} (2, {\mathbb R})
\end{eqnarray*}
Note that we have used the fact that ${\rm SL}(2, {\mathbb R}) \cap {\rm QC}_0 ({\bf D}) = \{ {\rm Id}_{\bf D} \}$.
As the resulting space contains all \teich spaces
of surfaces of the form ${\mathbb H}^2 / \Gamma$, we call the space ${\cal UT} $the universal \teich space.  For a comprehensive treatment of the subject, including the definitions of quasi-conformal and quasi-symmetric maps, we refer the reader to Lehto's book \cite{Le}.

The advantage in introducing the quasi-comformal maps of the  disc and the quasi-symmetric maps 
of the unit circle is that one can then describe the deformations of hyperbolic metrics by the pull-back
action of quasi-conformal maps $w: {\bf D} \rightarrow {\bf D}$ each of which is a solution to the Beltrami equation\index{Beltrami equation}
\[
w_{\overline{z}} = \mu (z) w_z .
\]  
for some  Beltrami coefficient $\mu$,  a ${\mathbb C}$-valued measurable function on $D$ with $|\mu| < 1$, 
an element of the unit ball $L^{\infty} ({\bf D})_1$ in the complex Banach space $L^{\infty} ({\bf D})$.
 
In other words, we can identify each element $w$ of ${\rm QC}({\bf D})$ with $\mu$ uniquely, provided the solution $w:{\bf D} \rightarrow {\bf D}$ fixes three points on the boundary.  
The existence and  uniqueness of the solution to the Beltrami equation is due to Morrey \cite{Mo}. The reader is referred to the books by Nag \cite{Na} and  Lehto \cite{Le}  for an exposition on the subject. 
We denote the solution to the Beltrami equation above by $w_\mu$.  

In particular, define the set of equivariant Beltrami differentials by
\[
L^{\infty} (\Gamma_0) = \{\mu \in L^{\infty} ( {\bf D}) \,\,\, | \,\,\, \mu(gz) \frac{\overline{g'(z)}}{ g'(z)} = \mu (z) 
\mbox{ a.e. on ${\bf D}$ for all $g$ in $\Gamma_0$} \}
\] 
which is a closed subspace in $L^\infty ({\bf D})$. The unit ball $L^{\infty} (\Gamma_0) \cap L^{\infty}_1$ is denoted by 
$L^\infty(\Gamma_0)_1$.

The \teich space for a Fuchsian group $\Gamma_0$, including the case of  the trivial group $\{ {\rm Id} \}$, is identified with
\[
\T (\Gamma_0) = L^\infty (\Gamma_0)_1/ \sim
\]  
where $\mu \sim \nu$ if and only if $w_\mu = w_\nu$ on $\partial {\bf D} = S^1$.  
The identification between the space of the Beltrami coefficients and the Fuchsian group 
is given by the following: If $\mu \in L^\infty (\Gamma_0)_1$ then $w_\mu$ conjugates $\Gamma_0$ 
to another Fuchsian group 
\[
\Gamma_\mu = w_\mu \Gamma_0 w_\mu^{-1}.
\]
Hence the equivalence class $[\mu]$ represents the hyperbolic surface ${\mathbb H}^2 / \Gamma_\mu$.

Under the light of this identification, as presented in a paper by Nag-Verjovsky \cite{NV}, the moduli 
of hyperbolic surfaces are determined by the quasi-symmetric homeomorphisms of the unit circle; a 
model mentioned but often insufficiently explained in the literature
of the string theory. We remark that this viewpoint was established by Beurling and Ahlfors in the 50s
 \cite{AB}.

In the linear theory to 
be developed in the following section, we will return to the viewpoint of the Beltrami equation.

\section{Riemannian Structures of $L^2$-pairing}

\subsection{$L^2$-pairing and its Levi-Civita connection}
\subsubsection{$L^2$-pairing of deformation tensors}
The tangent space $T_G {\cal M}$ of the space ${\cal M}$ at a metric $G$ is the 
space of smooth symmetric $(0,2)$-tensors on $\Sigma$. This linear space has a natural $L^2$-pairing\index{$L^2$-pairing} defined as follows.
\[
\langle h_1, h_2 \rangle_{L^2(G)}= \int_\Sigma \langle h_1(x), h_2(x) \rangle_{G(x)} d \mu_G (x)
\]
where the $h_i$'s are symmetric $(0,2)$-tensors indicating the directions of 
deformation of $G$ along the path $G + \varepsilon h_i + o(\varepsilon)$.
The integrand can be rewritten, using a local coordinate chart,  as
\begin{eqnarray*}
\langle h_1(x), h_2(x) \rangle_{G(x)} & = & \sum_{1 \leq i,j,k,l \leq 2} G^{ij}G^{kl}(h_1)_{ik}
(h_2)_{jl} \\
 & = & {\tr} \Big( (G^{-1} \cdot h_1)\cdot (G^{\-1} \cdot h_2) \Big)   
\end{eqnarray*}
where $A \cdot B$ denotes matrix multiplication and $\tr A$ is the trace of the matrix $A$.
This quantity is well defined, meaning it is invariant under  change of coordinate charts.
In particular it can be simplified by choosing a geodesic normal coordinate system where
$G_{ij}(p) = \delta_{ij}$ at its center $p$ as
\[
\langle h_1(p), h_2(p) \rangle_{G(p)} = \sum_{j,k} (h_1)^j_k(p)(h_2)^k_j (p) \,\,\, (= \tr (h_1 \cdot h_2))
\]   
the trace of the product of $2 \times 2$ matrices. From now on, 
we will use the Einstein notation of indices, omitting the summation symbols.

\subsubsection{Levi-Civita connection of the $L^2$-pairing}
The space of smooth metrics ${\cal M}$ defined on a manifold $N$ has always the $L^2$-pairing 
defined above.  Formally one can regard the $L^2$-pairing as a Riemannian metric on ${\cal M}$ and 
write down the Levi-Civita connection\index{Levi-Civita connection for the $L^2$-pairing} for it.  

We fix a coordinate chart around a point $p$ in $\Sigma$.  Let $h_1$, $h_2$ and $h_3$ be locally constant 
symmetric $(0, 2)$-tensors defined over each chart.  Then the brackets among the $h_i$'s vanish, 
namely
\[
[h_i, h_j] = 0
\]
where the bracket here is the Lie derivative of the tensor $h_j$ in the direction of $h_i$, where each deformation tensor is regarded as  a vector field on $\cal M$ evaluated at $G$.  Note that in what follows,  as all the quantities appearing below are tensorial, it suffices to consider point-wise calculations, namely
we may restrict to the locally constant tensors. 

The formula, which appears in the existence and uniqueness theorem of
Levi-Civita connection in any standard differential geometry textbook,  is 
\begin{align*}
\langle D_{h_1} h_2, h_3 \rangle & = &  \frac{1}{2} \Big( h_1 \langle h_2, h_3 \rangle + h_2 \langle h_1, h_2 
\rangle  - h_3 \langle h_1, h_2 \rangle \\
 &  & \,\,\, + \langle [h_1, h_2], h_3 \rangle - \langle [h_1, h_3], h_2 \rangle - \langle [h_2, h_3], h_1 \rangle \Big)
\\
 & = &  \frac12 \Big( h_1 \langle h_2, h_3 \rangle + h_2 \langle h_1, h_2 \rangle - h_3 \langle h_1, h_2 \rangle 
\Big)  
\end{align*}  
On the other hand, the pairing is a function of $G$, and the $h_i$'s are deformation tensors of $G$.  We write
down the derivatives $h_i \langle h_j, h_k \rangle$ as 
\begin{align*}
h_k \langle h_i, h_j \rangle_{L^2 (G)} &=& \frac{d}{dt} \int_\Sigma {\tr} \Big( G_t^{-1} \cdot h_i\cdot G_t^{-1} \cdot h_j \Big)  d \mu_{G_t(x)}  \Big|_{t=0} \\
 & = & \int_\Sigma {\tr} \Big(  G^{-1} \cdot (-h_k) \cdot G^{-1} \cdot h_i\cdot G^{-1} \cdot h_j \Big)  d \mu_{G(x)} \\
 & & \,\,\,  +\int_\Sigma {\tr} \Big(  G^{-1} \cdot h_i \cdot G^{-1} \cdot (-h_k) \cdot G^{-1} \cdot h_j \Big)  d \mu_{G(x)} \\ 
 & & \,\,\, \,\,\, + \int_\Sigma {\tr} \Big(  G^{-1} \cdot h_i \cdot G^{-1} \cdot h_j \Big) \frac12 ({\rm tr}_G h_k) d \mu_{G(x)} \\
 & =  & - \langle h_1 \cdot G^{-1}  \cdot  h_j, h_k \rangle_{L^2 (G)}  - \langle h_j  \cdot G^{-1}  \cdot  h_i, h_k \rangle_{L^2 (G)}   \\
 & & \,\,\, +  \frac14 \langle ({\rm tr}_G h_k) h_i, h_j \rangle_{L^2 (G)}  + \frac14 \langle  ({\rm tr}_G h_k) h_j, h_i  \rangle_{L^2 (G)}
\end{align*}
where $G_t = G + t h_k$.
When the $h_i$'s are trace-free,    the commutativity  $\tr (A \cdot B) = \tr (B \cdot A)$ is applied to the above integrands, and  $h_k \langle h_i, h_j \rangle_{L^2(G)}$ is invariant under  permutations of $i$,$j$ and $k$;
\[
h_k \langle h_i, h_j \rangle_{L^2 (G)} = h_{\sigma(k)} \langle h_{\sigma(i)}, h_{\sigma(j)} \rangle_{L^2 (G)} 
\] 
for any element $\sigma$ of the symmetric group ${\mathfrak S}_3$. 
By substituting the above expression  into the  formula for the connection $D$, we obtain the following 
relatively simple expression:

\begin{lemma}[Levi-Civita connection for $L^2$-pairing]
For trace-free symmetric $(0,2)$ tensors $h_i$ and $h_j$, the Levi-Civita connection for the $L^2$-pairing  at the tangent space $T_G {\cal M}$ of  the space ${\cal M}$ of smooth metrics is written as
\[
D_{h_1} h_2 = -\frac12 \Big[ h_1 \cdot  G^{-1} \cdot h_2  +  h_2 \cdot G^{-1} \cdot h_1 \Big] + \frac14 \Big[ (\tr_G h_1) h_2 +  (\tr_G h_2) h_1 -  \langle h_1, h_2 \rangle_{G(x)} G \Big].
\]
\end{lemma}

Note that the expression for the connection is symmetric in $1$ and $2$.  The author thanks Akira Yoshizato for pointing out an error in the calculation that appeared in the previous version.

\subsection{Tangential conditions and the \weil metric} 
When $G$ is a uniformizing metric of its conformal class, then the tangent 
space $T_G {\cal M}$ decomposes into the deformation of $G$ preserving the 
constant curvature condition, and its complement. This can be formally stated 
as follows.

In dimension two, the Riemann curvature tensor is completely determined by one 
scalar function, the sectional curvature $K$. Then the Ricci curvature tensor is of the form 
\[
R_{ij} = K G_{ij}
\]
namely $G$ is an Einstein metric.  The well-known variational formula (see \cite{Be}) of the 
Ricci tensor under a deformation $G + \varepsilon h$ at $\varepsilon = 0$ 
gives, after taking its trace:
\[
G^{ij} \dot{R}_{ij} = - \triangle_G {\rm Tr}_G h + \delta_G \delta_G h.
\]
Hence we have the following variational formula for the sectional curvature under the deformation of 
$G$ in the direction of $h$:
\begin{eqnarray*}
\dot{K} &=& G^{ij} \dot{R}_{ij} + \dot{G}^{ij} R_{ij} \\
&=&  G^{ij} \dot{R}_{ij} -h^{ij} K G_{ij} \\
&=&   - \triangle \tr_G h + \delta_G \delta_G h - K \tr_G h.
\end{eqnarray*}
We denote the quantity 
$
-(\triangle_G + K) \tr_G h + \delta_G \delta_G h
$
by ${\cal L}_G h$, where the differential operator ${\cal L}_G$ is sometimes called 
Lichnerowicz operator\index{Lichnerowicz operator}.
Hence if the deformation tensor $h$ is tangential to ${\cal M}_K$, then $h$ satisfies the following linear equation, which is the curvature-preserving condition      
\[
{\cal L}_G h= 0. 
\]  

Having characterized the tangential condition to ${\cal M}_K$, we 
additionally require the deformation tensor $h$ to be 
$L^2$-perpendicular to 
the diffeomorphism group ${\rm Diff}_0 \Sigma$ action. Consider a one-parameter family of 
diffeomorphisms $\varphi_t: \Sigma \rightarrow \Sigma$ with $\varphi_0 = {\rm 
Id}|_{\Sigma}$ and let  $\frac{d}{dt} \varphi_t |_{t=0}= X$ be a vector field on 
$\Sigma$. Recall that the Lie derivative $L_X G$ of the tensor $G$ in the direction $X$ is defined by
\[
L_X G = \frac{d}{dt} \varphi_t^* G \Big|_{t =0}. 
\]
Take a chart which gives a geodesic normal coordinate centered at $p$.  Then 
\[
L_X G (p)= X_{i;j} + X_{j;i}
\]
as $G_{ij} = \delta_{ij}$ and $G_{ij;k} = 0$ at $p$.  The condition that a symmetric $(0,2)$-tensor $h$ is $L^2$-perpendicular to the  diffeomorphism group ${\rm Diff}_0 \Sigma$ action is described as
\begin{align*}
0 = \langle h, L_X G \rangle_{L^2(G)}
\end{align*}
for all $X \in {\mathfrak X}(\Sigma)$. The right hand side can be rewritten, with respect to a geodesic normal coordinate,  as
\begin{eqnarray*}
\langle h, L_X G \rangle_{L^2(G)} & = & 
\int_{\Sigma} \langle h(x), L_X G(x) \rangle_{G(x)} d \mu_G (x) \\
 & = & \int_\Sigma h_{ij}(X_{i;j} + X_{j;i}) \,\, d \mu_G(x) \\
 & = & 2 \int_\Sigma h_{ij} X_{i;j} \,\, d \mu_G (x) \\
 & = & - 2 \int_\Sigma h_{ij;j} X_i \,\, d \mu_G (x) \\
 & = & - 2 \langle \delta_G h, X \rangle_{L^2(G)},
\end{eqnarray*}
where integration by parts has been used.  There is no boundary contribution 
as the surface $\Sigma$ is closed.  Therefore, for the tensor $h$ to be 
$L^2$-perpendicular to the diffeomorphism group ${\rm Diff}_0 \Sigma$ action, $h$ 
is required to be  {\it divergence-free}; $\delta_G h =0$. Note that $\delta_G h$ is here regarded as a tensor of $(1,0)$-type, that is, a vector field. In the normal coordinate 
system, the divergence-free condition is the same as $(\delta_G h)_i = h_{ij;j} = 0$. 
  
Now let $h$ be a deformation tensor tangential to ${\cal M}_{-1}$ at a 
hyperbolic metric $G$. Then $h$ satisfies the Lichnerowicz equation ${\cal L}_G h =0$;
\[
-(\triangle_G + K) \tr_G h + \delta_G \delta_G h = 0
\]
In addition, we require $h$ to be perpendicular to the diffeormorphism action,
which implies $\delta_G h = 0$, which in turn says that $h$ satisfies
$-(\triangle_G + K) \tr_G h =0$. When $K=0,-1$ which are the cases we are interested in,  the linear partial differential equation 
\[
-(\triangle_G + K) \tr_G h =0
\]
has only the trivial solution on the closed surface, forcing an additional condition $\tr_G h =0$. 

Therefore, we have so far characterized the conditions that a tangential vector
to the \teich space $\T_g = {\cal M}_K/{\rm Diff}_0 \Sigma$ needs to satisfy; namely the {\it trace-free} condition\index{trace-free tensors}
\[
\tr_G h = 0  
\]
which is $h_{ii} =0 \mbox{ in a normal coordinate system}$, and  
the {\it divergence-free condition}\index{divergence-free tensors}, also called the {\it transverse} condition\index{transverse tensors}
\[
\delta_G h = 0.
\]    
The so-called TT-tensors (for trace-free transverse)  appear in the study of minimal surfaces where they are the second fundamental forms of minimally embedded surfaces (see \cite{Os} for details), as well as in the study of  the Einstein equation where the tensors are a part of Cauchy initial values for the evolution problem associated to the so-called Einstein constraint equations (see \cite{FM} for references).   

We can now define the \weil metric on \teich space.

\begin{definition}[\weil metric \cite{FT1}]
The $L^2$-pairing of $T_G {\cal M}$ restricted to the trace-free, divergence-free tensors is called \weil metric on the \teich space $\T = {\cal M}_{K}/{\rm Diff}_0 \Sigma$. 
\end{definition}

As a $2\times 2$ matrix, the tangential tensor $h \in T_G{\T}$ can be expressed as 
$$ \left(
  \begin{array}{ c c }
     h_{11} & h_{12} \\
     h_{12} & -h_{11}
  \end{array} \right)
$$
with respect to a geodesic normal coordinate system centered at a point $p$ in $\Sigma$.  The integrand of the \weil pairing evaluated at $P$ becomes $2(h_{11}^2 + h_{12}^2)$. Then
the divergence-free condition is equivalent to the Cauchy-Riemann equation 
for $(h_{11} - i h_{12})(z)$ at the origin. We next look into this situation more closely.

\subsection{\weil metric and \weil cometric}
First from the discussion in modeling the \teich space as a homogeneous space 
of ${\rm QS}({\Gamma})$ for the Fuchsian group $\Gamma$,   without loss of generality, 
by using a M\"{o}bius 
transformation we may assume any given point $p$ to be the origin $O$ of
the Poincar\'{e} disc.   
Let $z=x + iy$ be the standard Euclidean coordinate system at the origin.   Note that this coordinate
system matches with the 
geodesic normal coordinate system at $O(=p)$, namely $G = \lambda(z)( dx^2 + dy^2 )$ with 
$\lambda(O) = 1$ and $\partial \lambda |_O = 0$, as the first derivatives of  $4/(1-|z|^2)^2$ at $z=0$ 
all vanish, which in turn makes all the Christoffel symbols vanish.
Then the function $(h_{11} 
- i h_{12})(z)$, where these indices denote the isothermal coordinates $x$ and $y$,
is holomorphic in $z$ at the origin.  

We recall that the cotangent space of \teich space $T^*_{[G]} \T
$ at a conformal structure $[G]$ has been identified with the space ${\rm QD}
(\Sigma)$ of holomorphic quadratic differentials\index{holomorphic quadratic differential} on the Riemann surface $
(\Sigma, [G])$. Thus the correspondence between the tangent vectors and the cotangent vectors is
\[
  h_{11} \, dx \otimes dx + h_{12}\, dx \otimes dy + h_{12} \, dy \otimes dx + (-h_{11}) \, dy \otimes dy \longleftrightarrow (h_{11} - i h_{12})(z) dz^2,
\]
the former with respect to a geodesic normal coordinate chart, and the 
latter with an isothermal coordinate chart. 
 The \weil cometric defined for the elements of   ${\rm QD}(\Sigma)$ has the form 
\[
\langle h_1^*, h_2^* \rangle_{L^2(G)} = \int_\Sigma \phi (z) \overline{\psi} (z) \frac{|dz|^2}{\rho^2(z)}
\] 
where $h^*_1 (z) = \phi(z) dz^2$ and $h^*_2 (z) = \psi (z) dz^2$ locally, and 
the hyperbolic metric $G$ with respect to the isothermal coordinate $z$ is
given as $\rho^2(z)|dz|^2$.  It is clear from the preceding argument that the two $L^2$-parings coincide, when restricted to the respective deformations of 
trace-free divergence-free tensors, and of holomorphic quadratic differentials.\\

\subsection{$L^2$-decomposition theorem of Hodge-type} 

We consider the 
$L^2$-decomposition of the tangent space $T_G {\cal M}$\index{$L^2$-decomposition for $T_G {\cal M}$}.  After having characterized 
the tangent vectors to the \teich space ${\cal M}_{-1}/{\rm Diff}_0 \Sigma$, it 
seems unnecessary to further investigate the linear structure.  However, the 
precise formulation of the $L^2$-decomposition becomes crucial in formulating 
the nonlinear strucutre, namely the curvature of the spaces. The following 
statement is an adaptation to dimension two of the theorem by Fischer-Marsden\cite{FM} 
concerning the decomposition of the deformation space of a constant scalar 
curvature metric in higher $(>2)$ dimensions.  It should be remarked that in the 1980s, Fischer and
Tromba \cite{FT1, FT2, Tr1} undertook the task of rewriting  \teich theory from a Riemannian geometric viewpoint.
In particular, they  laid out the decomposition theory of the deformation tensors
in $T_G {\cal M}_{-1}$.  Below, we develop a theory where the decomposition of the bigger linear space $T_G {\cal M} = T_G {\cal M}_{-1} \oplus (T_G {\cal M}_{-1} )^\perp$ is
addressed.
    
We have already identified the adjoint operator of the divergence operator $\delta_G$
with the Lie derivative of $G$ up to a constant; 
\[
\langle h, L_X G \rangle_{L^2(G)} = - 2 \langle \delta_G h, X \rangle_{L^2(G)}
\]    
which in turn can be stated as
\[
\delta^*_G: X \mapsto - \frac12 L_X G 
\] 
for $X \in {\mathfrak X}(\Sigma)$, the space of smooth vector fields on $\Sigma$.

We can also write down the adjoint 
operator of the Lichnerowicz operator ${\cal L}_G$ by noting the following:
\begin{eqnarray*}
\langle {\cal L}_G^* f, h \rangle_{L^2(G)} & = &  \langle f, {\cal L}_G h \rangle_{L^2(G)} \\
 & = & \int_\Sigma f(x) \big[(-\triangle_G - K)\tr_G h + \delta_G \delta_G h  \big](x) \,\, d \mu_G (x)\\
 & = &  \int_\Sigma \langle  \{ (-\triangle_G - K)f \}G + {\rm Hess}_G f, h \rangle_{G(x)} \,\, d \mu_G(x). 
\end{eqnarray*}
Hence 
\[
{\cal L}_G^* f =  (-\triangle_Gf - Kf)  G + {\rm Hess}_G f.
\]
    
For the following decomposition theorem \cite{Y1}, we restrict ourselves to the case $K \equiv -1$, i.e. when the surfaces are uniformized by hyperbolic metrics.     
\\

\begin{theorem} Suppose that $G$ is a hyperbolic metric on $\Sigma$ and that $h$ is a smooth symmetric $(0,2)$-tensor defined over $\Sigma$.  Then there is a unique $L^2$-orthogonal decomposition of $h$ as a tangent vector in $T_G {\cal M}$,
\[
h = P_G(h) + L_X G + {\cal L}^* f, 
\] 
where $P_G(h)$ is the projection of $h$ onto $T_G \T$, $L_X G$ is a
Lie derivative and ${\cal L}_G^* f$ is a tensor perpendicular to ${\cal M}_{-1}$.
Here the vector field $X$ solves the following equation uniquely
\[
\delta_G \delta_G^* X = - \frac12 \delta_G h
\]
and is smooth, 
the function $f$ solves the following equation uniquely
\[
{\cal L}_G {\cal L}_G^* f = {\cal L}_G h
\]
and is smooth.
Consequently $P_G(h)$ is uniquely determined to be a smooth tensor given by
\[
P_G (h) = h - L_X G - {\cal L}_G.
\]
Each of the three terms belongs to each of the mutually $L^2$-orthogonal 
components
\[
T_G {\cal M} = T_G {\T} \oplus_{L^2(G)} T_G {\rm Diff}_0 \Sigma \oplus_{L^2(G)} (T_G {\cal M}_{-1})^\perp. 
\]    

\end{theorem} 

We remark that this decomposition can be called of Hodge type for it identifies the tangential directions
to  \teich space with the intersection of the kernel of the differential operator $\delta_G$ and the kernel 
of ${\cal L}_G$; for both of those there are associated elliptic operators $\delta_G \delta_G^*$ and ${\cal L}_G {\cal L}_G^*$.    

\begin{proof}
The differential operators $\delta_G \delta_G^* $
and ${\cal L}_G {\cal L}_G^*$ are both elliptic, self-adjoint, and with trivial 
kernel (and hence trivial co-kernel). The triviality of the kernel of $\delta_G \delta_G^* $ follows
from first noting that 
$
0 = \langle \delta_G \delta^*_G X, X \rangle_{L^2(G)} = \langle \delta^*_G X,  \delta^*_G X \rangle_{L^2(G)} 
$
implies $\delta^*_G X = 0$ and then
from the non-existence of Killing vector fields on $\Sigma$ due to the negative curvature.  
The triviality of the kernel of ${\cal L}_G {\cal L}_G^*$  follows as $0=\langle {\cal L}_G {\cal L}^*_G f, f \rangle_{L^2(G)} =  \langle  {\cal L}^*_G f, {\cal L}^*_G f \rangle_{L^2(G)}$ implies ${\cal L}^*_G f =0$.  By taking the trace of the equation
$
{\cal L}_G^* f  =0, 
$
we obtain
$
-\triangle_G f + 2 f = 0
$
which implies $f \equiv 0$. This shows, by the standard theory of linear equations of elliptic type \cite{GT},  
that one can solve each of the two equations uniquely to specify the vector field $X=X(h)$ and the 
function $f=f(h)$, given the data $h$.

In showing the $L^2$-orthogonality, we need the following two lemmas, which trigger a series of 
orthogonal relations. 
 
\begin{lemma}
For any vector field $Y$ on $\Sigma$, we have ${\cal L}_G L_Y G =0$. 
\end{lemma} 

This follows from the simple observation that $L_Y G$ is a deformation tensor induced by a one-parameter family of isometries $\phi^*_t G$ with $\dot{\phi}_0 = Y$,  in particular preserving the 
curvature constraint, hence an element of $T_G {\cal M}_{-1}$, which is the kernel of the differential operator ${\cal L}_G$.   
 
\begin{lemma}
For any smooth function $\phi$ on $\Sigma$, we have $\delta_G {\cal L}_G^* f = 0$.   
\end{lemma} 
\begin{proof}
First choose a geodesic normal coordinate chart centered at $p$,  $\{ x^i\}$ so that $G = \delta_{ij}$ and $G_{ij;k}=0$ for all $i,j$ ad $k$ where $``;"$ stands for the covariant derivative.  Then

\begin{eqnarray*}
\delta_G {\cal L}_G^* f  & = & \delta_G \{ (- \triangle_G f + f) G + {\rm Hess}_G f \}  \\
 & = & - \{\triangle_G f+ f\}_j \delta_{ij} + f_{ij;j} \\
 & = &  - \{\triangle_G f+ f\}_j \delta_{ij} + f_{jj;i} + R_{ij} f_j \\
 & = & 0 
\end{eqnarray*}
where the Ricci identity is used to interchange the order of the covariant derivatives for the second
equality, and $R_{ij} = - \delta_{ij}$ on the hyperbolic surface $\Sigma$.
\end{proof}

We remark that an immediate consequence of the second lemma is that tensors of  type $L_Y G$  and 
type ${\cal L}_G^* \phi$ are mutually $L^2$-perpendicular for an arbitrary vector field $Y$ and an arbitrary function $\phi$, due to the
equality $\langle \delta_G {\cal L}_G^* \phi , -Y \rangle_{L^2(G)} = \langle {\cal L}_G^* \phi,  L_Y G \rangle_{L^2(G)}.$

Hence we get the first orthogonality:
\[
\langle L_X G, {\cal L}^*_G f \rangle_{L^2(G)} = 0.
\]
By projecting $h$ to $T_G \T$ and to $(T_G {\cal M}_{-1})^\perp$ respectively, we have
\begin{eqnarray*}
\langle P_G(h),  {\cal L}^*_G f \rangle_{L^2(G)} & = &  \langle h - L_X G - {\cal L}^*_G f,  {\cal L}^*_G f \rangle_{L^2(G)}  \\
 & = & \langle {\cal L}_G h -   {\cal L}_G L_X G - {\cal L}_G {\cal L}^*_G f,   f \rangle_{L^2(G)} \\
 & = & \langle {\cal L}_G h - {\cal L}_G {\cal L}^*_G f,  f \rangle_{L^2(G)} \\
 & = & 0
\end{eqnarray*}
Finally the orthogonality between $P_G(h)$ and $L_X G$ can be checked by
\begin{eqnarray*}
\langle P_G(h),  L_X G \rangle_{L^2(G)} & = &  \langle h - L_X G - {\cal L}^*_G f,  L_X G \rangle_{L^2(G)}  \\
 & = & \langle \delta_G h -   \delta_G  L_X G - \delta_G {\cal L}^*_G f,  -X \rangle_{L^2(G)} \\
 & = & \langle \delta_G h +  2 \delta_G  \delta_G^* X,  -X \rangle_{L^2(G)} \\
 & = & 0
\end{eqnarray*}
We have used above the fact  that $f$ and $X$ solve the elliptic system
\[
{\cal L}_G {\cal L}_G^* f = {\cal L}_G h, \,\,\, 
\delta_G \delta_G^* X = - \frac12  \delta_G h
\] uniquely.
\end{proof}

\subsection{$L^2$-decomposition theorem for the Universal \teich space\index{$L^2$-decomposition Theorem for the Universal \teich Space}.}   
\subsubsection{$L^2$-decomposition theorem}
In an attempt to introduce a Riemannian structure on the universal \teich space,  in particular the
\weil metric, we look at a subspace of the tangent space at the identity in ${\cal UT}$ consisting of
$L^2$-integrable tensors.  That subspace is a Hilbert space, and the quadratic form is the \weil pairing.
We generalize the $L^2$-decomposition theorem in the previous section in this context as follows \cite{Y1}.

\begin{theorem}
Suppose that $G_0$ is the standard hyperbolic metric on the unit disc ${\bf D}$, namely ${\mathbb H}^2 = ({\bf D}, G_0)$ and $h$ is an $L^2({\mathbb H}^2)$-integrable symmetric $(0,2)$-tensor defined over ${\mathbb H}^2$.   Then there is a unique $L^2$-orthogonal decomposition of $h$ as a tangent vector belonging to $T_{G_0} {\cal UT}$ as follows,
\[
h = P_{G_0}(h) + L_X G_0 + {\cal L}^*_{G_0} f
\] 
where $L_X G_0$ is a 
Lie derivative where  $X$ is  a vector field of finite $L^2({\mathbb H}^2)$-norm satisfying 
\[
\delta_{G_0} \delta^*_{G_0} X = - \frac12 \delta_{G_0} h,
\] 
where  ${\cal L}^*_{G_0} f$ is a symmetric $(0,2)$-tensor 
with the function $f$ satisfying the equation
\[
{\cal L}_{G_0} {\cal L}_{G_0}^* f = {\cal L}_{G_0} f
\]
and  where $P_{G_0}(h)$ is the projection of $h$ onto the universal \teich space
specified as $P_{G_0}(h)= h - L_X G_0 - {\cal L}_{G_0}^* h$. 
\end{theorem}

\begin{proof}
By a density argument, we can approximate $h$ by a sequence $\{h_i\}$ of compactly supported smooth 
symmetric $(0,2)$-tensors on ${\mathbb H}^2$, such that $\lim_{i \rightarrow \infty} \| h  - h_i\| = 0$. 
Hence we first consider the case where the tensor $h$ is compactly supported.  We treat the general case
at the end of the proof.

We first need to replace all the integration by parts argument in the $L^2$-decomposition 
theorem for the closed surface case by a strictly coercive property of the two elliptic differential operators
$\delta_{G_0} \delta^*_{G_0}$ and ${\cal L}_{G_0} {\cal L}_{G_0}^*$.  First we establish the statement 
for the compactly supported cases.        

\begin{lemma}
The differential operators $\delta_{G_0} \delta^*_{G_0}$ and ${\cal L}_{G_0} {\cal L}_{G_0}^*$
defined on ${\mathfrak X}_0^\infty ({\mathbb H}^2) \cap H^1 ({\mathbb H}^2)$ and $C_0^\infty ({\mathbb H}^2) \cap H^2({\mathbb H}^2)$ respectively satisfy inequalities
\[
\langle -\delta_{G_0} \delta^*_{G_0} X, X \rangle_{L^2} \geq C \|X\|_{H^1}^2 
\]
\[
\langle {\cal L}_{G_0} {\cal L}_{G_0}^* f, f \rangle_{\L^2} \geq C' \| f\|^2_{H^2} 
\] 
for some constants $C, C' >0$.
\end{lemma}

The proof of the lemma follows from integrations by parts inside the integral of the $L^2$-pairing, which are allowed for the functions and the tensors are compactly supported.  

This lemma, together with the fact that the two differential operators are self-adjoint and elliptic and the standard argument from the linear PDE theory \cite{GT} give us that there are unique solutions 
$X$ and $f$ to the equations 
$
\delta_{G_0} \delta^*_{G_0} X = - \delta_{G_0} h
$
and 
$
{\cal L}_{G_0} {\cal L}_{G_0}^* f = {\cal L}_{G_0} f
$
for a compactly supported data $h$.  This says that we have an $L^2$-decomposition of compactly supported tensors $h$'s.  

Now coming back to the general case where $h$ is $L^2$-integrable, let $\{h_i\}$ be an approximating 
sequence, each compactly supported, convergent to $h$ in $L^2$-norm.  For each $i$, we can solve the pair of equations $
\delta_{G_0} \delta^*_{G_0} X_i = - \delta_{G_0} h_i
$
and 
$
{\cal L}_{G_0} {\cal L}_{G_0}^* f_i = {\cal L}_{G_0} f_i
$
for $X_i$ and $f_i$.  We need to show that the sequences $\{X_i\}$ and $\{f_i\}$ are 
both convergent in the respective spaces, and that they solve the equation
$
\delta_{G_0} \delta^*_{G_0} X = - \delta_{G_0} h
$
and 
$
{\cal L}_{G_0} {\cal L}_{G_0}^* f = {\cal L}_{G_0} f
$.     
To see this, first note that 
\begin{align*}
\|L_{X_i - X_j} G_0\|^2 & = & \int_{\h} \langle L_{X_i-X_j} G_0, L_{X_i - X_j} G_0 \rangle_{G_0} d 
\mu_{G_0} \\
 & = & \int_{\h} \langle h_i-h_j, L_{X_i - X_j} G_0 \rangle_{G_0} d 
\mu_{G_0}  \\
 & \leq & \| h_1-h_j \|_{L^2} \| L_{X_i - X_j} G_0 \|_{L^2}.
\end{align*}
where the second equality is due to the $L^2$-decomposition for compactly supported tensors.
This, together with the coercivity of $\delta_{G_0} \delta_{G_0}^*$  which says that 
$C \|X_i - X_j \|_{L^2} \leq \|L_{X_i} G_0 - L_{X_j} G_0 \|_{L^2}$  for some $C>0$, gives  that
$C \|X_i - X_j \|_{L^2} \leq \|h_i - h_j \|_{L^2}$.  This shows that the sequence $\{X_i\}$ is Cauchy in
the space of $H^1$-integrable vector fields on $\h$, and hence convergent to some $X$.   
The elliptic regularity says that for $h$ smooth, so is $X$.  

By an analogous argument, one checks that $\{f_i\}$ is Cauchy in the Sobolev space $H^2$ on $\h$,
and it converges to some smooth $f$ for a smooth data $h$.   

The mutual $L^2$-orthogonality of $P_{G_0} (h_i)$, $L_{X_i} G_0$ and ${\cal L}_{G_0} f_i$ for
each $i$ then induces the orthogonality of $P_{G_0} (h)$, $L_{X} G_0$ and ${\cal L}_{G_0} f$,
proving the statement of the theorem.
\end{proof}

Here we observe this $L^2$-decomposition theorem from the viewpoint of the universal \teich space.
Recall that the universal \teich space is the quotient space of the  space of  hyperbolic metrics consists of  the pulled-back metric
of the standard Poincar\'{e} metric $G_0$ by  all the quasi-conformal (q.c.) diffeomorphisms of the  ${\bf 
D}$ where $G_1$ and $G_2$ are defined to be equivalent when
$G_2 = \phi^* G_1$ for some quasi-conformal diffeomorphism $\phi$ fixing the geometric boundary $
\partial {\bf D} = S^1$. 

By linearizing this picture at $G_0$, namely considering a one-parameter family of quasi-conformal 
diffeomorphisms $\phi_t$ with $\phi_0 = {\rm Id}_{\bf D}$,  and differentiating the pulled-back metrics
$\phi_t^* G_0$ at time $t=0$, we can identify the tangent space of the universal \teich space ${\cal UT}$ at $G_0$ as
\[
T_{G_0} {\cal M}_{-1} = \{ L_Z G_0  \, | \,  Z  \\ 
\mbox{ vector fields generating  
q.c.-diffeomorphisms on ${\bf D}$}  \}.  
\]
Now let $h$ be an  $L^2$-integrable deformation tensor of $G_0$ tangential to ${\cal M}_{-1}$, namely 
assume there is no third component of the type ${\cal L}_{G_0}^* f$ in the $L^2$-decomposition of $h$;  
then $h$ can be written  as a Lie derivative $L_Z G_0$ of $G_0$ for some vector field $X$.
As the $L^2$-decomposition gives $h= P_{G_0} (h) + L_X G_0$ for some vector field $X = X(h)$,
the tangential component $P_{G_0}(h)$ of $h$ to ${\cal UT}$ is of the form $L_{Z-X} G_0$.

Define $Z_{\overline{z}} := \mu$, and $\nu := X_{\overline{z}}$.  Recall from Section 3.5 that these equalities are 
the linearizations of the Beltrami equations at the identity map. Recall the Beltrami equation is of the 
form $w_{\overline{z}} = \mu w_z$.  Now take the Beltrami coefficients to be $\e \mu_0$ for some 
fixed $\mu_0$ and $|\e|$ sufficiently small so that $\e \mu_0$ remains in the unit ball $L^\infty ({\bf D})_1$
in the complex Banach space $L^\infty ({\bf D})$.  Differentiate the equation $w(\e)_{\overline{z}} = \e \mu_0 w(\e)_z$
in $\e$ and evaluate at $\e = 0$ to obtain
\[
\dot{w}_{\overline{z}} (0)= \mu_0 
\]
as $w(0) = z$. Denote the vector field $\dot{w}(0)$ by $V(\mu_0)$. In general the equation $V_{\overline{z}}= \mu$ can be solved uniquely on the disc \cite{Ah1, NV, Na} using an integral kernel on the disc.  

The resulting vector field $V(h):= Z-X$ defined on the disc  is a particular type  such that
$V_{\overline{z}} =: P[\mu]$ is a so-called harmonic Beltrami differential.  Note that $P[\mu] = \mu - \nu$.
For  detailed expositions on harmonic Beltrami differentials, see \cite{Ah1, W4, NV}.

Let $\rho(z)$ be the area density 
of the hyperbolic metric $G_0$ with respect to the Euclidean area density of the unit disc.
Recall that $Z = Z(h)$ is an $L^2({\h})$-integrable vector field. On the Poincar\'{e} disc model of $\h$,
as the area density $\rho(z) =  \frac{4}{(1-|z|^2)^2}$ blows up as $|z| \rightarrow 1$, 
$Z(z)$ has to decay as $z$ approaches to the geometric boundary $\{ |z|=1 \}.$  Namely the 
one-parameter families of quasi-conformal diffeomorphisms $Z$ generates belong to ${\rm QC}_0({\bf D})$.  

The statement of the $L^2$-decomposition theorem corresponds to the so-called Ahlfors's integral
projection operator. Let $B$ be the space $L^\infty(\Gamma_0)_1$ of Beltrami differentials with 
$\Gamma_0$ the trivial group $\{ {\rm Id} \}$.  Let ${\cal B} \subset B$ be the space of harmonic Beltrami 
differentials, where a Beltrami differential $\mu$ is harmonic if $\rho(z) \overline{\mu}$ is holomorphic,
or $\mu = \rho^{-1}  \overline{\phi}$ for some holomorphic quadratic differential $\phi$ on ${\bf D}$. 
This correspondence can be better understood by looking at
\[
\rho \, dz d \overline{z}  \overline{ \Big[\mu \frac{d\overline{z}}{dz}\Big]}  = \rho \overline{\mu} \,d z^2.
\]
  
Ahlfors  \cite{Ah1} introduced a bounded linear operator $P: B \rightarrow {\cal B}$ 
given by 
\[
P[\mu] = \frac{-3(z - \overline{z})^2}{\pi} \int_{\h} \frac{\mu(\eta)}{(\eta - \overline{z})^4} d \sigma (\eta)   
\]     
where $d \sigma$ is the Euclidean area element, and $z$ and $\eta$ are the Euclidean coordinates for 
the upper half space, which is a model of  the hyperbolic plane $\h$.   This map is indeed a projection for
one checks that $P[\mu] = \mu$ when $\mu \in {\cal B}$.   The kernel of the projection map 
is denoted by $N$, and it is known that the space $B/N$ is identified with the tangent space 
$T_{G_0} {\cal UT}$.  Now recall our $L^2$-decomposition theorem says that an $L^2$-integrable 
tensor $h$ tangential to the space of hyperbolic 
metrics ${\cal M}_{-1}$ can be decomposed as
\[
h  =  P_{G_0}(h) + L_X G_0 
\]
while 
\[
h = L_Z G_0
\]
for some vector field $Z$.  Each of the three Lie derivatives $L_Z G_0, L_X G_0$ and $L_{Z-X} G_0$ 
is identified uniquely to Beltrami differentials $\mu$, $\nu$ and $\mu - \nu$ respectively, via the $\overline{\partial}$ -equations 
\[
Z_{\overline{z}} = \mu, \,\,\,  X_{\overline{z}}= \nu  \mbox{ and } (Z-X)_{\overline{z}} = \mu -\nu. 
\]
 Now the correspondence 
between the two representations of the tangent space is given by $P_{G_0} [\mu] = \mu - \nu$
and $\nu$ is an element of the kernel $N$ of the projection operator $P_{G_0}: B \rightarrow {\cal B}$. 

\subsubsection{\weil complex structure\index{\weil complex structure}} 
We explain here that the paper of Nag-Verjovsky \cite{NV} identifies vector fields on $S^1$ with  tangent vectors of the universal \teich space at the identity, which is a natural thing to do as the universal \teich space is defined as 
\[
{\cal UT} = {\rm QS}(S^1) / {\rm SL}(2, {\mathbb R}),  
\] 
the space of quasi-symmetric self-maps of $S^1$ fixing three points (say, $(1,0), (0,1)$ and $(-1, 0)$ for 
example) on the circle. Each tangent vector $\Theta$ is obtained by linearizing the solution to the Beltrami
equation  near the identity as shown above restricted the resulting vector field defined on the 
unit disc  ${\bf D}$ onto the unit circle $S^1 = \partial {\bf D}$;
namely $\Theta = w_{\e \mu_0} (z)$ with $|z|=1$ where 
\[
w_{\e \mu_0} (z) = z + \e \dot{w}[\mu_0](z) + o(\e) \,\,\, \e \rightarrow 0.
\]
When the universal \teich space ${\cal UT}$ is regarded as the space of Beltrami differentials
$L^\infty({\bf D})_1/ \sim$, there is a natural complex structure $J : \mu \mapsto i \mu$ on the tangent space of  ${\cal UT}$ at the identity.
This induces a complex structure $\tilde{J}$ on the other representation of the universal \teich space as
one can linearize the one-parameter family of quasi-conformal diffeomorphisms $w_{\e i \mu_0}$ at
the origin $\e =0$ and restrict the resulting vector field to the unit circle.    Nag-Verjovsky \cite{NV} showed that
this $\tilde{J}$ is the Hilbert transform\index{Hilbert transofrm}, a statement attributed to S.Kerckhoff: \\

\noindent {\bf Theorem}  Using $\theta$ as  coordinate on $S^1$, and $z = e^{i \theta}$, define $\Theta = u
(\theta) (\del/\del \theta)$, where $\dot{w}[\mu_0](e^{i \theta}) = i z u(\theta)$, namely $u(\theta)$  is the 
magnitude of the vector field $\Theta$ at the point $z = e^{i \theta}$.  Then $\tilde{J} \Theta = u^* (\theta)$
where   $\dot{w}[i \mu_0](e^{i \theta}) = i z u^*(\theta)$, where $u^*(\theta)$ is given by
\[
u^* (z) = {\rm Im}(D(z)) + (cz + \overline{c} \overline{z} + b) 
\]
on $\{ |z| = 1\}$ for a certain $b \in {\mathbb R}$ and $c \in {\mathbb C}$, and $D(z)$ is an element of the disc algebra 
$A({\bf D})$ (namely functions holomorphic in ${\bf D}$ and continuous on $\overline{{\bf D}}$ such that
${\rm Re} D(z) = u(z)$ on $z \in S^1$.)  

Note that the statement is for the universal \teich space, but one can restrict to the \teich space 
of a Fuchsian group, as the $\Gamma$-equivariance can be incorporated into the proof of this theorem. 
   
\subsubsection{${\rm Diff} S^1/{\rm SL}(2, {\mathbb R})$ as a Hilbert manifold\index{Hilbert manifold}}
  
One of the merits in looking at the $L^2$-integrable deformation tensors of the Poincar\'{e} metric $G_0$ is that it provides a Hilbert space which acts as the tangent space equipped with the \weil pairing.
In the paper \cite{NV}, the Lie algebra of ${\rm Diff} S^1$ is identified 
as the algebra of $C^\infty$-smooth vector fields on $S^1$.  The complexification of the Lie Algebra is the Virasoro algebra generated by 
\[
L_n = e^{i n \theta} \frac{\del}{\del \theta} = i z^{n+1} \frac{\del}{\del z},  \,\,\, n \in {\bf Z}, z = e^{i \theta}.
\]
A tangent vector to the  homogeneous space ${\rm Diff} S^1 / {\rm SL}(2, {\mathbb R})$ at the identity $[{\rm Id}]$
is of the form 
\[
\Theta = \sum_{m \neq -1, 0, 1} v_m L_m, \,\,\, \overline{v_m} = v_{-m}.
\]
Note that the omission of $m = -1, 0, 1$ is due to the fact that $L_{-1}, L_0$ and $L_1$ span the subspace 
that is the complexification of ${\it sl} (2, {\mathbb R})$ within the compliexified Lie algebra of ${\rm Diff} S^1$.
As the diffeormorphisms are $C^\infty$-smooth, each vector field $\Theta = u(\theta) \del / \del \theta$  can 
be identified with a  $2 \pi$-periodic $C^\infty$ real-valued function $u(\theta)$.    
 
There is a natural complex structure $\tilde{J}$ at the identity, which is  conjugation;
\[
\tilde{J} \Theta = \sum_{m \neq -1, 0, 1} - i \, {\rm sgn}(m) v_m L_m \,\,\,
\]
Now $\tilde{J}$ is an almost complex structure by definition, but it is also known that it is indeed integrable, and that the 
right multiplications/translations by the elements of ${\rm Diff} S^1$ are biholomorphic automorphisms of 
the homogeneous space ${\rm Diff} S^1/{\rm SL}(2, {\mathbb R})$ \cite{NV}.   
 
As each smooth diffeomorphism of $S^1$ extends to a smooth diffeomorphism of the closed disc $S^1 
\cup D$, which in turn is  quasi-comformal, ${\rm Diff} S^1$ is a subset of the space ${\rm QS} (S^1)$ of 
quasi-symmetric maps.   Thus one can think of the homogeneous space 
${\rm Diff} S^1 / {\rm SL}(2, {\mathbb R})$ as a subset of the universal \teich space ${\cal UT} = {\rm QS}(S^1) / {\rm SL}(2, {\mathbb R})$.  

We introduce the following theorem by Nag-Verjovsky \cite{NV}.\\

\noindent {\bf Theorem} {\it The natural inclusion ${\rm Diff} S^1 / {\rm SL}(2, {\mathbb R}) 
\hookrightarrow {\cal UT}$ is holomorphic. }
\\

The proof of this statement follows, using the homogeneous structures of both ${\rm Diff} S^1 / {\rm SL}(2, {\mathbb R})$ and ${\cal UT}$,  from checking that the complex structure $\tilde{J}$ and $J$ coincide at the identity
of ${\cal UT}$.  We omit the details and refer the reader to the paper \cite{NV}.
In short, the conjugation $\tilde{J}$ of $\Theta = u(\theta) \del / \del \theta$ can be shown to  coincide with the Hilbert transform 
$u^*(\theta)$ of $u(\theta)$ so that by setting $D(e^{i \theta} ) = u(\theta) + i u^*(\theta)$,   $D$ is identified 
with an element of the disc algebra $A({\bf D})$, as appeared in the above theorem. 

In this situation Nag and Verjovsky show that the only possible homogeneous K\"{a}hler form
$\omega$ on ${\rm Diff} S^1/{\rm SL} (2, {\mathbb R})$ given at the identity is 
\[
\omega (L_m, L_n) = a (m^3-m) \delta_{m, -n}, \,\,\,  m,n \in {\bf Z} \backslash \{\pm 1, 0\}
\]   
for $a$ a purely imaginary number.  This form had been previously known as the Kirillov-Kostant symplectic form \cite{TT}
on  ${\rm Diff} S^1/{\rm SL} (2, {\mathbb R})$.  

Having the complex structure $\tilde{J}$ and the K\"{a}hler form $\omega$  at hand, there is a natural K\"{a}hler metric on ${\rm Diff} S^1/{\rm SL} (2, {\mathbb R})$, specified as $g(v, w) = - \omega(v, \tilde{J} w)$. 
When the tangent vectors $\Theta_1$ and $\Theta_2$ at the identity of ${\rm Diff} S^1/{\rm SL} (2, {\mathbb R})$ are written as Fourier series $\sum v_m L_m$ and $\sum w_m L_m$ respectively, the metric has the form
\[
g(\Theta_1, \Theta_2) = -2ia {\rm Re} \Big[ \sum_{m=2}^\infty v_m \overline{w_m} (m^3 - m) \Big].
\]     
This series converges absolutely when the vector fields $\Theta_i$ are elements of the Sobolev space
$H^{3/2}(S^1)$, when these vector fields $\Theta_i$ are identified with  $u_i(\theta) \del/\del \theta$. To keep the metric positive definite, we need $a = ib$ for $b >0$.  

Nag and Verjovsky  proceed to show that this K\"{a}hler metric is indeed the \weil metric defined on the $L^2(G_0)$-integrable tensors in $T_{G_0} {\cal UT}$.  The proof is by finding an explicit correspondence between 
an $L^2(G_0)$-integrable Beltrami differential $\mu$ and the vector field $\dot{w}[\mu](e^{i \theta})=: V(e^{i \theta})$ on $S^1$, by solving the $\overline{\del}$-equation $V_{\overline{z}} = \mu$ using the integral kernel.   

This series of results can be summarized by the fact  that  the Hilbert manifold structure imposed on the homogeneous space ${\rm Diff} S^1/{\rm SL}(2, {\mathbb R})$ is affiliated to the space of $H^{3/2}$-smooth vector fields on $S^1$. In exponentiating those vector fields with respect to the \weil pairing, 
we expect to understand the global structure of the homogeneous space.  We will come back to this issue
in Section 4.6.

\subsection{\weil geodesic equation}

Having the linear structure of the tangent space $T_G {\cal M}$ at each hyperbolic metric $G$, and the covariant derivative $D$ on ${\cal M}$, we proceed to write down the \weil geodesic equation\index{\weil Geodesic Equation}.  

First note that the space ${\cal M}$ of smooth metrics on the surface $\Sigma$ contains the space ${\cal M}_{-1}$ of 
hyperbolic metrics as a smooth submanifold, as the function $K$ which assign each metric its sectional curvature:
\[
K : {\cal M} \rightarrow C^\infty (\Sigma)
\]
has the constant function $-1$ as a regular value.
This follows from a standard argument (see \cite{FM, Be}), namely the linearized operator, which is
the Lichnerowicz operator ${\cal L}_G$
\[
DK(G): T_G {\cal M} \rightarrow C^\infty (\Sigma)
\] 
is surjective, as we have already seen in the proof of $L^2$-decomposition theorems.

A consequence of the $L^2$-decomposition is that we can see that the quotient map
${\cal Q} :{\cal M}_{-1} \rightarrow {\cal M}_{-1} / {\rm Diff}_0 \Sigma$ is a Riemannian 
submersion, a point of view initiated by Earle-Eells \cite{EE} in the 1960's, as the linearization of the map 
\[
D {\cal Q}(G): T_G {\cal M}_{-1} \rightarrow T_G  \T
\] 
sends $P_G(h) + L_X G$ to $P_G (h)$ where $L_X G$ is perpendicular to the tangent space
$T_G \T$.  Namely the tangent space $T_G {\cal M}_{-1}$ is split into the horizontal space $T_G \T$
and the vertical space $T_G {\rm Diff}_0 \Sigma$, and the latter space is the kernel of the linear map
$D {\cal Q}(G)$.  

A standard result on Riemannain submersions then tells us that given a \weil geodesic $\sigma: [0, T] \rightarrow \T$ and a hyperbolic metric $G_0$ with ${\cal Q}(G_0) = \sigma (0)$,  there exists a unique 
path $G_t$ in ${\cal M}_{-1}$, itself a geodesic in ${\cal M}_{-1}$ with ${\cal Q}(G_t) = \sigma(t)$ for $t \in [0, T]$, and its $L^2$-metric length is equal to the \weil length of $\sigma$.  The path $G_t$ is called
the horizontal lift of $\sigma$ with  initial point $G_0$.

In what follows, we will identify each \weil geodesic with its horizontal lift with a suitable initial metric.  
 
Let $\Pi_G$ be the projection map 
\[
\Pi_G : T_G {\cal M} \rightarrow (T_G {\cal M}_{-1})^{\perp}
\]
defined by $\Pi_G (h) = {\cal L}_G f$ where $f$ satisfies ${\cal L}_G {\cal L}_G^* f = {\cal L}_G h $. 

For (a horizontal lift of) a \weil geodesic $\{G_t\} \subset {\cal M}_{-1} \subset {\cal M}$, the tangent vector $\dot{G}_t := \frac{d}{d\tau} G_\tau |_{\tau = t}$ is an element of  $T_{G_t} {\cal M}_{-1}$ at each time $t$, namely 
\[
\Pi_{G_t} (\dot{G}_t) = 0 .
\] 
Furthermore,  as $G_t$ is a horizontal lift of a \weil geodesic, we have 
$\dot{G}_t = P_{G_t} (\dot{G}_t)$.  

Since $G_t$ is a geodesic in ${\cal M}_{-1}$ where the space ${\cal M}_{-1}$ is a Riemannian
manifold with its metric being the $L^2$-pairing, the geodesic curvature vector of the curve $G_t$ vanishes.

Recall that we have specified the Levi-Civita connection $D$ of the $L^2$-pairing defined on the tangent 
bundle $T {\cal M}$ in the above section.  The connection $D$ then induces the Levi-Civita 
connection $\nabla$ on its submanifold ${\cal M}_{-1}$ by the Gauss formula
\[
\nabla_X Y = (D_X Y)^{T{\cal M}_{-1}},
\]
where the right hand side is the tangential component of $D_X Y$ to $T_G {\cal M}_{-1}$. Hence
the fact that $G_t$ is a geodesic in ${\cal M}_{-1}$, is equivalent to 
\[
\nabla_{\dot{G}_t} \dot{G}_t = 0
\]
which in turn is equivalent to
\[
D_{\dot{G}_t} \dot{G}_t = \Pi_{G_t}  (D_{\dot{G}_t} \dot{G}_t )
\]
as $(\Pi_{G_t}  (D_{\dot{G}_t} \dot{G}_t ))^{T{\cal M}_{-1}}=0$ by definition.
This last equation is the \weil geodesic equation in the context of the $L^2$-geometry of the space of metrics $\cal M$. 

\begin{theorem}
Given a horizontal lift $G_t$ of a \weil geodesic, we have the following expression for the second $t$-derivative of $G_t$ in $T_G {\cal M}$;
\[
\frac{d^2}{dt^2} G_t \Big|_{t=0} = \Big( \frac14 \| \dot{G}_0 \|^2 + \alpha \Big) G_0 + L_Z G_0
\]
where $\alpha = - \frac12 (\triangle_{G_0} -2 )^{-1} \|\dot{G}_0\|^2$ which is nonnegative, and $Z$ is a vector field on $\Sigma$.
\end{theorem}

\noindent {\bf Remark} Tromba \cite{Tr2} has a similar calculation (Theorem 2.1) to obtain an expression of the
second derivative of a horizontal lift of a \weil geodesic, 
\[
\frac{d^2}{dt^2} G_t \Big|_{t=0} = \frac12 \| \dot{G}_0 \|^2 G_0 + L_W G_0
\]
with $\tr_{G_0} (L_W G_0) = 0$ which differs from the one above.  The calculation is seemingly based on the assumption that the space ${\cal M}$ with respect to the $L^2$-metric is a linear space so that the geodesic curvature vector of an arc-length parameterized path $G_t$ is $\ddot{G}_t$.  Thus the \weil geodesic equation (Eqn.(2.1))  in \cite{Tr2} is $\ddot{G}_t = \Pi_{G_t} (\ddot{G}_t) $ instead of our $D_{\dot{G}_t} \dot{G}_t = \Pi_{G_t}  (D_{\dot{G}_t} \dot{G}_t )$.  We also note here that it was claimed in the proof of the same theorem  (Eqn.(2.3)) that the trace-free part of a symmetric $(0,2)$-tensor which is divergence-free is again divergence-free, which does not hold  in general.  
 
\begin{proof}
Recall that the geodesic curvature vector $\kappa$ of an arc-length parameterized cu $u(t)$ has the expression
\[
\kappa (t) = [\nabla_{\dot{u}} \dot{u}](t) = \ddot{u}^\alpha(t) + \Gamma^\alpha_{\beta \gamma} \dot{u}^\beta(t) \dot{u}^\gamma(t). 
\]
In our setting, this is equivalent to 
\[
\kappa(t) = [D_{\dot{G_t}} \dot{G_t}](t) = \ddot G_t + D_{\dot{G_t}} \dot{G_t},
\]
where $D$ is the Levi-Civita connection for the $L^2$-metric defined on $\cal M$.  

For a horizontal lift $G_t$ of a \weil geodesic $\sigma(t)$, the velocity vector is tangential to the \teich space,
\[
\Pi_{G_t} (\dot{G}_t) = 0,
\]
 while the geodesic curvature vector has no tangential component to the \teich space,
\[
D_{\dot{G_t}} \dot{G_t} =  \Pi_{G_t}  (D_{\dot{G}_t} \dot{G}_t ).
\]
Hence we have the  expression 
\[
 \ddot G_0 + D_{\dot{G_0}} \dot{G_0} = \Pi_{G_0} [\ddot G_0] +  \Pi_{G_0} [D_{\dot{G_0}} \dot{G_0}] 
\]
which is reorganized as 
\[
 \ddot G_0 = \Pi_{G_0} [\ddot G_0] - [D_{\dot{G_0}} \dot{G_0}]^{T_{G_0}{\cal M}_{-1}} 
\]
as $D_{\dot{G_0}} \dot{G_0} - \Pi_{G_0} [D_{\dot{G_0}} \dot{G_0}] $ constitutes the tangential component
to ${\cal M}_{-1}$ in the $L^2$-decomposition of $T_{G_0} {\cal M}_{-1}$.   
On the other hand, differentiating the tangential condition $\Pi_{G_t} (\dot{G}_t) = 0$ in $t$ yields 
\[
\frac{d}{dt} \Pi_{G_t}  \dot{G}_t \Big|_{t=0} = \frac{d}{dt} \Pi_{G_t}  \dot{G}_0 \Big|_{t=0} + \Pi_{G_0} \ddot{G}_0 = 0.
\]
Combining these, we have an expression for $\ddot{G}_0$;
\[
\ddot{G}_0 = - \frac{d}{dt} \Pi_{G_t}  \dot{G}_0 \Big|_{t=0}   - [D_{\dot{G_0}} \dot{G_0}]^{T_{G_0}{\cal M}_{-1}}. 
\]
The term $[D_{\dot{G_0}} \dot{G_0}]^{T_{G_0}{\cal M}_{-1}}$ can be computed further by using the 
explicit expression for the Levi-Civita connection as follows. Recall the formula
\[
D_{h_1} h_2 = -\frac12 \Big[ h_1 \cdot  G^{-1} \cdot h_2  +  h_2 \cdot G^{-1} \cdot h_1 \Big] + \frac14 \Big[ (\tr_G h_1) h_2 +  (\tr_G h_2) h_1 -  \langle h_1, h_2 \rangle_{G(x)} G \Big],
\]
which becomes
\[
D_{h_1} h_2 = -\frac12 h_1 \cdot  G^{-1} \cdot h_2  -\frac12 h_2 \cdot G^{-1} \cdot h_1  - \frac14  \langle h_1, h_2 \rangle_{G(x)} G 
\]
for the  trace-free symmetric $(0,2)$ tensors $h_1$ and $h_2$.  When $G=G_0$ and $h_1=h_2=\dot{G}_0$, we have
\[
D_{\dot{G_0}} \dot{G_0} = - \dot{G}_0 \cdot G_0^{-1} \cdot \dot{G}_0 -\frac14 \|\dot{G}_0\|^2_{G_0(x)} G_0
\]
which by using the geodesic normal coordinates so that $G_0 = \delta_{ij}$ and $(\dot{G}_0)_{11}=-(\dot{G}_0)_{22}$,  the matrix multiplication gives 
\[
- \dot{G}_0 \cdot G_0^{-1} \cdot \dot{G}_0 = -\left(
  \begin{array}{ c c }
     (\dot{G}_0)_{11}^2 +  (\dot{G}_0)_{12}^2 & 0 \\
     0 & (\dot{G}_0)_{11}^2 +  (\dot{G}_0)_{12}^2
  \end{array} \right) = - \frac12 \|\dot{G}_0\|^2 G_0.
\]
We recall that the explicit expression for the  Levi-Civita connection is valid only for the locally constant symmetric $(0,2)$ tensors.  Here the tensor $\dot{G}_0$ is treated as such, as the quantities in the calculation are tensorial.  

This says, in the light of the $L^2$-decomposition theorem, that the deformation tensor $D_{\dot{G_0}} \dot{G_0}$ is purely conformal, hence point-wise (and thus $L^2$) orthogonal to the trace-free tensors, which in turn implies that the tensor has no tangential component to the \teich space.  By taking the projection of  $D_{\dot{G_0}} \dot{G_0}$ to the tangent space $T_{G_0}{\cal M}_{-1}$, the resulting tensor is along the diffeomorphism fiber, hence $[D_{\dot{G_0}} \dot{G_0}]^{T_{G_0}{\cal M}_{-1}} = L_X G_0$ for some 
smooth vector field $X$.  Hence we have so far established
\[
\ddot{G}_0 = - \frac{d}{dt} \Pi_{G_t}  \dot{G}_0 \Big|_{t=0}   - L_X G_0. 
\]
We proceed to calculate the term $ \frac{d}{dt} \Pi_{G_t}  \dot{G}_0 \Big|_{t=0} $. 

First of all as a consequence of the $L^2$-decomposition theorems, we have the following formula for the third 
component of the decomposition, which is in the orthogonal directions to the space of constant curvature 
metrics. 

\begin{proposition}
In the $L^2$-decomposition theorems, where an arbitrary smooth tensor $h$ has the following decomposition,
\[
h= P_G(h) + L_X G + {\cal L}_G^* f,
\]
where $X$ and $f$ are the unique solutions of the equation $\delta_G \delta_G^* X = -\delta_G h$ and 
${\cal L}_G {\cal L}_G^* f = {\cal L}_G h$ respectively,  we have
\[
(\triangle_G -1)f = (\triangle_G -2)^{-1} {\cal L}_G h.
\] 
\end{proposition}  

\begin{proof}
Note the following equalites
\begin{eqnarray*}
{\cal L}{\cal L}^* f & = & -( \triangle -1) \tr_G ({\cal L}^* f) + \delta_G \delta_G {\cal L}^* f \\
 & = &  -( \triangle -1) \{-2(\triangle -1)f + \triangle f \}+ \{-(\triangle -1)f \delta_{ij} + f_{ij} \}_{;ij} \\
 & = &   (\triangle -2)(\triangle -1) f  \\
 & = & {\cal L} h, 
\end{eqnarray*}
where the last inequality follows from ${\cal L}{\cal L}^* f  = {\cal L} h$ As the differential operator $\triangle -2$ is invertible, we obtain the statement.
\end{proof}

Having this statement at hand, we move on to write down
the projection operator $\Pi_G: h \mapsto {\cal L}^*_G h$  as 
\begin{eqnarray*}
\Pi_{G_t} \dot{G_0} &=& \{(\triangle_{G_t} -1)f_t \} G_t + {\rm Hess}_{G_t} f_t \\
 &=& - \{(\triangle_{G_t} -2)^{-1}{\cal L}_{G_t} \dot{G_0}\} G_t +  {\rm Hess}_{G_t} f_t 
\end{eqnarray*}
where $f_t$ is the solution of ${\cal L}_{G_t} {\cal L}_{G_t}^* f_t = {\cal L}_{G_t} \dot{G}_0$.  We have 
at $t=0$, ${\cal L}_{G_0} \dot{G}_0 = 0$ and thus $f_0 =0$.  Using these equalities, the time-derivative of $\Pi_{G_t} \dot{G_0}$ at $t=0$ can be written as
\[
\frac{d}{dt} \Pi_{G_t} \dot{G_0} = - (\triangle_{G_0} -2)^{-1}\Big(  \frac{d}{dt} {\cal L}_{G_t} \dot{G_0} \Big|_{t=0} \Big)G_0 + {\rm Hess}_{G_0}  \Big( \frac{d}{dt} f_t  \Big|_{t=0} \Big). 
\]  
Note that the Hessian term $ {\rm Hess}_{G_0}  \Big( \frac{d}{dt} f_t  \Big|_{t=0} \Big)$ is a Lie derivative
$L_{\nabla \dot{f}_0} G_0$.

We now calculate the term $  \frac{d}{dt} {\cal L}_{G_t} \dot{G_0} \Big|_{t=0}$.
\begin{eqnarray*}
  \frac{d}{dt} {\cal L}_{G_t} \dot{G_0} \Big|_{t=0} &=&   \frac{d}{dt} \Big( - (\triangle_{G_t} -1) \tr_{G_t}
 \dot{G}_0 + \delta_{G_t} \delta_{G_t} \dot{G}_0 \Big) \Big|_{t=0}  \\
 & = &  - (\triangle_{G_0} -1)  \frac{d}{dt} (\tr_{G_t}
 \dot{G}_0) \Big|_{t=0} + \delta_{G_0} \frac{d}{dt}  (\delta_{G_t} \dot{G}_0) \Big|_{t=0}  \\
 & = & -(\triangle_{G_0} -1) (- (\dot{G}_0)^{ij} (\dot{G}_0)_{ij}) +   \delta_{G_0} (- \frac34 \nabla \| \dot{G}_0\|^2) \\
 & = & (\triangle_{G_0} -1) \| \dot{G}_0\|^2 - \frac34 \triangle_{G_0} \| \dot{G}_0\|^2 \\
 & = & (\frac14 \triangle_{G_0} -1) \| \dot{G}_0 \|^2.
\end{eqnarray*}
In the third equality, the fact  $\frac{d}{dt}  (\delta_{G_t} \dot{G}_0) \Big|_{t=0}=- \frac34 \nabla \| \dot{G}_0\|^2$ was used.  This follows from the following calculation; as 
\[
(\delta_G h)_i = G^{jk} h_{ij;k} = G^{jk} (h_{ij,k}- h_{pj} \Gamma^p_{ik} - h_{ip} \Gamma^p_{jk} )
\]
where the semi-colon is used to denote the covariant derivative, we have
\begin{eqnarray*}
\frac{d}{dt}  (\delta_{G_t} \dot{G}_0) \Big|_{t=0} &=& - (\dot{G}_0)^{jk} (\dot{G}_0)_{ij,k} \\
 &  & - G_0^{jk}  (\dot{G}_0)_{pj} \frac{d}{dt} \Big( \frac12 (G_t)^{pq} \{(G_t)_{iq, k} + (G_t)_{kq, i} - (G_t)_{ik;q}\} \Big) \Big|_{t=0} \\
 & & - G_0^{jk} (\dot{G}_0)_{ip} \frac{d}{dt} \Big(\frac12 (G_t)^{pq} \{(G_t)_{jq, k} + (G_t)_{kq, j} - (G_t)_{jk;q}\} \Big)\Big|_{t=0} \\
 & = & - (\dot{G}_0)^{jk} (\dot{G}_0)_{ij, k} \\
  &  & - (G_0)^{jk} (\dot{G}_0)_{pj}  \Big(\frac12 (G_0)^{pq} \{(\dot{G}_0)_{iq, k} + (\dot{G}_0)_{kq, i} - (\dot{G}_0)_{ik;q}\} \Big)\\
 & & - (G_0)^{jk} (\dot{G}_0)_{ip}  \Big(\frac12 (G_0)^{pq} \{(\dot{G}_0)_{jq, k} + (\dot{G}_0)_{kq, j} - (\dot{G}_0)_{jk;q}\} \Big). 
\end{eqnarray*}
By using the fact that at $t=0$, $G_0$ can be expressed as $\delta_{ij}$ at each point, which in turn makes the traceless transverse deformation tensor $(\dot{G}_0)_{ij, k}$ fully symmetric in $i, j$ and $k$, we get a tensorial expression for our result;
\[
\frac{d}{dt}  (\delta_{G_t} \dot{G}_0) \Big|_{t=0}  = - \frac34 \|\dot{G}_0\|^2_{;i}.
\]  

By inserting  $\frac{d}{dt} {\cal L}_{G_t} \dot{G_0} \Big|_{t=0} =  (\frac14 \triangle_{G_0} -1) \| \dot{G}_0 \|^2$ 
we have 
\begin{eqnarray*}
\frac{d}{dt} \Pi_{G_t} \dot{G}_0  & = & - (\triangle_{G_0} -2)^{-1}\Big(  \frac{d}{dt} {\cal L}_{G_t} \dot{G_0} \Big|_{t=0} \Big)G_0 + L_{\nabla \dot{f}_0} G_0 \\
 & = & - \Big[ (\triangle_{G_0} -2)^{-1}  (\frac14 \triangle_{G_0} -1) \| \dot{G}_0 \|^2 \Big] G_0 + L_{\nabla \dot{f}_0} G_0 \\
 & = & - \Big[  (\triangle_{G_0} -2)^{-1}  \frac14 (\triangle_{G_0} -2-2) \|\dot{G}_0 \|^2 \Big]G_0 + L_{\nabla \dot{f}_0} G_0 \\
 & = & - \frac14 \|\dot{G}_0\|^2 G_0 + \Big[ \frac12 (\triangle_{G_0} -2)^{-1} \|\dot{G}_0\|^2 \Big]  G_0 + L_{\nabla \dot{f}_0} G_0
\end{eqnarray*}

\begin{lemma}
For a non-negative function $f$ on $\Sigma$, $(\triangle_{G_0} -2)^{-1} f$ is non-positive.
\end{lemma}
\begin{proof}
We will show that $u := (\triangle_{G_0} -2)^{-1} f  \leq 0$.  By supposing that $u$ attains its maximum 
at a point $p$ with $u(p) > 0$, we have $[\triangle_{G_0} u](p) \leq 0$. As $-2u(p) <  0$, $f(p) =  [(\triangle_{G_0} -2)u](p) < 0$, a contradiction to the hypothesis $f \geq 0$.
\end{proof}
By  setting $\alpha = - \frac12 (\triangle_{G_0} -2)^{-1} \|\dot{G}_0\|^2 \geq 0$ and $Y =  \nabla \dot{f}_0$, 
we have an expression for $\ddot{G}_0$;
\begin{eqnarray*}
\ddot{G}_0 & = & - \frac{d}{dt} \Pi_{G_t}  \dot{G}_0 \Big|_{t=0}   - L_X G_0 \\
 & = & - \Big[  (- \frac14 \|\dot{G}_0\|^2  - \alpha) G_0 + L_{Y} G_0\Big] -  L_X G_0 \\
  & = & ( \frac14 \|\dot{G}_0\|^2  + \alpha) G_0 + L_{Z} G_0
\end{eqnarray*}
where we denoted the vector field $-X-Y$ by $Z$.
\end{proof}

\section{Harmonic Map Parameterizations}
\subsection{General setting for harmonic maps}
We recall the theory of harmonic maps\index{harmonic maps}. Our treatment of the 
subject is by no means complete, and interested readers are referred
to several standard texts (for example \cite{Jo, EL2, SY}) available. 

Let $u:(M, g) \rightarrow (N, G)$ be a $C^1$-map. Let $du$ denote the section of the bundle 
$E := T^* M \otimes u^*(TN)$ for which there is the induced metric
$\langle X \otimes Y, W \otimes Z \rangle_E := g^*_x (X, W) G_{u(x)}(Y, Z)$. When $\{x^i\}$ (and $\{ y^\alpha \}$) is a local coordinate system near a point $p$ in $M$ (and a point $u(p)$ in $N$ respectively,) locally $du: TM \rightarrow TN$ is expressed as 
\[
du = \sum_{i, \alpha} \frac{\del u^\alpha}{\del x^i} dx^i \otimes \frac{\del}{\del y^\alpha}.
\]
Then one can define the energy of the map as
\[
E(u) = \int_M \frac12 \| du \|^2 d \mu_g 
\] 
where the integrand $\frac12 \|du \|^2_E (x) = \frac12 \langle du(x), du(x) \rangle_E$ is the energy density, also denoted by $e(u)$, locally  written as
\[
\frac12 \|du \|^2(x)  =  \frac12 g^{ij}(x) G_{\alpha \beta}(u(x)) \frac{\del u^\alpha}{\del x^i} \frac{\del u^\beta}{\del x^j} 
  =  \frac12 \tr_g (u^* G)     
\]
where 
\[
g (x) = g_{ij}(x) dx^i \otimes dx^j, \,\,\, g^*(x) = g^{ij}(x) \frac{\del}{\del x^i} \otimes \frac{\del}{\del x^j} ,\,\,\, 
G (y) = G_{\alpha \beta} (y) dy^\alpha \otimes dy^\beta, 
\]
$u^* G$ is the pulled-back metric tensor of $G$ by $u$:
\[
(u^* G)_{ij} = G_{\alpha \beta}(u(x)) \frac{\del u^\alpha}{\del x^i}(x) \frac{\del u^\beta}{\del x^j}(x) \, dx^i \otimes dx^j 
\]
and $\tr_g (u^* G) = g^{ij} (u^* G)_{ij}$. 

This bundle $E$ has an induced Levi-
Civita connection from the connections $\nabla^{T^*M}$ and $\nabla^{TN}$ of $g$ and $G$ respectively,
\[
\nabla^E_X (Y \otimes Z) = (\nabla_X^{T^*M} Y) \otimes Z + Y \otimes 
\nabla_X^{u^* TN} Z  
\]    
where $\nabla_X^{u^* TN} Z(x) := \nabla_{du(X)}^{TN} (u_*Z) (x).$ 

Consider the situation  where $u$ is stationary, that is, the first 
variation of the energy functional vanishes under arbitrary smooth variations
of the map of the form $u_\e (x)$ with $u_0 = u$ and $\frac{d}{d\e}u_\e|_{\e=0} = W
$.  Note that each $W$ is a smooth section of the bundle $E = T^* M \otimes u^* TN$. Writing down what this means pointwise, we obtain
\[
\delta E (u)(W) := \frac{d}{d\e} E(u_\e) |_{\e =0} 
  = \frac{d}{d\e} \int_M \frac12 \langle d u_{\e}, d u_{\e} \rangle_E d \mu_g |_{\e =0} 
\]
\[  
  =  \int_M \langle d W, d u_0 \rangle_E d \mu_g   
  =  \int_M \langle W, d^* d u_0 \rangle_E = 0.
\]
This holds for arbitrary $W$, which in turn implies that $d^* d u_0=0$.
The adjoint operator $d^*$ of the differential $d$ acting on the smooth 
sections of $T^*M \otimes u^* TN$ is $X \mapsto \tr_g (\nabla^E X)$. 
Locally, we have 
\begin{eqnarray*}
\nabla^E du  &= & g^{ij} \nabla^E_{\del_j} \Big( u_i^\beta dx^i \otimes 
\frac{\del}{\del y^\beta} \Big) \\
 & = & u_{ij}^\beta dx^j \otimes dx^i \otimes \frac{\del}{\del y^\beta} \\
 & & + \Big[ u_i^\beta (\nabla^{T^*M}_{\del_j} dx^i)\otimes dx^j \Big]\otimes
  \frac{\del}{\del y^\beta} + u^\beta_k dx^i \otimes 
  \nabla_{\del_j}^{u^*TN} \frac{\del}{\del y^\beta} \\
  & = & \Big( u^\beta_{ij} + u^\beta_k \Gamma_{kj}^i(x) + \Gamma^\beta_{\alpha \gamma}(u(x)) u^\alpha_i u^\gamma_j \Big) dx^i \otimes dx^j \otimes \frac{\del}{\del y^\beta} \\
  & = & ([{\rm Hess}(u)]_{ij}+ \Gamma^\beta_{\alpha \gamma} u^\alpha_i u^\gamma_j) dx^i \otimes dx^j \otimes \frac{\del}{\del y^\beta}. 
\end{eqnarray*}
By taking the $g$-trace of the above, $d^* du$ is written  as 
\[
d^* du = \tr_g (\nabla^E du )= (\triangle_g u^\beta + g^{ij} \Gamma^\beta_{\alpha \gamma} u^\alpha_i u^\gamma_j) \frac{\del}{\del y^\beta}.
\] 
The vanishing of $d^* du$ is called the harmonic map equation, and locally
written as
\[
\triangle_g u^\beta + g^{ij} \Gamma^\beta_{\alpha \gamma} u^\alpha_i u^\gamma_j = 0
\]
for $1 \leq i,j \leq {\rm dim} M$ and $1 \leq \beta \leq {\rm dim N}$.
When ${\rm dim} M = 1$, this is nothing but the geodesic equation.

In what follows, the target manifolds of harmonic maps are of 
non-positive sectional curvature, and  the following 
theorem covers all the situations we will be concerned with.  

We now quote the following existence and uniqueness statements of harmonic maps in situations we are interested in.  This version comes from a collection of results by Eells-Sampson \cite{ES} who 
showed the existence and the regularity, and by Hartman \cite{Har} and 
Al'bers \cite{Al} independently who showed the uniqueness. \\

\noindent {\bf Theorem} (Existence and Uniqueness of Harmonic Maps) {\it Let $(M^n, g)$ be a closed manifold, and $(\Sigma^2, G)$  a surface of non-positive sectional curvature.  Suppose there is a continuous map $\phi: (M^n, g) \rightarrow (\Sigma^2, G)$.  Then there exists a smooth harmonic map homotopic to $\phi$.  When the sectional curvature of $G$ is strictly negative and the image of the map is not a point or a closed geodesic, then the harmonic map is unique.} \\

Furthermore by utilizing the inverse function theorem, Eells-Lemaire \cite{EL} and Koiso \cite{Ko} showed; \\

\noindent {\bf Theorem} (Smooth dependence on target metric variations)
{\it
Let $(M^n, g)$ be a closed manifold, and $(\Sigma^2, G)$  a closed surface with a hyperbolic metric $G$.  For a smooth deformation $G_t$ of the hyperbolic metric $G =: G_0$ in the space of smooth metrics on $\Sigma$, the resulting harmonic maps $u_t: (M, g) \rightarrow (\Sigma^2, G_t)$ are smoothly dependent in $t$.     
 } \\
 
In Eells-Lemaire's statement, there is a technical condition that 
the Hessian of the energy functional of the harmonic map $u$ under   
variations of the map is positive-definite.  This is satisfied for 
the harmonic map $u$ into the hyperbolic surface $(\Sigma, G)$, for the 
harmonic map in this case is the unique  energy minimizing map in its
homotopy class.

\subsection{Harmonic maps between surfaces}
Recall that a two-dimensional Riemannian surface is a Riemann surface, 
i.e. each Riemannian metric $g$ is conformal to $dz \otimes d
\overline{z} = dx \otimes dx + dy \otimes dy$ for some local coordinate 
chart $z = x+ iy$, so that $g = \lambda(z) dz \otimes d\overline{z}$ for 
some function $\lambda >0$. Such a $z$ is called an isothermal coordinate.  
Now consider the situation when both the domain $M$ and the target $N$ 
are Riemannian  surfaces, and for $p$ and $u(p)$, choose isothermal coordinate $z=x 
+iy$ and $w = u_1+ iu_2$ around $p$ and $u(p)$ respectively. We then have
\[
g = \lambda(z) |dz|^2 \mbox{ and }  G = \rho(w) |d w|^2,
\]
and with the differential operators 
\[
\frac{\del}{\del z} = \frac12 \Big( \frac{\del}{\del x} - i \frac{\del}{\del y} \Big), \,\,\, \frac{\del}{\del \overline{z}} = \frac12 \Big( \frac{\del}{\del x} + i \frac{\del}{\del y} \Big) 
\]
we have 
\[
e(u) = \frac12 \| du \|_E^2 = \frac12 g^{ij}G_{\alpha \beta} u^\alpha_i u^\beta_j = \frac{\rho(u(z))}{\lambda(z)} \Big( \Big|\frac{\del u}{\del z} \Big|^2 + \Big|\frac{\del u}{\del \overline{z}} \Big|^2\Big). 
\]
Introduce the following notation
\[
|\del u|^2 := \frac{\rho(u(z))}{\lambda(z)} \Big|\frac{\del u}{\del z} \Big|^2 \mbox{ and } |\overline{\del} u|^2 := \frac{\rho(u(z))}{\lambda(z)}  \Big|\frac{\del u}{\del \overline{z}} \Big|^2.   
\]
Then we have
\[
e(u)=|\del u|^2 + |\overline{\del} u|^2 
\]
while the Jacobian $J(u)$ of the map $u$ is equal to 
\[
J(u) = \sqrt{{\rm det} (du)} =  \frac{\rho(u(z))}{\lambda(z)} \Big( 
\frac{\del u_1}{\del x}\frac{\del u_2}{\del y} - \frac{\del u_1}{\del y}\frac{\del u_2}{\del x}\Big) = |\del u|^2 - |\overline{\del} u|^2. 
\]
The harmonic map equation expressed with those isothermal coordinates turns into
\[
u_{z \overline{z}} + \frac{\rho(u)_u}{\rho(u)} u_{\overline{z}} u_z = 0.
\] 
Note that the conformal factor $\lambda$ for $g$ does not appear in the 
harmonic map equation above, which is explained  by the fact that 
the energy is conformally invariant, namely its value is unchanged by
replacing $g = \lambda(z) |dz|^2$ by $\tilde{g} : = \tilde{\lambda}(z)|dz|^2$, and hence if $u$ is harmonic with respect to $g$, then it is
harmonic with respect to $\tilde{g}$.  

Now we make a remark about linearizing the harmonic map equation 
using the isothermal coordinate $z$ of the Poincar\'{e} disc  
$({\bf D}, G_0=\rho(z)|dz|^2)$ around the identity map. Let $u_t: (D, G_0) \rightarrow (D, G_0)$
be a one-parameter family of harmonic maps with $u_0 = {\rm 
Id}_{\bf D}$ and $\frac{d}{dt} u_t \Big|_{t=0} = V(z)$.  Then as
the harmonic map equation is satisfied for each $t$, 
differentiating the equation in $t$ at $t=0$ when $u_0(z) = z$, one obtains 
\[
V_{z \overline{z}} + \frac{\rho(z)_z}{\rho(z)} V_{\overline{z}} = 0
\] 
which in turn says that $\rho(z) V_{\overline{z}}(z)$ is 
anti-holomorphic.
 
We have previously encountered the equation 
$V_{\overline{z}} = \mu(z)$ as the linearization of the Beltrami 
equations $w(\varepsilon)_{\overline{z}} = \varepsilon \mu 
w(\varepsilon)_z$ with $V(z) := \frac{d}{d \varepsilon} 
w(\varepsilon)|_{\varepsilon = 0} (z)$.  In particular, $\mu$ was said to be harmonic 
when $\rho(z) \overline{\mu}(z)$ is locally holomorphic in the 
isothermal  coordinate $z$.    

Combining these observations together, we conclude that   
the tangent space at the identity map to the space of harmonic 
diffeomorphisms $\{ u_t:{\mathbb H}^2/\Gamma_0 \rightarrow {\mathbb 
H}^2/\Gamma_t \}$ where 
$\{\Gamma_t\}$ is the set of deformations of the Fuchsian groups 
are represented by the space of harmonic Beltrami differentials.

One further remark relevant to this observation is that instead of the harmonic 
diffeormorpshisms $u_t$, one can consider the family of Douady-Earle 
extensions $w_t$ and its tangent space at the identity map, 
to get the same conclusion, as described at the very end of the paper \cite{DE}.

\subsection{The \teich of the torus and its \weil metric}

We give an explicit description of the \teich space of the torus\index{\teich space of torus}.
We furthermore specify the \weil metric on it. The 
identification of the space is done through harmonic maps.  This 
sets a model for higher genus surfaces in the next  section.

First choose a reference torus ${\mathbb T}^2_0$ which may as 
well be chosen to be 
${\mathbb R}^2/{\mathbb Z}^2$ where ${\mathbb Z}^2$ is the standard integer 
lattice $\Gamma_0 \simeq {\mathbb Z}^2$ in the $x$-$y$ plane. 
We denote the resulting flat metric by $g_0$. 
Let ${\cal M}_0$ be the set of all flat metrics on ${\mathbb T}^2$ of unit area. The space of all  smooth metrics on the torus is thus {\it uniformized} by the elements of ${\cal M}_0$. Each element $({\mathbb T}^2, g)$ of ${\cal M}_0$ 
can be parameterized by a harmonic map $u: ({\mathbb T}^2, g_0) \rightarrow ({\mathbb T}^2, g)$ in the same homotopy class as the identity map ${\rm Id}: {\mathbb T}^2_0 \rightarrow {\mathbb T}^2$ as shown above.  Then a standard 
formula (see for example \cite{EL2}) often referred to as the Bochner-Weitzenb\"{o}ck  formula, applied to this situation  
says that 
\[
\triangle_{g_0} \frac12 \|d u \|^2 = \| \nabla^E du \|^2.  
\]
Integrating this equality over the domain surface ${\mathbb T}_0$ 
we conclude that the map $u$ is totally geodesic ($\nabla^E (du) 
\equiv 0$), namely $u$ is an affine map.   

This in turn says that the pulled-back metric $u^* g$ is locally constant; namely 
\[
(u^* g)_{ij} dx^i\otimes dx^j = \frac{\del u^\alpha}{\del x^i} \frac{\del u^\beta}{\del x^j} g_{\alpha \beta} dx^i\otimes dx^j
\]
is constant. Note that by looking at the pulled-back metric $u^* g$ of the elements of $g$ in ${\mathcal M}_0$, we are actually looking at the point $[g]$ in the \teich space $\T_1 = {\cal M}_0 / {\rm Diff}_0 {\mathbb T}^2$, as the pull-back actions of the diffeomorphisms are isometries so that $[g]=[\phi^*g]$ for $\phi \in {\rm Diff}_0 {\mathbb T}^2$.   By identifying ${\mathbb T}_0$ with the fundamental region $[0,1] \times [0,1]$ in ${\mathbb R}^2$, we can regard $u^* g$ as an inner product structure on ${\mathbb R}^2$ with its determinant of the bilinear form equal to one, due to the unit volume normalization. By introducing an equivalence relation under 
the rotations around the origin parameterized by ${\rm SO}(2)$ which preserve the standard inner product structure, we can identify the space of such inner product structures as
\[
\{ G = \left(
  \begin{array}{ c c }
     g_{11} & g_{12} \\
     g_{21} & g_{22}
  \end{array} \right) \,\,  : \,\, g_{11} g_{22} - g_{12}^2 = 1 \} / _\sim  = {\rm SL}(2, {\mathbb R})/{\rm SO}(2)
\]
which is the hyperbolic space ${\mathbb H}^2$ as a set.  

To identify the metric structure, now we recall the characterization of the tangent vectors of \teich spaces as {\it trace-free, transverse} symmetric $(0,2)$ tensors; 
\[
\tr_G h = 0  \,\,\, \mbox{ and } \,\,\, \delta_G h = 0.
\]
We also saw that such an $h$ can be identified with a holomorphic quadratic differential $h^* = \phi(z) dz^2$.  On a torus $({\mathbb T}^2, g)$, the (complex) dimension of the space of  holomorphic quadratic differentials is one, namely differentials are locally constant $c \, dz^2$ for an isothermal complex coordinate $z$. This says that the trace-free transverse deformation tensor $h \in T_{g} \T_1$, with respect to the standard coordinate of ${\mathbb R}^2$ is of the form of a traceless  matrix $(h_{ij})$ with constant components.  By the above argument, a flat metric $g$ with unit volume is represented as a point $[G]$ in  ${\rm SL}(2, {\mathbb R})/{\rm SO}(2)$.  Let $h$ and $k$ be two traceless transverse tensors in $T_{[G]} \T$, then the \weil pairing is 
\[
\langle h, k \rangle_{[G]} = G^{ij} G^{kl} h_{ik} h_{jl}. 
\]
Note that this is precisely the left-invariant Riemannian metric of the homogeneous space ${\rm SL}(2, {\mathbb R})/{\rm SO}(2)$, which makes the space isometric to the hyperbolic plane ${\mathbb H}^2$.

\begin{theorem}
The \teich space of torus with the \weil metric is isometric to the hyperbolic disc.
\end{theorem}

We remark that the affine harmonic map $u: ({\mathbb T}^2, g_0) \rightarrow ({\mathbb T}^2, g)$ is also the \teich map in the sense that the map is extremal in minimizing the complex dilatation
as a quasi-conformal map between the two surfaces, as explained in \cite{Le} and \cite{BPT}.  Hence in this instance, the \teich geometry and the \weil  geometry coincide, and the \weil geodesics are \teich geodesics.

\subsection{\teich space of higher genus surface}

We saw the effectiveness of harmonic maps in parametrizing points in the \teich space of the torus with respect to the \weil metric.  In higher genus surfaces $(g>1)$, \weil geometry involves much non-linearity, but curiously it has many geometric structures, due to the existence of many convex functionals. 

In what follows, we will be solely concerned with the cases where the target manifolds of harmonic maps are hyperbolic surfaces.  In particular
we are interested in the variational theory of the harmonic maps where 
the variable is the hyperbolic metric on the target surface, or rather, 
the equivalent classes of hyperbolic metrics, representing points in the 
relevant \teich space.  To be precise, we fix a domain manifold $(M, g)$ 
which is compact without boundary equipped with a Riemannian metric $g$,  a 
topological surface $\Sigma$ of higher genus $g(\Sigma)>1$, and a 
continuous map $\Phi: M \rightarrow \Sigma$. When the surface $\Sigma$ is equipped with a hyperbolic metric $G$, by the above existence and 
uniqueness statements,  there exists a smooth harmonic map $u$ 
which is homotopic to $\phi$, which is unique (up to rotations in case 
of $M = S^1$) in its homotopy class $[\Phi]$.  Naturally the harmonic 
map $u$ depends on the hyperbolic metric $G$. Let $\tilde{G}$ be a hyperbolic metric on $\Sigma$ such that $\tilde{G} = \varphi^* G$ for some $\varphi$ in ${\rm Diff}_0 \Sigma$. Then
the map $u \circ \varphi^{-1}: (M, g) \rightarrow (\Sigma, \tilde{G})$ is
still a (unique) harmonic map homotopic to $\Phi$, for the map $\varphi^{-1}: 
(\Sigma, G) \rightarrow (\Sigma, \tilde{G})$ is an isometry by definition of the pulled-back metric.  We remark that when $\varphi$ is
an element in ${\rm Diff} \Sigma \setminus {\rm Diff}_0 \Sigma$, then
 $ \varphi^{-1} \circ u: (M, g) \rightarrow (\Sigma, \tilde{G})$ is still 
 harmonic as $\varphi^{-1}: (\Sigma, G) \rightarrow (\Sigma, \tilde{G})$  
is an isometry, but the composite map $ \varphi^{-1} \circ u$ is no 
longer homotopic to $\Phi$.  

This observation tells us that the correspondence $G \mapsto u(G)$
is well-defined when seen as $[G] \mapsto [u(G)]$ where $[G]$ is a
point in the \teich space of $\Sigma$, and $[u(G)]$ is an equivalence 
class where $u(G) \sim u(\varphi^*G)(= u \circ \varphi^{-1})$ when $\varphi$ is in ${\rm Diff}_0 
\Sigma$. In particular, the energy functional ${\cal E}(G) := E(u(G))$ of the map $u(G)$ 
\[
{\cal E}: {\cal M}_{-1} \rightarrow {\mathbb R}
\] 
can be seen as a functional defined on $\T$;
\[
{\cal E}: {\cal M}_{-1}/{\rm Diff}_0 \Sigma \rightarrow {\mathbb R} 
\]   
where ${\cal E}([G]):= E(u(G))$.

We now demonstrate the following theorem \cite{Y1}.  

\begin{theorem}[\weil convexity of energy\index{\weil convexity of energy}]
The energy functional ${\cal E}: {\T} \rightarrow {\mathbb R}$
is strictly convex with respect to the \weil metric. 
\end{theorem}

\begin{remark} The following proof closely follows the one given in Tromba \cite{Tr2} {\it except}
for the value of the second time derivative of the path $G_t$.  
\end{remark}

\begin{proof}
We will show that  for a horizontal lift $G_t \in {\cal M}_{-1}$ of an arbitrary \weil geodesic $\sigma(t) \in \T$, $\frac{d^2}{dt^2} {\cal E} (G_t) > 0$.  As seen above, for each $G_t$, we have a unique harmonic map $u_t:(M, g) \rightarrow (\Sigma, G_t)$.  We denote the first variation tensor of $G_t$
by $\dot{G}_t$ and the first variation vector field  $\frac{d}{dt} u_t $ by $W_t$ respectively.   We note that the following calculation is done 
around $t =0$ so that $G_t = G_0 + \dot{G}_0 t + \frac12 \ddot{G}_0 t^2 + o(t^2)$ in ${\cal M}$, but the proof works for any other value of $t$.  

The energy functional ${\cal E}(G_t) = E(u_t)$ is 
\[
E(u_t) = \int_M \frac12 \tr_g (u_t^* G_t) d \mu_g 
\]
where the dependence on $t$ appears on the hyperbolic metric $G_t$ and the harmonic  map $u_t$.
The first time-derivative is 
\begin{eqnarray*}
\frac{d}{dt} {\cal E}(G_t) \Big|_{t = 0} & = & \frac12 \int_M \tr_g \Big( \frac{d}{dt} u_t^* G_t \Big|_{t = 0}  \Big) d \mu_g \\
 & = & \frac12 \int_M  \tr_g \Big( \frac{d}{dt} u_0^* G_t \Big|_{t = 0} \Big)  +  \tr_g \Big( \frac{d}{dt} u_t^* G_0 \Big|_{t = 0}  \Big) d \mu_g    \\
 & = & \frac12 \int_M  \tr_g \Big( u_0^* \dot{G}_0 \Big)  +  \tr_g \Big( u_0^* [L_{W_0} G_0] \Big) d \mu_g 
\end{eqnarray*} 
where $ u_0^* [L_{W_0} G_0] $ is the pulled-back tensor of the Lie derivative.  Note that the second term 
$\int_M \tr_g \Big( u_0^* [L_{W_0} G_0] \Big) d \mu_g$ is the first variation of the energy in the direction of 
$W_0(u_0(x))$, which vanishes since the map $u_0$ is harmonic.  Actually as $u_t$ is harmonic for all $t$, we have  $\delta (E(u_t)) (W_t) := \int_M \tr_g \Big( u_t^* [L_{W_t} G_t] \Big) d \mu_g =0$ for all $t$.  Hence the first time-derivative should be written as
\[
\frac{d}{dt} {\cal E}(G_t) \Big|_{t = 0}  =  \frac12 \int_M  \tr_g \Big( u_t^* \dot{G}_t \Big) d \mu_g \Big|_{t = 0}. 
\]

As for the second time-derivative,  we get 
\begin{eqnarray*}
\frac{d^2}{dt^2} {\cal E}(G_t) \Big|_{t = 0} & = & \frac{d}{dt} \Big( \frac12 \int_M \tr_g \Big( u_t^* \dot{G}_t \Big)  d \mu_g \Big) \Big|_{t = 0}  \\
 & = & \frac12 \int_M  \tr_g \Big( u_0^* \ddot{G}_0  \Big)  +  \tr_g \Big(  u_0^*  [L_{W_0} \dot{G}_0]  \Big) d \mu_g.   
\end{eqnarray*} 

On the other hand, we differentiate the vanishing condition of the first variation of the energy 
\[
\int_M \tr_g \Big( u_t^* [L_{W_t} G_t] \Big) d \mu_g =0
\]
in $t$ and evaluate at $t=0$, to obtain the following relation;
\[
\frac12 \int_M  \tr_g \Big( u_0^* [L_{W_0} [L_{W_0} G_0]] \Big) d \mu_g +  \tr_g \Big( u_0^* [L_{W_0} \dot{G}_0] \Big) d \mu_g =0.
\]
Note that the first term $\frac12 \int_M  \tr_g \Big( u_0^* [L_{W_0} [L_{W_0} G_0]] \Big) d \mu_g$ is the second variation $\delta^2 (E(u))(W_0, W_0)$ of the energy in the directions of $W_0$ and $W_0$. 

Hence for now, we have an expression of the 
second time-derivative of the energy functional ${\cal E}(G_t)$;
\[
\frac{d^2}{dt^2} {\cal E}(G_t) \Big|_{t = 0} = \frac12 \int_M  \tr_g \Big( u_0^* \ddot{G}_0  \Big) d \mu_g - \delta^2 (E(u))(W_0, W_0).
\]

We use the following estimate \cite{Tr2} for the second variation term:
\begin{lemma}
\[
\delta^2 (E(u))(W_0, W_0) \leq \frac18 \int_M \tr_g ( u_0^*[\|\dot{G}_0\|^2 G_0] ) d \mu_g. 
\]
\end{lemma}
\begin{proof}
By using the equality obtained above, we have
\[
\delta^2 (E(u))(W_0, W_0) = - \frac12 \int_M  \tr_g \Big(  u_0^*  [L_{W_0} \dot{G}_0]  \Big) d \mu_g 
= -\frac{d}{dt} \Big[ \frac12 \int_M \tr_g (u_t^* \dot{G}_0) d \mu_g \Big] \Big|_{t=0} 
\]
\[
=
-\frac{d}{dt} \Big[ \frac12 \int_M g^{ij}(x) (\dot{G}_0)_{\alpha \beta} 
(u_t(x)) (u_t)^\alpha_i (u_t)^\beta_j d \mu_g \Big] \Big|_{t=0}
\]
\[
= -\frac12 \int_M (\dot{G}_0)_{\alpha \beta, \gamma} W_0^\gamma g^{ij} u^\alpha_i u^\beta_j d \mu_g
- (\frac12 + \frac12) \int_M  (\dot{G}_0)_{\alpha \beta} g^{ij} W^\alpha_i u_j^\beta d \mu_g
\]
\[
=  \frac12 \int_M (\dot{G}_0)_{\alpha \beta} W^\alpha \triangle_g u^\beta d \mu_g - \frac12 \int_M (\dot{G}_0)_{\alpha \beta} g^{ij} 
W_i^\alpha u_j^\beta d \mu_g 
\]
\[
=  -\frac12 \int_M (\dot{G}_0)_{\alpha \beta} W^\alpha g^{ij} 
\Gamma^\beta_{\gamma \delta} u^\gamma_i u^\delta_j d \mu_g - 
\frac12 \int_M (\dot{G}_0)_{\alpha \beta} g^{ij} 
W_i^\alpha u_j^\beta d \mu_g 
\]
\[
= - \frac12 \int_M (\dot{G}_0)_{\alpha \beta} g^{ij} (\nabla^{u^*T\Sigma}_i W)^\alpha u_j^\beta d \mu_g 
\]
where the fifth equality comes from integration by parts, the sixth from the harmonic map equation, and the sixth uses the trace-free condition of $(\dot{G}_0)$ with respect to the geodesic normal coordinates.  

With respect to the geodesic normal coordinates we have $g^{ij}(x) = \delta^{ij}$ and $G_{\alpha \beta}(u(x)) = \delta_{\alpha \beta}$ and  the Cauchy-Schwarz inequality says
\[
\delta^2 (E(u))(W_0, W_0) = -\frac12 \int_M (\dot{G}_0)_{\alpha \beta} g^{ij} (\nabla^{u^*T\Sigma}_i W)^\alpha u_j^\beta d \mu_g 
\] 
\[
\leq \sum_{i=1}^n \Big(\-\frac18 \int_M \{(\dot{G}_0)_{11}^2 + (\dot{G}_0)_{12}^2 \}\{(u_i^1)^2 + (u_i^2)^2 \} d \mu_g 
+ 
\]
\[
\frac12 \int_M [(\nabla_i^{u^*TN}W)^1]^2 + [(\nabla_i^{u^*T\Sigma}W)^2]^2 d \mu_g \Big)
\]
\[
= \frac1{16} \int_M \tr_g (u^* [\|\dot{G}_0\|^2 G_0]) d \mu_g 
+ \frac12 \int_M \|\nabla^{u^*T\Sigma}W\|^2 d \mu_g. 
\]
Finally the second variation formula for the energy functional 
at a harmonic map $u_0$ gives 
\begin{eqnarray*}
\delta^2 (E(u_0))(W_0, W_0) &=& \int_M \|\nabla^{u^*T\Sigma}W\|^2 d \mu_g
- \langle R^\Sigma(W_0, du_0)du_0, W_0 \rangle_{L^2(G)} \\
&  \geq &  
\int_M \|\nabla^{u^*TN}W\|^2 d \mu_g,
\end{eqnarray*}
where the inequality is due to the negative sectional curvature of the surface $\Sigma$.

Combining the pair of inequalities, we obtain
\[
\delta^2 (E(u_0))(W_0, W_0) \leq \frac1{16} \int_M \tr_g (u^* [\|\dot{G}_0\|^2 G_0]) d \mu_g + \frac12 \delta^2 (E(u_0))(W_0, W_0) 
\]
which gives the statement of the lemma.
\end{proof}

Now we conclude the proof of the convexity by inserting the expression of $\ddot{G}_0$ obtained above within the integrand
\begin{align*}
\frac{d^2}{dt^2} {\cal E}(G_t) \Big|_{t = 0} & = &\frac12 \int_M  \tr_g \Big( u_0^* \ddot{G}_0  \Big) d \mu_g - \delta^2 (E(u))(W_0, W_0) \\
& \geq & \frac12 \int_M  \tr_g \Big( u_0^* \Big[ (\frac14\|\dot{G}_0\|^2 +\alpha ) G_0\Big]\Big) d \mu_g  \\
 & & + \int_M \tr_g \Big(u_0^* [L_Z G_0]\Big) d \mu_g - \frac18 \int_M \tr_g \Big( u_0^* [\|\dot{G}_0\|^2 G_0] \Big) d \mu_g \\
 & = & \frac12 \int_M  \tr_g \Big( u_0^* [\alpha  G_0 ]\Big) d \mu_g \geq 0.
\end{align*}
The term $\int_M \tr_g \Big(u_0^* [L_Z G_0]\Big) d \mu_g$ vanishes as this is the first variation of the energy along a one-parameter family of isometries.
Note that the inequality is an equality when $\alpha = - \frac12 (\triangle_{G_0} -2)^{-1} \|\dot{G}_0\|^2 \geq 0$ is zero as well as the integral of the curvature $\langle R^\Sigma(W_0, du_0)du_0, W_0 \rangle_{L^2(G)}$ is zero. The former never occurs for nontrivial geodesics $G_t$, which in turn implies that the energy functional is strictly \weil convex. 
\end{proof}

\subsection{Applications of \weil convexity}

We now introduce some applications, by specifying the domain manifold $(M, g)$ and the homotopy classes of the harmonic maps. 
In \cite{Y1}, in addition to the convexity, a condition for the properness of the energy functional was obtained. Recall that a strictly convex functional, which is also proper, has a unique point in its domain where the value is minimized.   

\begin{theorem}
Suppose that we have a family of harmonic maps $u: (M^n, g) \rightarrow (\Sigma^2, G)$ with varying hyperbolic metrics $G$ within a homotopy class so that the induced map $u_*: \pi_1(M) \rightarrow \pi_1(\Sigma^2)$ has a finite index image in 
$\pi_1(\Sigma^2)$. Then the energy functional ${\cal E}(G): \T(\Sigma) \rightarrow {\mathbb R}$ is proper, and hence there exists a unique $\cal E$-minimizing point $[G]$ in $\T$.
\end{theorem}

Schoen and Yau in \cite{SY0} considered harmonic maps of Riemann surfaces into a three dimensional manifold in order to find minimal immersions.  There the energy functional of those maps as the conformal structures of the {\it domain} surfaces are varied is shown to be proper provided the induced maps on $\pi_1$'s is {\it injective.} In a sense, our result is dual to theirs. 

The proof of this theorem is in \cite{Y1}. Here we present a \\
\noindent {\bf Sketch of Proof.} 
For each $C \in {\mathbb R}$, consider the sub-level set $S(C) := \{ G \in {\cal 
M}_{-1}: {\cal E}(G) \leq C < \infty \}$ of the energy functional.  We show that
${\cal M}_{-1} / {\rm Diff}_0 N $ is sequentially compact in the \teich space  ${\cal M}_{-1} / {\rm Diff}_0 N$.  The B\"ochner-Weitzenb\"{o}ck formula
combined with the standard elliptic estimate, often referred to as the De Giorgi-Nash-Moser iteration scheme\index{De Giorgi-Nash-Moser iteration}, provide us with the following estimate
\[
\sup_{x \in M} e(u)(x) \leq C \int_M e(u) d \mu_g 
\] 
where the constant $C>0$ depends on the sectional curvature $K_G$, which in 
our situation equals  $-1$, but independent of the choice of the hyperbolic 
metric $G$. This says that on the sub-level set $S(C)$, the energy density $e(u)$ 
has an upper bound, uniform in the point $x \in M$ as well as in $G \in {\cal 
M}_{-1}$. A geometric consequence of this fact is that the diameter of the image of the harmonic map $u: (M^n, g) \rightarrow (\Sigma^2, G)$ is uniformly bounded. 
 This prohibits the hyperbolic surface $(\Sigma^2, G)$ in $S(C)$ to develop a pinching neck, a consequence of the Collar Lemma.  This is because  for  all the  closed geodesics (not necessarily simple),  which are transverse to the pinching simple closed geodesic, their  hyperbolic length blows up due to the finite index condition of $u_*\pi_1(M)$ in $\pi_1(\Sigma)$.  
 
 This, namely the existence of a pinching neck,  would contradict the upper bound condition for the diameter obtained above. 
 
 Using the resulting lower bound of the length of the shortest simple closed geodesics, the Mumford-Mahler Compactness theorem\index{Mumford-Mahler Compactness theorem} (see, for example, \cite{Tr1, Jo}) says that the image of the sub-level set $S(C) \subset {\cal M}_{-1} / {\rm Diff}_0 \Sigma$ projected down to the moduli space ${\cal M}_{-1} / {\rm Diff} \Sigma$ is compact.  Given a sequence of points $G_i$ in 
 \teich space, this allows to find a convergent subsequence of the projected (from \teich space) sequence $[G_i]$ in the moduli space.  This can be rephrased as follows: there exists a sequence of diffeomorphisms $\{f_k\}$ of $\Sigma$ and a subsequence $\{ G_k\}$ of $\{G_i\}$ so that $\{ f_k^* G_k \}$ is convergent in ${\cal M}_{-1}$.  For the sequence of harmonic maps $\{u_k : (M, g) \rightarrow (\Sigma, G_k)\}$ in the given homotopy class, each $ f_k^{-1} \circ u_k$  is harmonic, though not necessarily in the same homotopy class.  As   $f_k^* G_k$ converges to a hyperbolic metric with the $C^\infty$ topology in ${\cal M}_{-1}$, for sufficiently large values of $k$, the homotopy type of  the harmonic maps $f_k^* G_k$ stabilizes, and one checks that  the $f_k$'s are homotopic to each other for the large $k$'s. This says that the subsequence $\{G_k\}$ converges to some hyperbolic metric $G_\infty$, proving the sequencial compactness of the sub-level set $S(C)$.

Now we make some specific choices of the domain manifold $(M^n, g)$ of the harmonic maps $u: (M^n, g) \rightarrow (\Sigma^2, G)$.  

\subsubsection{Harmonic maps from copies of $S^1$ to $\Sigma$.}   
This situation has been first investigated by  Wolpert \cite{W1}, who showed that the hyperbolic length functional ${\cal L}_\sigma:{\T} \rightarrow {\mathbb R}$ of each simple closed geodesic $\sigma$ is \weil convex.  He then chose a set $\{\sigma_i\}_1^N$ of $\sigma$'s which {\it fill} the surface $N$, namely, the complement of the loops are homeomorphic to a set of open discs, to make the functional $\sum_{i=1}^N {\cal L}_{\sigma_i}: {\T} \rightarrow {\mathbb R}$ proper.  This provides a proof of the so-called Nielsen Realization Problem as was demonstrated in \cite{W1}, which was first proven by 
S. Kerckhoff \cite{Ke} using the convexity of ${\cal L}_\sigma$ along earthquake deformations.

The comparison should be made between ${\cal E}$ and ${\cal L}_\sigma$ when 
the harmonic map $u: S^1 \rightarrow (\Sigma, G)$ maps onto the simple closed geodesic $\sigma$. In this case, the harmonic  map is not unique, and the lack of uniqueness corresponds to rotations of the domain $S^1 (\cong [0,1]/\sim)$, which would induces an $S^1$-action on the map $u$. However, the energy is invariant under the $S^1$-action and it is equal to the square of the hyperbolic length ${\cal E}(G) = {\cal L}_\sigma^2$. Note that Wolpert's \weil convexity of ${\cal L}_\sigma$ thus implies the \weil convexity of  ${\cal E}(G)$, but not vice-versa.    Another remark concerning  convexity is that in our harmonic map setting, the closed geodesic need not be simple; for example the image of the harmonic map may be an immersed closed geodesic in the surface.  
In this regard, recently M. Wolf has another proof of  \weil convexity \cite{Wo2} of hyperbolic length of geodesics, 
not necessarily simple nor closed.

\subsubsection{Harmonic maps from the surface $\Sigma$ to itself}
When the domain is the surface $\Sigma$ itself, and when the harmonic maps $(\Sigma_0, g) \rightarrow (\Sigma, G)$ with the hyperbolic metric $G$ varying in ${\cal M}_{-1}$,
are all homotopic to the identity map ${\rm Id}: \Sigma_0 \rightarrow \Sigma$, the energy 
functional is both convex and proper, as the map $u_*: \pi_1(\Sigma) \rightarrow \pi_1(\Sigma)$ is the identity map, hence surjective.  Therefore there exists a unique 
minimizer $[G_0]$ in $\T$. These maps are known to be diffeomorphisms by the results of Jost-Schoen \cite{JS}. The minimizer is specified by the hyperbolic metric
uniformizing the domain metric $g$, namely the hyperbolic metric $G_0$ conformal 
to $g$.  This can be seen by the following argument. 

Recall that the energy of the harmonic  map
$u$ is written with respect to the isothermal coordinates $z$ around $x$ and $w$  
around $u(x)$ as 
\[
\int_\Sigma e(u) d \mu_g = \int_\Sigma (|\del u|^2 + |\overline{\del} u|^2) d \mu_g. 
\]
On the other hand for the maps homotopic to the identity, the mapping degree is one. 
Hence, the integral of the pulled-back volume form $u^* d \mu_G$ coincides with the integral of the pulled-back K\"{a}hler form which is equal to minus the
Euler characteristic of the surface $\Sigma$;
\[
\int_\Sigma J(u) d \mu_g = \int_\Sigma ( |\del u|^2 - |\overline{\del} u|^2) d \mu_g = -\chi(\Sigma).
\]
Note that  
\[
E(u) + \chi (\Sigma) = \int_\Sigma ( e(u) - J(u) ) d \mu_g = \int_\Sigma 2 |\overline{\del} u|^2 d \mu_g 
\]
implies that  the energy functional ${\cal E}$ differs from 
the $\overline{\del}$-energy\index{energy!$\overline{\del}$-} ${\cal E}_{\overline{\del}} := \int_\Sigma 2 |\overline{\del} u|^2 d \mu_g$ by a topological constant.  
Therefore the \weil convexity of the energy functional ${\cal E}$ is equivalent to the \weil convexity of the $\overline{\del}$-energy functional 
${\cal E}_{\overline{\del}}: {\T} \rightarrow {\mathbb R}$ where ${\cal E}_{\overline{\del}}(G) := \int_\Sigma 2 |\overline{\del} u|^2 d \mu_g$ for the harmonic  map
$u: (\Sigma_0, g) \rightarrow (\Sigma,G)$. 

We remark that the $\overline{\del}$-energy measures the quasi-conformality of the map $u$, in particular, is equals to zero when the map $u$ is conformal. In particualar, the usual Dirichlet energy and the $\overline{\del}$-energy can be written down by using the Beltrami coefficient $\mu(z) := u_{\overline{z}}/u_z$ as
\[
{\cal E}(G) = \int_\Sigma (1+|\mu|^2)|\partial u|^2 d \mu_g(z), \,\,\,
{\cal E}_{\overline{\del}}(G) =  \int_\Sigma 2|\mu|^2|\partial u|^2 d \mu_g(z).
\]
The $\overline{\del}$-energy functional is strictly \weil convex and proper, hence the collection of  sublevel sets $\{S(C)\}_{C \in {\mathbb R}_{\geq 0}}$ provide an exhaustion of the \teich space, with a point $[G_0] := \cap_{C>0} S(C)$  where the $\overline{\del}$-energy functional is uniquely minimized. This point $[G_0]$ is characterized by the conformal harmonic identity map, namely when $g$ and $G_0$ are conformal.

The harmonic map is a canonical object in relating a pair of hyperbolic surfaces,
where other canonical ways to relate them include the \teich map (see \cite{Le} for references) and the earthquake map \cite{Th1}. The \teich map requires no metric structures, only the conformal structures of the domain and the target, while the harmonic map requires the hyperbolic metric of the target surface, but only the conformal structure of the domain surface. The earthquake map does require hyperbolic structures on both the domain and the target surfaces. In this context, we draw the reader's attention to two papers 
on harmonic map theory between surfaces.  One is Y. Minsky's thesis \cite{Mi}, where the asymptotic behavior of harmonic maps are investigated in relation to measured foliations, and the other is C. Mese's work \cite{Me} where a conjecture by 
M. Gerstenhaber  and H. Rauch from 1954 was resolved, based on the preceding results by M. Leite \cite{Lei} and E. Kuwert \cite{Ku}. It says that the \teich map is a harmonic map between a pair of Riemann surfaces where the target is equipped with a singular metric induced from the \teich differential relating the two conformal structures.

\subsubsection{Harmonic maps from K\"{a}hler manifolds \index{harmonic maps from K\"{a}hler manifolds} to  $\Sigma$} 
The next situation is when the domain is a closed K\"{a}hler manifold $(M, g)$.
Consider a homotopy class of $\phi: (M, g) \rightarrow (\Sigma, G)$ so that the 
condition for  properness of the energy functional is met.   
As the hyperbolic surfaces $(\Sigma, G)$ are K\"{a}hler, by a result of Sampson (see \cite{Sa}), any holomorphic map between K\"{a}hler manifolds     
is harmonic, and in particular, energy-minimizing in the homotopy class $[\phi]$.  
As was stated in \cite{Y1}, the point $[G_0]$ in $\T$ minimizes 
the energy functional ${\cal E}: \T \rightarrow {\mathbb R} $.

\subsection{$\overline{\del}$-energy functional on the Universal \teich space}
\subsubsection{Asymptotically conformal harmonic maps\index{Asymptotically Conformal Harmonic Maps}\index{Harmonic Maps!Asymptotically Conformal}}

 In this section, we restrict to the situation in \cite{Y1} where there is a \weil geodesic $G_t$ in the universal \teich space ${\cal UT}$, and we parameterize the varying hyperbolic metrics by harmonic maps from the Poincar\'{e} disc $({\bf D}, G_0)$.  Recall that we already looked at the linear structure of the 
 universal \teich space, when the deformation tensor is induced by sufficiently smooth vector fields.
 In what follows, we will look into some nonlinear structures of the \weil geometry. 
 
 To make the setting precise, first recall that the Uniformaization Theorem guarantees that each complete hyperbolic metric $G_t$ on the unit disc ${\bf D}$ can be represented by $\varphi_t^* G_0$ with a quasi-conformal diffeomorphism
 $\varphi_t: {\bf D} \rightarrow {\bf D}$. We impose that for each $t$ the map $\varphi_t$ fix three points, say $(1,0), (0,1)$ and $(-1,0)$ on the boundary $\del {\bf D} = S^1$, to fix an ${\rm SL}(2, {\mathbb R})$ gauge. 
  We consider a harmonic map 
 \[
 u_t: ({\bf D}, G_0) \rightarrow ({\bf D}, G_t) 
 \]
which, viewed as a map from the unit disc ${\bf D}^2$ to itself, has the trivial extension   
 to the geometric boundary: $u_t |_{S^1} = {\rm Id}_{S^1}$.  This condition is in place of specifying the homotopy type of the harmonic map into the compact surfaces.  Recalling the  definition of the pulled-back metric, we have 
 \[
 ({\bf D}, G_0) \xrightarrow{u_t}  ({\bf D}, G_t)  \xrightarrow{\varphi_t} ({\bf D}, G_0)
 \]
 where the map $\varphi_t$ is an isometry. Thus the composite map 
 \[
 \varphi_t \circ u_t : ({\bf D}, G_0) \rightarrow ({\bf D}, G_0) 
\]
is a harmonic map with the asymptotic boundary condition $\varphi_t \circ u_t |_{S^1} = \varphi_t |_{S^1}$,
which is a quasi-symmetric map from $S^1$ to itself.  
  
 As defined above, given a harmonic map from a surface to a surface, one can define  the $\overline{\del}$-energy ${\cal E}_{\overline{\del}}(G) := \int_\Sigma 2 |\overline{\del} u|^2 d \mu_g$.  In our setting here, the usual energy is divergent as the area functional is divergent. Thus disregarding the $\del$-energy  $\int_\Sigma 2 |\del u|^2 d \mu_g$ amounts to a renormalization of the energy functional.  Indeed, in \cite{Y1}, this was made analytically  rigorous.  

First we remark that for $\phi \in {\rm QC}({\bf D})$ the $\overline{\del}$-energy of the harmonic map $u: ({\bf D}, G_0) \rightarrow ({\bf D}, G=\varphi^* G_0) $ is equal to the $\overline{\del}$-energy of the harmonic map  $\varphi \circ u: ({\bf D}, G_0) \rightarrow ({\bf D}, G_0)$. Hence our  $\overline{\del}$-energy functional ${\cal E}_{\overline{\del}} (G)$ formally defined on the universal \teich space ${\cal UT}$ is identified with the  $\overline{\del}$-energy of the harmonic map  $\varphi_t \circ u_t : ({\bf D}, G_0) \rightarrow ({\bf D}, G_0)$ where the $G$ dependence is replaced by the $\varphi_t |_{S^1}$-dependence.
This enables us to use harmonic maps to identify the points of the universal \teich space ${\rm QS}(S^1)/{\rm SL}(2, {\mathbb R})$, in analogy with the compact surface cases.  

\begin{theorem}
For a quasi-conformal map $\varphi: {\bf D} \rightarrow {\bf D}$ with boundary restriction $\varphi|_{S^1}: S^1 \rightarrow S^1$ in  ${\rm QS}(S^1) \cap C^2(S^1; S^1)$, let $u$ be the harmonic map 
from $({\bf D}, G_0)$ to itself with  asymptotic condition $u |_{S^1} = \varphi |_{S^1}$.   Then the $\overline{\del}$-energy ${\cal E}_{\overline{\del}} (\varphi^* G_0)$ is finite. 
\end{theorem}

We note that for a $C^{1, \alpha}$ diffeomorphism $f: S^1 \rightarrow S^1$, it is known by the work of Li-Tam \cite{LT} that there exists a unique proper harmonic map $u: ({\bf D}^2, G_0) \rightarrow ({\bf D}^2, G_0)$ with $u|_{S^1} = f$.  In particular the identity map is the unique harmonic map from the Poincar\'e disc to itself with
the trivial asymptotic boundary condition.  Therefore under the hypothesis of the above theorem,  the 
correspondence between the hyperbolic metric $\varphi^*G_0$ and the harmonic map $u$ is justified.     
  
The proof of this statement \cite{Y1} is by writing down the asymptotic expansion of the harmonic map near the geometric boundary.  It turns out, from the analysis of Li-Tam, that the harmonic map
is asymptotically conformal, namely the map $u$ becomes conformal as the point approaches  the
geometric boundary measured with respect to its defining function.  Then we impose the fact that the map has zero tension field $\tau(u)$, which is equivalent to the harmonic map equation.  The deviation of $u$ from a conformal map can be controlled by the fact that the consecutive terms in the asymptotic expansion   decay in  certain rates imposed by the vanishing of the tension field.  The $C^2$ regularity condition was used here in the expansion. 

In fact, we obtain that in the upper half plane model $G_0 = (dx^2 + dy^2)/y^2$, the energy density of $u$ as well as the area density/Jacobian of $u$ behave as
\[
e(u) = 1 + O(y^2), \,\,\, J(u) = 1 + O(y^2)
\] 
where $y>0$ is the defining function of the geometric boundary $\{y=0\}$. Thus the $\overline{\del}$-energy density $e(u) - J(u)=  2 |\overline{\del}u |^2$ is of the form $O(y^2)$, which in turn implies that its integral $\int_{{\mathbb H}^2} 2 |\overline{\del}u |^2 d \mu_{G_0}$ is finite.  Note that if the map is conformal, we have $e(u) = J(u) = 1$ everywhere, and the $\overline{\del}$-energy vanishes in that case.
We do not expect $C^2$ regularity of the asymptotic boundary map $\varphi$ to be optimal.  The optimal regularity should be closely linked to the regularity of the quasi-symmetric diffeomorphism obtained by integrating  a family of $H^{3/2}$-smooth vector fields on $S^1$, where the integration is specified by the \weil exponential map.  However, little understanding of the exponential map in infinite dimensional settings (cf. \cite{Ham, MM}) exists, and such a direction of research would be highly nontrivial as well as important.   

\subsubsection{${\overline{\del}}$-energy  as a \weil potential\index{\weil potential}}

Having the ${\mathbb R}$-valued functional ${\cal E}_{\overline{\del}}$ at hand, we look at its behavior near the hyperbolic metric $G_0$, or equivalently the identity map ${\rm Id}: ({\bf D}^2, G_0) \rightarrow ({\bf D}^2, G_0)$. 

\begin{theorem}
Suppose that $G_t = \varphi_t^* G_0$ is a \weil geodesic so that $\dot{G}_0$ is a Lie derivative $L_Z 
G_0$ where $Z$ is a divergent (in $G_0$ sense)  vector field inducing an $H^{3/2}$-smooth vector field 
on the geometric boundary $S^1$.  We further suppose that the one-parameter family of the 
quasi-symmetric maps $\varphi_t |_{S^1}$ are all $C^2$ smooth, corresponding to the finite $\overline
{\del}$-energy harmonic maps $u_t$.  Then the ${\overline{\del}}$-energy  ${\cal E}_{\overline{\del}} (G_t)$ 
satisfies 
\[
{\cal E}_{\overline{\del}} (G_0)= 0, \,\,\, \frac{d}{dt} {\cal E}_{\overline{\del}} (G_t) \Big|_{t=0} = 0 \,\,\, \mbox{and }
\frac{d^2}{dt^2} {\cal E}_{\overline{\del}} (G_t) \Big|_{t=0} = \frac12.
\]
In particular, the ${\overline{\del}}$-energy ${\cal E}_{\overline{\del}}$ defined on  ${\rm Diff} S^1/{\rm SL}(2, {\mathbb R}) \subset {\cal UT}$  is \weil convex at $[G_0]$. 
\end{theorem} 

\begin{proof}
The first statement ${\cal E}_{\overline{\del}} (G_0)= 0$ is clear, as the identity map $u_0: ({\bf D}, G_0) \rightarrow ({\bf D}, G_0)$ is the unique harmonic map, which is conformal.  

We denote by $W_{t_0}$ the vector field $\frac{d}{dt} u_t |_{t=t_0}$.   We first show
\begin{lemma}
The vector field $W_0$ is identically zero. 
\end{lemma}
\begin{proof}
Recall that for $G_t = \varphi_t^* G_0$, the map $\varphi_t \circ u_t: ({\bf D}, G_0) \rightarrow ({\bf D}, G_0)$ is a harmonic map with the asymptotic boundary condition $\varphi_t \circ u_t|_{S^1} = \varphi_t|_{S^1} $.  Denote $\varphi_t \circ u_t$ by $\tilde{u}_t$, and its time derivative $\frac{d}{dt} \tilde{u}_t |_{t=0}$ by $V(z)$, where $z$ is the standard complex coordinate for the unit disc ${\bf D}$. The harmonic map 
equation for $\tilde{u}_t$  is then
\[
\tilde{u}_{z \overline{z}}^t + \frac{\rho(\tilde{u}^t)_{\tilde{u}^t}}{\rho(\tilde{u}^t)}\tilde{u}^t_{\overline{z}}\tilde{u}^t_z = 0
\]   
where we have shifted the index  $t$  up for $\tilde{u}_t$ temporarily, and $\rho(w)= \frac{4}{(1-|w|^2)^2}$.  By differentiating the harmonic map equation in $t$, and evaluating at $t=0$, we have 
\[
V_{z \overline{z}} + \frac{\rho(z)_z}{\rho} V_{\overline{z}} = 0
\] 
which is equivalent to $\del_z(\rho(z) \del_{\overline{z}}V) = 0$.  Recall that this says that
the deformation tensor $L_V G_0$ is a traceless divergence-free tensor.

On the other hand, as $\varphi_t^* G_0$ is a \weil geodesic, by the $L^2$-decomposition theorem, its time derivative a $t=0$ is a traceless divergence free tensor, which can be written as $L_X G_0$ where $X = \frac{d}{dt} G_t |_{t=0}$.  
If we set $\frac{d}{dt} u_t |_{t=0} = W$, then by differentiating $\varphi_t \circ u_t$ at $t=0$,  we have $V(z) = X(z) + W(z)$.  As the $L^2(G_0)$ integrable space of traceless divegence-free tensors  are linear, 
$L_W G_0$ is again a traceless divergence-free tensor, or equivalently $\del_z(\rho(z) \del_{\overline{z}}
W) = 0$.   As $u_t$ is a one-parameter family of harmonic maps from $({\bf D}, G_0)$ to $({\bf D}, G_t)
$ fixing the geometric boundary $\del {\bf D}^2=S^1$, we have $W|_{S^1}=0$. As the only vector field 
satisfying $\del_z (\rho(z) \del_{\overline{z}}W) = 0$ with $\del {\bf D}=S^1$ is the trivial vector field, 
we obtain that $W_0=0$. 
\end{proof}

This lemma should be contrasted with the observation by Ahlfors \cite{Ah1} that  one can choose $3g-3$ harmonic Beltrami differentials so that their complex linear combinations  form a \weil geodesic normal coordinates.   Also M. Wolf \cite{Wo1} shows that the \weil geodesic  $G_t$ is approximated by a path $G_0 + t \dot{G_0}$ up to order two where $\dot{G}_0$ is a deformation tensor induced by  a harmonic Beltrami differential, which in this context, is equal to the linearized Hopf differential of the one-parameter family of harmonic maps whose target metrics are $G_t$.    

Now we differentiate the energy density $e(u_t, G_t) = \frac12 \tr_{G_0} (u^*_t G_t)$ and the area density 
$J(u_t, G_t)=\frac{\sqrt{\det(u_t^*G_0)}}{\sqrt{\det G_0}}$ in $t$.  
\begin{eqnarray*}
\frac{d}{dt}  e(u_t, G_t) \Big|_{t=t_0} & = & \frac{d}{dt}  \frac12 \tr_{G_0} (u^*_t G_t) \Big|_{t=t_0} \\
 & = & \frac12  \tr_{G_0} (u^*_{t_0} \dot{G}_{t_0}) + \frac12  \tr_{G_0} (u^*_{t_0} [L_{W_{t_0}} G_{t_0}]). 
\end{eqnarray*}
Note here that when $t_0=0$, $\frac{d}{dt}  e(u_t, G_t) \Big|_{t=t_0} =0$ as $W_0 = 0$, $u_0(z) = z$ and $\tr_{G_0} (\dot{G}_0)=0$.  Differentiate this expression one more time, and obtain
\begin{eqnarray*}
\frac{d^2}{dt^2}  e(u_t, G_t) \Big|_{t=t_0} & = & \frac{d}{dt} \Big[ \frac12  \tr_{G_0} (u^*_{t} \dot{G}_{t}) \Big]\Big|_{t=t_0} 
+\frac{d}{dt} \Big[ \frac12  \tr_{G_0} (u^*_{t} [L_{W_{t}} G_{t}])  \Big] \Big|_{t=t_0} \\
 & = & \frac12  \tr_{G_0} (u^*_{t_0} \ddot{G}_{t_0}) +  \frac12  \tr_{G_0} (u^*_{t_0} [L_{W_{t_0}} \dot{G}_{t_0}])  \\
  & & + \frac12  \tr_{G_0} (u^*_{t_0} [L_{\dot{W}_{t_0}} G_{t_0}]) + \frac12  \tr_{G_0} (u^*_{t_0} \{L_{W_0}[L_{W_{t_0}} \dot{G}_{t_0}]\})  \\
  & & + \frac12  \tr_{G_0} (u^*_{t_0} [L_{W_{t_0}} \dot{G}_{t_0}]).
\end{eqnarray*} 
By evaluating the expression at $t_0=0$, we get 
\[
\frac{d^2}{dt^2}  e(u_t, G_t) \Big|_{t=0} = \frac12  \tr_{G_0} (u^*_{t_0} [L_{\dot{W}_{0}} G_{0}])  +  \frac12  \tr_{G_0} (u^*_{0} \ddot{G}_{0}).
\]
Recall that in the Fuchsian setting  the term $\frac12  \tr_{G_0} (u^*_{t_0} [L_{W_{t_0}} G_{t_0}])$ was  identically zero, as this is the first variation of the energy and the maps $u_t$ were all harmonic as $t$ varies.   This would be the case here too {\it if} the variational vector field $W$ were compactly supported.  
However in the current setting, the hyperbolic norm of the vector fields $W_{t_0}\,\,\,  (t_0 \neq 0)$ is asymptotically  divergent, and hence the term cannot be assumed to vanish.   

The derivatives of the area density are given by
\begin{eqnarray*}
\frac{d}{dt} J(u_t, G_t) \Big|_{t=t_0} &=& \frac{d}{dt}  \frac{\sqrt{\det(u_t^*G_0)}}{\sqrt{\det G_0}} \Big|_{t=t_0} \\
 & = & \frac12 \tr_{u^*_{t_0} G_{t_0}} (u_{t_0}^*[L_{W_{t_0}}] )\frac{\sqrt{\det(u_{t_0}^*G_0)}}{\sqrt{ \det G_0}} \\
 & &  + \frac12 \tr_{u^*_{t_0} G_{t_0}} (u_{t_0}^* \dot{G}_{t_0}) \frac{\sqrt{\det(u_{t_0}^*G_0)}}{\sqrt{ \det G_0}} \\
 &=&  \frac12 \tr_{u^*_t G_t} (u_{t_0}^*[L_{W_{t_0}}] )\frac{\sqrt{\det(u_{t_0}^*G_0)}}{\sqrt{\det G_0}}.
\end{eqnarray*}
The second equality is due to the fact that  $\frac12 \tr_{u^*_t G_{t_0}} (u_{t_0}^* \dot{G}_{t_0})$ is zero for
$\frac12 \tr_{G_{t_0}} (\dot{G}_{t_0})$ is zero as $\dot{G}_{t_0}$ is traceless.  When $t_0$ is zero, the fact that $W_0=0$ implies that $\frac{d}{dt} J(u_t, G_t) |_{t=t_0} = 0$.   Combined with $\frac{d}{dt}  e(u_t, G_t) |_{t=t_0} =0$ as shown above, it follows that 
\[
\frac{d}{dt} {\cal E}_{\overline{\del}} (G_t) \Big|_{t=0} = 0.
\] 
Lastly we have 
\begin{eqnarray*}
\frac{d^2}{dt^2} J(u_t, G_t) \Big|_{t=t_0} &=& \frac{d}{dt}  \Big[ \frac12 \tr_{u^*_t G_t} (u_{t}^*[L_{W_{t}}G_{t}] ) \frac{\sqrt{\det(u_{t}^*G_0)}}{\sqrt{ \det G_0}} \Big] \Big|_{t=t_0} \\
 & = & \frac12 \tr_{u_{t_0}^* G_0} (u_{t_0}^* [L_{\dot{W}_{t_0}} G_{t_0}]) \frac{\sqrt{\det(u_
 {t_0}^*G_0)}}{\sqrt{ \det G_0}}  \\
 & & -\langle L_{W_{t_0}} G_{t_0}, L_{W_{t_0}} G_{t_0} \rangle_{u_{t_0}^* G_{t_0}} \frac{\sqrt{\det(u_
 {t_0}^*G_0)}}{\sqrt{ \det G_0}} \\
 &  & - \langle L_{W_{t_0}} G_{t_0}, \dot{G}_{t_0} \rangle_{u_{t_0}^* G_{t_0}} \frac{\sqrt{\det(u_{t_0}
 ^*G_0)}}{\sqrt{ \det G_0}}  \\
 & & + \frac12 \tr_{u^*_t G_t} (u_{t_0}^*[L_{W_{t_0}}G_{t_0}] ) \Big[  \frac{d}{dt}  \frac{\sqrt{\det (u_t^*G_0)}}{\sqrt{\det G_0}} \Big] \Big|_{t=t_0}.  
\end{eqnarray*}
Setting $t_0 = 0$ again, we obtain 
\[
\frac{d^2}{dt^2} J(u_t, G_t) \Big|_{t=0} = \frac12 \tr_{G_0} (L_{\dot{W}_0} G_0).
\]
As the $\overline{\del}$-energy ${\cal E}_{\overline{\del}} (G_t)$ is the integral of $e(u_t, G_t) - J(u_t, G_t)
$, we have 
\[
\frac{d^2}{dt^2} {\cal E}_{\overline{\del}} (G_t) \Big|_{t=0} = \int_{{\mathbb H}^2} \frac{d^2}{dt^2}  \Big[ e(u_t, G_t) - J(u_t, G_t) \Big]  \Big|_{t=0}  d \mu_{G_0} =  \int_{{\mathbb H}^2} \frac12 \tr_{G_0} (\ddot{G}_{0}) d \mu_{G_0}.
\]
Finally at $t=0$, we have the following simple equality
\[
\frac{d}{dt} \tr_{G_t(z)} \dot{G}_t(z) \Big|_{t=0} = - \langle \dot{G}_0(z), \dot{G}_0(z) \rangle_{G_0(z)} + \tr_{G_0(z)} \ddot{G}_0(z) =0
\]
as the \weil geodesic has traceless tangent vectors; $\tr_{G_t} \dot{G}_t=0$ for all $t$.  

Therefore we have 
\[
\frac{d^2}{dt^2} {\cal E}_{\overline{\del}} (G_t) \Big|_{t=0} =   \int_{{\mathbb H}^2} \frac12 \langle \dot{G}_0, \dot{G}_0 \rangle_{G_0} d \mu_{G_0} = \frac12.
\]

\end{proof}

Variations of  \weil convexity have been obtained in several contexts previously.  Wolf \cite{Wo1}  showed that  for the harmonic maps $u_t:(\Sigma, G_0) \rightarrow (\Sigma, G_t)$ for a closed surface $\Sigma$ homotopic to the identity map, the Hessian of the energy functional is equal to twice the \weil pairing. 
Fischer-Tromba \cite{FT3} showed that instead of varying the target hyperbolic metrics, by varying the domain metrics along 
\weil geodesics, the Hessian of the resulting energy functional gives twice the \weil pairing as well.   

Also one should mention the work of Takhtajan-Teo \cite{TT} where they looked at the universal \teich space 
as a union of uncountable  components each of which has a structure of Hilbert manifold.  In particular, 
the connected component containing the identity is a topological group whose Hilbert structure is the 
space of  $H^{3/2}$-integrable vector fields we have encountered in the work of Nag-Verjovsky \cite{NV}.  
Also they formulated another \weil K\"{a}hler potential at the identity, called universal Liouville action.  
One should note that this is not an exhaustive list  of \weil potentials,  as there have been many different 
approaches to the universal \teich space, some from theoretical physics.

\section{Metric Completion and CAT(0) Geometry}
\subsection{Metric completion of  \teich space\index{metric completion of  \teich space}\index{\weil Metric Completion of  \teich Space}}
We have shown the existence of many convex and proper functionals defined on $\T$ with respect to the \weil metric. On the other hand, it has been known (Wolpert \cite{W0}, Chu \cite{Ch}) that the \weil metric is not complete.  Namely 

\begin{proposition}
Suppose that $\sigma: [0, T) \rightarrow \T$ is a \weil geodesic which cannot be extended beyond $T < 
\infty$.  Then for any sequence $\{t_n \}$ with $\lim_{n \rightarrow \infty} t_n = T$, the hyperbolic length of 
the shortest geodesic(s) on the surface $(\Sigma, \sigma(t_n))$ goes to zero.  
\end{proposition}

The statement follows from the so-called Mumford-Mahler compactness of  moduli space (see for example \cite{Jo, Tr1}), namely 
if there exists a lower bound for the injectivity radius, then the \weil geodesic $\sigma$ lies away from the nodal surfaces, namely in the interior of the
\teich space $\T$, contradicting the inextensibility of $\sigma$ beyond $T$.  

Hence along an inextensible \weil geodesic $\sigma: [0, T) \rightarrow \T$, the convex and proper functional ${\cal E}(\sigma(t))$ blows up as $t \rightarrow T$.  We have already seen such an occurence in the proof of the properness of the energy functional.  Namely, when a simple closed geodesic is pinched, then any simple closed geodesic transverse to the pinched loop gets arbitrarily 
long, which would then induce the blow up of the energy of the harmonic map, due to the De Giorgi-Nash-Moser estimate. 

In other words,  the pinching of necks are the only cause of the incompleteness.  This observation is available since the 1970s when Bers \cite{Ber} and Abikoff \cite{Ab} formulated the so-called augmented \teich 
space\index{augmented \teich space}, and when H. Masur  analyzed in 1976 the decay of the \weil metric tensor as a neck is 
pinching.  Masur in his paper even used the notation $\Tbar$ to denote the augmented \teich space. 
In 2000 (2001 arXiv paper \cite{Y5}, part of which  appeared in \cite{Y2} in 2004)  the author proposed  to look at the \weil metric completion $\Tbar$ of the \teich
space $\T$ as a CAT(0) space; a non-positively curved geodesic space, which as a set is the
augmented \teich space. Recall that a metric space $(X, d)$  is a CAT(0) space\index{CAT(0) space}, (or an NPC space\index{NPC space} as called in \cite{KS1} for non-positively curved space)  when
\begin{itemize}
\item   $(X, d)$ is a length space.  That is, for any two points $P$ and $Q$, in $X$, the distance $d(P, Q)$ is realized as the length of a rectifiable curve connecting $P$ and $Q$.  Such curves are called geodesics. 
\item For any three points $P$, $Q$ and $R$ in $X$, and choices of geodesics $\gamma_{PQ}$ of length $r$, $\gamma_{QR}$ of length p, and $\gamma_{RP}$ of length $q$ connecting the respective points, the following comparison property  holds: For any $0<\lambda<1$ write $Q_\lambda$ for the point on $\gamma_{QR}$ satisfying 
\[
d(Q_\lambda, Q) = \lambda p, \,\,\, d(Q_{\lambda}, R) = (1-\lambda) p.
\]
On the (possibly degenerate) Euclidean triangle of side length $p,q$ and $r$, and opposite vertices $\overline{P}, \overline{Q}$ and $\overline{R}$, there is a corresponding point 
\[
\overline{Q}_\lambda = \overline{Q} + \lambda (\overline{R} - \overline{Q}).
\]
The CAT(0) hypothesis is that the metric distance $d(P, Q_\lambda)$ from $Q_\lambda$ to its opposing vertex is bounded above by the Euclidean distance $|\overline{P} - \overline{Q}_\lambda|$.  This can be written as 
\[
d^2(P, Q_\lambda) \leq (1-\lambda) d^2(P, Q) + \lambda d^2(P, R) - \lambda (1-\lambda) d^2(Q, R).
\]
\end{itemize}
When a metric space $(X, d)$ is CAT(0), then it follows that a geodesic is unique 
given its end points, and that  the space is simply connected and contractible \cite{KS1,BH}.  The most familiar examples of
CAT(0) spaces are the simply connected complete Riemannian manifolds having non-positive sectional curvature.
In particular,  a \teich space with \weil metric is a CAT(0) space, although it is an incomplete metric space.

We note that the space $\Tbar$  is not a compactification but a metric completion,
each point representing a Cauchy sequence in $\T$.  The author's contribution in this regard is the simple 
observation that  the triangle comparison property  of the \weil distance function (as appears in the definition of CAT(0) space) on $\T$ extends to the metric completion $
\Tbar$ as  a point-wise convergence of a sequence of convex 
functions produces a convex function; an observation which took nearly 25 years to materialize since the 
paper of Masur's \cite{Ma} appeared, over which period the field of metric space geometry had 
sufficiently matured and begun to be widely studied. 

The space $\Tbar$ is bigger than $\T$ by the set of all  nodal surfaces
resulting from degeneration of the neck-pinchings from the original surface $\Sigma$.   
We describe the setting of the metric completion $\Tbar$ more carefully. 
Detailed descriptions can be found in \cite{Y2} and also in \cite{W3}. We first let $\cal S$  be the free homotopy classes of
homotopically nontrivial simple closed curves on the surface $\Sigma_0$.  This set can be identified with the set of simple closed
geodesics on the surface with a hyperbolic metric. Then define the complex of curves\index{ complex of curves} $C({\cal S})$
as follows.  The vertices/zero-simplices of $C({\cal S})$ are the elements of $\cal S$.
An edge/one-simplex of the complex consists of a pair of homotopy classes of disjoint simple
closed curves. A $k$-complex consists of $k+1$ homotopy classes of mutually disjoint
simple closed curves.  A maximal set of mutually disjoint simple closed curves, which produces
a pants decomposition of $\Sigma_0$, has $3g-3$ elements.  We say a simplex $\sigma$ in
$C({\cal S})$ {\it precedes} a simplex $\sigma'$ provided $\sigma \subset \sigma'$, and we write $\sigma \geq
 \sigma'$. We say a simplex $\sigma$ in  $C({\cal S})$ {\it strictly precedes} a simplex $\sigma'$ provided
$\sigma \subsetneq \sigma'$, and write $\sigma > \sigma'$. This defines a partial
ordering by reverse inclusion in the complex of curves $C({\cal P})$, and thus makes it a partially ordered set
(poset.)  We define the null set to be the $(-1) $-simplex.  Then there is a $C({\cal S}) \cup \emptyset$--valued
function $\Lambda$, called labeling, defined on $\Tbar$ as follows.
Recall a point $p$ in $\T$ represents a marked Riemann surface $(\Sigma, f)$ with an
orientation-preserving homeomorphism $f: \Sigma_0
\rightarrow \Sigma$.  The \weil completion $\Tbar$ consists of bordification points
of $\T$ so that $\Sigma$ is allowed to have nodes, which are geometrically
interpreted as simple closed geodesics of zero hyperbolic length. Thus a point $p$ in $\Tbar \backslash
\T$ represents a marked noded Riemann surface $(\Sigma, f)$ with $f: \Sigma_0
\rightarrow \Sigma$. We now define $\Lambda(p)$ to be the simplex of free homotopy classes on
$\Sigma_0$ mapped to the nodes on $\Sigma$. We denote the fiber of
$\Lambda: \Tbar \rightarrow C({\cal S}) \cup \emptyset$ at a point $\sigma \in C({\cal S}) \cup \emptyset$ by ${\T}_{\sigma}$.
We denote its \weil completed space by ${\Tbar}_\sigma$.
The completed space $\Tbar$ has the stratification\index{stratification} in the sense of \cite{BH}
\[
\Tbar = \cup_{\sigma \in C({\cal S}) \cup \emptyset} {\T}_\sigma
\]
where the original \teich space $\T$ is expressed as ${\T}_\emptyset$.
And each stratum ${\T}_\sigma$ is the \teich space of the nodal surface
$\Sigma_\sigma$.  Here an important fact is that the set of nodal 
surfaces exactly corresponds to the set of admissible degenerations of conformal 
structures while the surfaces are uniformized by 
hyperbolic metrics.
  
The following is a set of properties satisfied by the \weil completion of  \teich space.  
\\

\noindent{\bf Properties of $\Tbar$}\\
1) $\Tbar$ is a union of strata ${\Tbar}_{\sigma}$;\\
2) if ${\Tbar}_{\sigma} = {\Tbar}_{\tau}$ then $\sigma= \tau$; \\
3) if the intersection ${\Tbar}_{\sigma} \cap {\Tbar}_{\tau}$ of two strata
is non-empty, then it is a union of strata;\\
4) for each $p \in {\Tbar}$ there is a unique $\sigma (p) \in C({\cal S})$ such that
the intersection of the strata containing $p$ is ${\Tbar}_{\sigma(p)}$. \\

The author showed in \cite{Y2} that this stratification is very much compatible with the
\weil geometry.  Namely for each
collection $\sigma \in C({\cal S})$, each boundary \teich space ${\T}_\sigma$ is a 
\weil geodesically convex subset of
$\Tbar$. Here geodesic convexity means that given a pair of points in ${\T}_\sigma
$, there is a distance-realizing
\weil geodesic segment connecting them lying entirely in ${\T}_\sigma$. 
We do not reproduce the proof here, but mention a key idea of the proof.
Consider the sub-level set $S(\sigma, \varepsilon) := \{x \, | \,  {\cal L}_\sigma (x) < \varepsilon \}$ of the geodesic length functional ${\cal L}_\sigma$ of a simple closed geodesic indexed by $\sigma$ in $C({\cal S})$, a subset in the CAT(0) space $(\Tbar, d)$.   As the length functional is \weil convex, for each 
$\varepsilon>0$, $S(\sigma, \varepsilon)$ is a convex subset of $\Tbar$.  By varying the value of $\varepsilon$, we have a collection of nested convex sets $S(\sigma, \varepsilon)$: for $0< \varepsilon_1 < \varepsilon_2$ 
\[
S(\sigma, \varepsilon_1) \subset S(\sigma, \varepsilon_2)
\]      
with $S(\sigma, \varepsilon_1) \cap S(\sigma, \varepsilon_2) = S(\sigma, \varepsilon_1)$ convex.
In  light of this fact, the frontier set $\Tbar_{\sigma} = \{x| {\cal L}_\sigma(x) =0\}$ can be regarded as
\[
\cap_{\varepsilon >0} S(\sigma, \varepsilon),
\]
a convex set, as an intersection of all the {\it receding convex sets} $S(\sigma, \varepsilon)$ as $\varepsilon \rightarrow 0$.  

\subsection{The \weil metric tensor near the strata} 
Indeed, one can see what happens locally on the surface $\Sigma$ when a neck pinches to become a node near the frontier sets $\T_\sigma$.  The model case is the standard hyperbolic cylinder $A_{|t|} = \{z \, | \, |t|/c < |z| < c \}$ 
\[
d s^2_{|t|} = \Big(\frac\pi{\log |t|} \csc \frac{\pi \log |z|}{\log |t| } \Big| \frac{dz}{z}\Big| \Big)^2.
\]
Here the hyperbolic length ${\cal L}_0(t)$ of the {\it waist} of the cylinder is equal to $2 \pi^2/\log (1/|t|)$.
As $|t| \rightarrow 0$, the hyperbolic  annulus $(A_{|t|},  d w^2_{|t|} )$ converges pointwise to the hyperbolic metric on two copies of the punctured disc $\{z \, | \, 0< |z| < c\}$ 
\[
ds^2_0 = \Big( \frac{|dz|}{|z| \log|z|}\Big)^2  
\]
which models the standard hyperbolic cusp.  

In \cite{Y2} it was shown that the {\it thin} part of the surface can be approximated by $|\sigma|$ copies of the standard hyperbolic cylinder, and the \weil metric tensor $ds_{\rm WP}^2$ behaves as 
\[
ds^2_{\rm WP} = ds^2_{\sigma} + 4 \pi^3 (1+ O(\|u\|^3) ) 
\Big[ 
\sum_{j=1}^{|\sigma|} du_j^2 + \frac14 (u_j)^6 d \theta_j^2
\Big]
\]
where 
\[
\theta_j = {\rm arg}\, t_j \mbox{ and }  u_j = \Big( \log \frac1{|t_j|} \Big)^{-1/2}.
\] 
where the constant $4 \pi^3$ is due to Wolpert \cite {W4}  The proof of this asymptotic 
expansion follows the line of arguments in  Masur's 1976 paper \cite{Ma}, where the components of the \weil {\it 
cometric} tensor was written as various pairings of holomorphic quadratic differentials.  

This expansion captures the singular behavior of the \weil metric, namely as the necks pinch,  the 
Fenchel-Nielsen twist parameters approximated by $\theta_j$ become increasingly ineffective in terms of \weil norm. Indeed, Wolpert (see his exposition \cite{Wobook}) demonstrated that the \weil sectional curvature of the plane 
spanned by $\{\nabla \ell_j, J \nabla \ell_j \}$ blows down as $O(\ell_j^{-1})$, where $J$ is the \weil complex structure.  Geometrically the blow-down behavior of the \weil sectional curvature describes that the transversal section to the frontier set $\T_\sigma$ is modeled by an ${\mathbb R}$-tree of uncountable degree, or the universal covering space of an incomplete cusp as pictured in \cite{Wobook}, whose vertex/cuspidal point is represented by the nodal surface $\Sigma_\sigma$.  We will utilize this 
\weil asymptotic structure to construct the \weil geodesic completion in the next section.

\subsection{The \weil isometric action of the mapping class group and equivariant harmonic maps}

We consider harmonic maps into the \weil completed \teich space $\Tbar$.
In particular, we are concerned with how the maps behave with respect to the stratification and isometric actions.  It turns out that the harmonic map respects the \weil stratification property as much as it can, in the sense that one can impose an equivariance condition on the harmonic map, and the push-forward $\pi_1$
of the domain as a subgroup of the mapping class group forces the map to stay in certain strata.

It should be remarked at this point that the full \weil isometry group of  \teich space is known to be the extended mapping class group $\widehat{{\rm Map}(\Sigma)}$ \cite{MW}.  Also one can refer to the exposition \cite{DW2} in this Handbook.

We first need to recall 
the Thurston classification theorem\index{Thurston classification theorem of mapping classes} \cite{Th} of elements of the mapping class 
group $\map$.  An element $\gamma$ of $\map$ is classified as one of the following three types:\\
1) it is of finite order, also called periodic or elliptic;\\
2) it is reducible if it leaves a tubular neighborhood of 
a collection $\sigma$ of closed geodesics $c_1,... c_n$ invariant;\\
3) it is pseudo-Anosov (also called irreducible) if there is $r>1$ and transverse measured
foliations $F_+, F_{-}$ such that $\g (F_{+}) = r F_{+}$ and 
$\g (F_{-}) = r^{-1} F_{-}$.  In this case the fixed point set of the
$\g$ action in ${\cal PMF}( \Sigma)$ (the Thurston boundary of $\T$) 
is precisely $F_+, F_{-}$.\\

As for the classification of subgroups\index{classification of mapping class subgroups}, McCarthy and Papadapoulos~\cite{MP} have shown
that  the subgroups of $\map$ is classified into four classes:\\
1) subgroups containing a pair of independent pseudo-Anosov elements
(called {\it sufficiently large subgroups});\\
2) subgroups fixing the pair $\{ F_+ (\g ), F_{-} (\g ) \}$ of fixed points
in ${\cal PMF} (\Sigma)$ for a certain pseudo-Anosov element $\g \in \map$ (such groups
are virtually cyclic);\\
3) finite subgroups;\\
4) infinite subgroups leaving invariant a finite, nonempty system of
disjoint, nonperipheral simple closed curves on $\Sigma$
(such subgroups are called reducible.)

Those group-theoretic classifications are relevant to the \weil geometry in the sense that 
the pseudo-Anosov elements are loxodromic \cite{W2}, \cite{DW}, namely the infimum 
of the translation distance $d(x, \gamma x)$ is achieved on a unique pseudo-Anosov axis in $\T$, and
the reducible (by $\sigma$) elements are loxodromic in the respective stratum $\T_
\sigma$. This implies that there are neither parabolic elements nor parabolic subgroups in the sense of symmetric spaces of noncompact type.  This statement was first claimed by the author in \cite{Y5}, where
a proof was presented by noting that for a one-dimensional harmonic map $u:(a, b) \rightarrow \T$ and for the distance function $d(*, \Tbar_\sigma) $ to the stratum $\T_\sigma$,  the pulled back distance function $u^*d(*, \T_\sigma)$ is super-harmonic, namely concave, then use the minimum principle to show that either the image of the map $u$ entirely lies in the stratum $\T_\sigma$ or otherwise entirely in the interior  $\T$.  The argument is suggestive, but the known differentiability of $d$ near the stratum was not enough to make the proof complete.  Wentworth-Daskalopoulos and Wolpert then came up with proofs by a length comparison argument, or no-refraction argument (see \cite{W2, DW} or \cite{Wobook} Chap.5) that the harmonic image, or equivalently an open \weil geodesic segment cannot lie partly in $\T$ and partly in $\T_\sigma$.   Once this is established, the existence of the pseudo-Anosov axis follows readily.

For harmonic maps with higher dimensional domains, 
we define a functional defined on the \weil\ completion $\Tbar$
of the \teich\ space $\T$.   \\

\begin{definition}
Suppose $\Gamma$ is a finitely generated subgroup of $\map$, with
$\{ {\g}_i \}_{1\leq i \leq l}$ its generators.  Then we define a functional
$\delta$ on $\Tbar$ by
\[
\delta(x) = \max_{1 \leq i \leq l} d(x, {\g}_ix).
\]
\end{definition}

Note here that $\delta: \Tbar \rightarrow {\bf R} \cup \{ \infty \}$
is a convex functional, since each $d(x, {\g}_ix)$ is convex
on $\Tbar$ due to the NPC curvature condition (see for example \cite{BH} II.2), and since
the maximum of finitely many convex functionals is again convex.  \\

\begin{definition}
Given a subgroup $\Gamma$ of $\map$, the isometric action of $\Gamma$
on $\Tbar$ is said to be {\it proper} if the sublevel set
\[
S(M) = \{ x \in \Tbar : \delta(x) < M < + \infty \} 
\]
is bounded in $\Tbar$ for all $M < \infty$.
\end{definition}

A simple yet important consequence of the lack of parabolic elements in the mapping class group, which follows from the results in \cite{KS2, DW}, is that a   representation 
$\rho : \pi_1 (M) \rightarrow {\rm Isom}( {\Tbar} )$, where $\rho(\pi_1 (M))$ is a sufficiently large subgroup of $\map$ induces a proper isometric action on the \teich space. 
Furthermore, when $\rho(\pi_1 (M)) $ is reducible by $\sigma \in C({\cal S})$, then the isometric action of $\rho(\pi_1 (M))$ on $\T_\sigma$ is proper.
A general framework developed by Korevaar-Schoen \cite{KS2} provides existence of equivariant 
harmonic maps for representations which induces a proper isometric action on the \teich space.
\\

\noindent {\bf Theorem (Existence)}
{\it Suppose $M$ is a compact manifold without boundary.  Suppose that a
representation $\rho : \pi_1 (M) \rightarrow {\rm Isom}( {\Tbar} )$ induces  
an isometric action on the \teich space by a sufficiently large subgroup $\rho(\pi_1 (M))$.  Then there exists an energy minimizing harmonic map 
$u : \tilde{M} \rightarrow \Tbar$ which is $\rho$-equivariant.
($\tilde{M}$ is the universal covering space of $M$.) Moreover the 
map $u$ is Lipschitz continuous. 
}
\\

When $M$ is not compact but complete, under mild additional conditions which 
often are met for many applications, we still have an existence theorem.

As for uniqueness of the map, it follows from the standard argument using the so-called geodesic homotopy \cite{Ha}, with some extra attention needed when the harmonic map touches upon several distinct strata.  A complete proof is provided in  \cite{Y6}.
\\

\noindent {\bf Theorem (Uniqueness)}
{\it The  harmonic map is unique
within the class of finite energy maps which are $\rho$-equivariant, provided
that the image of $\rho$ is not reducible by any element $\sigma$ of $C({\cal S})$ and that the image of the map is not contained in a geodesic.
}\\

Note that the condition for the theorem is satisfied when $u_* \pi_1(M)$ is a sufficiently large subgroup of the mapping class group.

We will discuss an application.
A K\"{a}hler manifold $M$ of real dimension four is said to have a structure of a holomorphic
Lefschetz fibration\index{Lefschetz fibration} if the following descriptions hold.
There exists a holomorphic map
$\Pi : M^4 \rightarrow B^2$ where the base space $B$ is a surface, such as 
${\mathbb C}P^1$.
The map $\Pi$ has finitely many critical points $N_i, \,\,  i=1,..., n$ 
in disjoint fibers $F_i = \Pi^{-1} (P_i), \,\,  i=1, ..., n$, each of which is 
a nodal surface, while away from those disjoint fibers, each  fiber of the 
map $\Pi$ is a Riemann surface of varying conformal structures of a fixed genus
$g$. The neighborhood of each critical point $N_i$ 
can be described by local complex coordinates  $z, w$ on $M$ and 
$t$ on $B$ such that $\Pi: (z, w) \mapsto zw (= t)$ where 
$N_i = (0, 0) \in {\bf C}^2$ and $P_i = 0 \in {\bf C}$.  

The picture above can be transcribed as saying that there exists a 
$\rho$-equivariant holomorphic map $\tilde{u} :\widetilde{B \backslash \{ P_i \}}
\rightarrow \T$, where $\rho: \pi_1 (B \backslash \{ P_i \}) 
\rightarrow \map $ is the monodromy representation of the fibration.

The uniqueness theorem of harmonic maps has an immediate application, which was
first proved by Shiga~\cite{Sh} by a 
different method.\\
\\
\noindent {\bf Corollary}
{\it Given a holomorphic Lefschetz fibration of higher genus, its monodromy representation is sufficiently
large.} 

\begin{proof}
It is well known that a holomorphic map between two K\"{a}hler 
manifolds is energy-minimizing~\cite{Sa}.  Hence the map $\tilde{u}
: \widetilde{B \backslash \{ P_i \}} \rightarrow \Tbar$ 
is the unique $\rho$-equivariant harmonic map whose existence
and uniqueness have been so far established.  To see that $\Gamma$
is sufficiently large, note that if it weren't, then we have three other possible
cases.  The first being when $\Gamma$ is a finite group can be
excluded since each local monodromy is of infinite order.
The second being the case when $\Gamma$ is virtually cyclic, fixing
a pair of points in the Thurston boundary.  Then the image of
the $\rho$-equivariant harmonic map $\tilde{u} : \widetilde{B \backslash \{ P_i \}} 
\rightarrow \Tbar$
is a $\Gamma$ invariant \weil\ geodesic, which lies entirely in 
the interior of the \teich\ space $\T$.  The projection of the invariant
geodesic down to the moduli space is a loop located 
away from any of the divisors, which in turn says that there is no
sequence of points $\{ q_j \}$ in $B \backslash \{P_i\}$ 
over which a cycle on 
the Riemann surface represented by $u(P_i)$ vanishes
(or equivalently a neck is pinched), which contradicts  the
fact that $M^4$ has singular fibers/vanishing cycles.  

Lastly the third possible case to be excluded is 
when $\Gamma$ can be reduced by a collection
of $C$ mutually disjoint simple closed curves.  
Then the harmonic image of $\tilde{u}$ is entirely contained
in $\overline{{\T}_C}$, which implies that every fiber is a Riemann
surface with nodes where the nodes are obtained by pinching each
simple closed curves in $C$.  
This certainly is not the case when $M$ is a Lefschetz fibration. 
\end{proof}

\section{Geodesic Completion and CAT(0) geometry}

\subsection{Construction of the geodesic completion}
Having established the \weil metric completion $\Tbar$, it is natural to seek for a geodesic completion\index{geodesic completion}\index{\weil geodesic completion} of the \teich space with respect to the \weil distance function. Such a space was constructed in \cite{Y3}
where it is named \teich Coxeter complex\index{\teich Coxeter complex}. We will briefly describe the construction of the space below.

The theory of Coxeter group was developed in an attempt to understand
combinatorial and geometric characterizations of tessellation of standard spaces 
such as ${\mathbb R}^n$,
$S^n$ and ${\mathbb H}^n$ by reflecting convex polygons across the sides of the 
polygons.  the sides of those polygons are totally geodesic in the ambient space
as they function as the interfaces between two open convex sets. Each reflection
is of order two, and the group generated by all the reflections is called the Coxeter group.
Naturally the geometry of the vertices (given as the cone angles) of the polygon comes into the structure of the Coxeter group.
Around 1960, Jacques Tits  introduced the notion of an abstract reflection group, which he
named ``Coxeter group" $(W, S)$. It is a group $W$ generated by a set of reflections $S$ and a collection
of relations among the reflections $\{ (ss')^{m(s, s')} \}$. Here $m(s, s')$ denotes the order of
$ss'$ and the relations range over all unordered pairs $s, s' \in S$ with $m(s, s') \neq \infty$.
In other words, $m(s, s') = \infty$ means no relation between $s$ and $s'$.

The data $[m(s, s')]_{(s, s')}$ can be regarded as a matrix, and is said to constitute a Coxeter matrix.
the following list presents a set of evidence that the \weil geometry of $\Tbar$ and the Coxeter theory are indeed compatible
\begin{itemize}
\item The \teich space $\T$ is \weil geodesically convex.
\item For every element $\sigma$ of the complex of curves $C({\cal S})\cup \emptyset$, the frontier set $\Tbar_\sigma$ is a complete convex subset of $\Tbar$, altogether forming a stratified space $\Tbar$. 
\item Given a point $p$ in ${\T}_\sigma \subset {\Tbar}$,  the Alexandrov tangent cone with respect to the \weil distance function is isometric to
${\mathbb R}^{|\sigma|}_{\geq 0} \times T_p {\T}_\sigma$, where ${\mathbb R}^{|\sigma|}_{\geq 0}$ is  the 
orthant in ${\mathbb R}^{|\sigma|}$ with the standard metric (Wolpert \cite{W3}). 
\item $\Tbar$ is the closure of the \weil convex hull of
the vertex set given by the {\it zero-dimensional} \teich spaces of  the maximally degenerate surfaces $\{ {\T}_\sigma \,\,\, | \,\,\, |\sigma|=3g-3 \}$ (Wolpert \cite{W2}). 
\end{itemize}

The Alexandrov tangent cone\index{Alexandrov tangent cone} $C_q X$ of a CAT(0) space (see for example \cite{BH}) is a space of equivalence classes of constant 
speed geodesics originating at $q$ where two geodesics are deemed equivalent when they share the same speed and they form a zero Alexandrov  angle\index{Alexandrov angle}.  To be precise, the Alexandrov angle is defined by
\[
\cos \angle (\gamma_0, \gamma_1) = \lim_{t \rightarrow 0} \frac{d(q, \gamma_0(t))^2 + d(q, \gamma_1(t))^2 - d(\gamma_0(t), \gamma_1(t))^2}{2 d(q, \gamma_0(t)) d(q, \gamma_1(t))}.
\]

The significance of the \weil tangent cone structure in our context is that it describes the geometry 
around the vertices, given as the \weil tangent cone angles, when $\Tbar$ is seen 
as a convex polygon. This picture specifies a particular choice of the Coxeter 
matrix.
Namely for each $\sigma$ with $|\sigma|=1$, one can {\it reflect}
$\Tbar$ across the totally geodesic stratum ${\Tbar}_\sigma$.  Now for $\tau = \sigma \cup \sigma'$
with  $\sigma$ and $\sigma'$ representing a pair of disjoint simple closed 
geodesics, the relation
$m(s_\sigma, s_{\sigma'}) =2$ has a geometric meaning where four copies of $
\Tbar$
can be glued together around a point $q \in {\T}_\tau$ to form a  space whose tangent cone
at $q$ is a union of four copies of  ${\mathbb R}^{2}_{\geq 0} \times T_p {\T}_\tau$ (each ${\mathbb R}_{\geq 0}^2$
is regarded as a quadrant in the plane) isometric to ${\mathbb R}^2  \times T_p {\T}_\tau$ on which the reflections $s_\sigma, s_{\sigma'}$ act as reversing of the orientations of the $x, y$ axes for ${\mathbb R}^2$.

Hence we define a Coxeter group $(W, S)$\index{Coxeter group}  by letting the generating set $S$ be the elements of ${\cal S}$, and the relations among the elements of the generating set are specified by the Coxeter matrix\index{Coxeter matrix} ${m(s,s')}$ whose components satisfy
i) $m(s,s)=1$, ii) if $s \neq s'$, and if there is some simplex $\sigma$ in $C({\cal S})$
containing $s$ and $s'$, then define $m(s,s')=2$, and
iii) if $s \neq s'$, and if the geodesics representing $s$ and $s'$ intersect on $\Sigma_0$ then
$m(s,s') = \infty$.  This group has a {\it geometric realization} acting on a collection of copies of $\Tbar$'s.  

We remark that the Coxeter group
with such a Coxeter matrix as above is said to form a {\it cubical} complex\index{cubical complex}, for it has a canonical geometric realization where each generating element is represented as a linear orthogonal reflection across the face of the infinite dimensional unit cube
in ${\mathbb R}^{|S|}$.  

However we instead form a geometric realization   $D({\Tbar}, \iota)$ 
as the set which is the quotient of $W \times {\Tbar}$ by the following
equivalence relation
\[
(g, y) \sim (g', y') \Longleftrightarrow y = y' \mbox{ and } g^{-1}g' \in W_{\sigma (y)}
\]
where ${\Tbar}_{\sigma (y)}$ denotes the smallest stratum containing $y$, and the subgroup
$W_{\sigma (y)}$ fixes the stratum ${\Tbar}_{\sigma (y)}$ pointwise.  We write $[g,y]$
to denote the equivalence class of $(g, y)$. Furthermore $\iota$ denotes a simple
morphism of groups, which specifies a system of subgroups $W_\sigma \subset W$, compatible with 
the poset structure of the complex of curves ${\cal C}({\cal S})$ (see \cite{Y3} for details.)   

 The remarkable phenomena here is that despite the fact that the generating set is
infinite, we have a geometric realization of the Coxeter group $(W, S)$ action (which is very far from linear)
on a space modeled on a finite dimensional space $\T$, albeit the partial bordification $\Tbar$ encodes non-locally compact geometry due to the singular behavior of the \weil metric tensor.  Note that the singularity is
also manifest in the fact that the reflecting wall $\Tbar_\sigma$ with $|\sigma|=1$ is of complex codimension
one, instead of  being a real hypersurface as in the standard Coxeter theory, a situation reminiscent of the theory of complex reflection group (see~\cite{ST} for example.)  We also make a remark that
the \weil metric defined on each stratum ${\T}_\sigma$ is K\"{a}hler, hence the space $D({\Tbar}, \iota)$  providing a geometric realization of the Coxeter group $W$ is
a ``simplicial" complex with each face equipped with a K\"{a}hler metric, a situation unattainable in
a  {\it real} simplicial complex, where the reflecting walls are real hypersurfaces.
The space obtained by the action of the Coxeter group on $\Tbar$,
which we will call development $D({\Tbar}, \iota)$, is then shown
to be ${\rm CAT}(0)$ via the Cartan-Hadamard theorem\index{Cartan-Hadamard theorem} \cite{AB}, and also to be
geodesically complete. The reason for the negative curvature is due to the following geometric
construction;\\

\noindent {\bf Theorem}(Y. Reshetnyak~\cite{Re}, Bridson-Haefliger~\cite{BH})  {\it Let $X_1$ and $X_2$ be ${\rm CAT}(0)$ spaces (not necessarily complete) and let $A$ be a complete metric space.  Suppose that for $j=1,2$, we are given isometries $i_j : A \rightarrow A_j$. where
$A_j \subset X_j$   is assumed to be convex.  Then  $X_1 \sqcup_A X_2$ is a ${\rm CAT}(0)$ space.} \\

Here the space $X_1 \sqcup_A X_2$ is the quotient space of the disjoint union $X_1 \coprod X_2$ by the equivalence relation
generated by $i_1(a) \sim i_2(a)$ for all $a \in A$.  The resulting space, called gluing space or amalgamation of the CAT(0) spaces $X_1$ and $X_2$ along $A$ has a canonical distance between $x \in X_j$ and $y \in X_{j'}$ defined as follows:
\begin{align*}
d(x, y) = &  d_j(x, y) = d_{j'}(x, y) &  \mbox{ if } j = j' \\
d(x, y) = &  \inf_{a \in A} \{ d_j (x, i_j(a)) + d_{j'} (x, i_{j'}(a)) \} &  \mbox{ if } j \neq j'
\end{align*}

This is used in \cite{Y3}  to glue the copies of $\Tbar$ along the strata $\Tbar_\sigma$, where the number of copies 
at each $x \in \Tbar$ is determined by the number of nodes $|\sigma|(x)|$; for example if $|\sigma|(x)| = 2$, the tangent cone at the point $x$ is the product of a first quadrant  in ${\mathbb R}^2$ and the tangent space of $T_x \T_{\sigma(x)}$, and by putting togehter $4= 2^{|\sigma(x)|}$ copies of $\Tbar$, the 
tangent cone of the resulting enlarged space is isometric to  ${\mathbb R}^2$ times the tangent space of 
$\T_{\sigma(x)}$.  This glueing construction provides a locally CAT(0) space around each point in $\Tbar$.  Then knowing that the Coxeter complex  $D({\Tbar}, \iota)$ is simply connected, one can apply  the Cartan-Hadamard theorem to identify   $D({\Tbar}, \iota)$ with a geodesically complete CAT(0) space.

\subsection{Finite rank\index{finite rank} properties of ${\Tbar}$}

The \weil metric defined on $\T$ is a smooth Riemannian metric whose sectional curvature is negative everywhere (\cite{Jo}, \cite{Tr1}, \cite{W4}.)
Hence the only flats (i.e. isometric embeddings of Euclidean space ${\mathbb R}^n, n \geq 1$) are the \weil geodesics.
The fact that there is no strictly negative upper bound for the sectional curvature is explained by the fact that the
\weil completion $\Tbar$ does have higher dimensional flats (see the concluding remark in~\cite{Y2}, as well as
Proposition 16 in~\cite{W2}.)   Those flats arise when the collection $\sigma$ of mutually disjoint simple closed
geodesics separates the surface $\Sigma$ into multiple components.  Then the frontier set ${\T}_\sigma$
is a product space of the \teich spaces of the components separated by the nodes.  The number of the connected components is bounded by
$g+(\big[\frac{g}{2}\big]-1)$, which is achieved when the surface $\Sigma_\sigma$ is a union of $g$ 
once-punctured tori
and $\big[\frac{g}{2}\big]-1$ four-times-punctured spheres. One can construct isometric embedding of Euclidean spaces
of dimension up to $g+(\big[\frac{g}{2}\big]-1)$ by considering a set of \weil geodesic lines, whose existence is
established in \cite{DW} and \cite{W2}, in the different components of
the product space ${\T}_\sigma$.  In this sense the \weil completion $\Tbar$ is a space of {\it finite rank} where the
rank is bounded by $g+(\big[\frac{g}{2}\big]-1)$.   This rank has been known to coincide with  Brock-Farb's geometric rank of \map, as studied by
Behrstock-Minsky \cite{BM}.

There is another definition of rank, which we call  FR, as first appeared in \cite{KS3}.
\\

\begin{definition}\label{FR} An NPC (${\rm CAT}(0)$) space $(X, d)$ is said to be an {\bf FR} space if there exist $\varepsilon_0 >0$ and
$D_0$ such that any subset of $X$ with diameter $D > D_0$ is contained in a ball of radius $(1-\varepsilon_0)D/\sqrt{2}$.
\end{definition}

We make several remarks about FR spaces.  The definition can be interpreted as follows. If $X$ is an FR space, then among
all the closed bounded convex sets $F$ in $X$ with its diameter larger than $D_0$, there exists some positive integer $k$
such that
\[
\inf_{F \subset X} \frac{D(F)}{R(F)}\geq \sqrt{2} \sqrt{\frac{k+1}{k}}
\]
where $D(F)$ and $R(F)$ are the diameter and the circum-radius of $F$ respectively.
It is well known that ${\mathbb R}^k$ is FR with the optimal/largest choice of $\varepsilon_0
= 1- \sqrt{k/k+1} >0$ which is realized by the standard k-simplex.  An infinite-dimensional
Hilbert space is not an FR space, while a tree is an FR space with $\varepsilon_0
= 1- \sqrt{1/2} >0$.  It was shown in~\cite{KS2} that a Euclidean building is an FR space with
$\varepsilon_0 = 1- \sqrt{k/k+1} >0$ with $k$ the dimension of chambers.  A CAT(-1) space
(e.g. the hyperbolic plane ${\mathbb H}^2$) is FR with $\varepsilon_0$ which can be made
arbitrarily close to $1-
\sqrt{1/2}$   by taking the value of $D_0$ large.  Heuristically the number $\varepsilon_0>0$ detects the maximal
dimension of flats inside the space $X$, that is, the rank of the given NPC space.

We show in \cite{Y3} that

\begin{theorem}
The \weil geodesic completion $D({\Tbar}, \iota)$ of a \teich space $\T$ is FR.
\end{theorem}

\begin{corollary}
The \weil completion $\Tbar$ of a \teich space $\T$ is FR.
\end{corollary}

The proof of the theorem determines a lower bound of $\varepsilon_0$ to be
$1-\sqrt{k/k+1}$ with $k =6g-6$, which is larger (for $g>1$) than the {\it maximal} dimension
of the flats as described above, which was $g+(\big[\frac{g}{2}\big]-1)$.  The particular value of
$k$ here should be regarded as the maximal dimension of a flat.  Namely if
there is a flat, its dimension cannot exceed $6g-6$.    On the other hand, the existence of those flats arising from
the product structure of the frontier sets of $\Tbar$ does not necessarily imply that the space is FR, as there may be
infinite dimensional flats
elsewhere.  The definition of  FR spaces utilizes only the convexity
of the distance function to describe the finite dimensionality of possibly existent flats,
without directly dealing with the singular behavior of the \weil metric tensor near the frontier set $\partial {\Tbar}$.
The statement of the theorem~\cite{Y3} says that despite the lack of local compactness near the frontier set $\partial {\Tbar}$,
the \weil completion $\Tbar$ of the \teich space $\T$ exhibits  finite rank characteristics.

\subsection{\weil geodesic completeness}

Given a Riemannian manifold $M$ and a codimension-{\it two} submanifold $S \subset M$, the open manifold
$N:=M\backslash S$ has $M$ as its metric completion as well as the
geodesic completion with respect to the Riemannian distance
function. (Consider $M={\mathbb R}^2$, $S=\{0\}$ and $N$ the punctured
plane, for example.) This of course is expected with the Hopf-Rinow 
theorem available in the manifold setting, which in effect demonstrates
the equivalence of metric completeness and geodesic completeness.  

Here the analogous picture is given by taking $N = \T$, $S=
{\Tbar} \backslash {\T}$, and $M$ being either $\Tbar$ or $D({\Tbar},
\iota)$, depending on whether the completion is taken to be metric
or geodesic. The disparity, that $\Tbar$ is metrically complete
but not geodesically complete, is caused by the singular behavior
of the \weil metric near the strata, where the points in the
frontier set can be modelled as vertex points of cusps~\cite{W2}.

A consequence of the geodesic extension property (namely each geodesic can be extended to a geodesic line) of the development $D({\Tbar}, \iota)$ is that for each point $p \in D({\Tbar}, \iota)$,
the inverse map ${\rm exp}_p^{-1}$ of the ``exponential map" from $D({\Tbar}, \iota)$ to the
tangent cone $C_p D({\Tbar}, \iota)$, which is isometric to ${\mathbb R}^{|\sigma(p)|}$,  is  surjective,
as every geodesic segment starting at $p$ can be extended to a geodesic line so that the image by the inverse exponential map is  an entire real line through the origin of
${\mathbb R}^{|\sigma(p)|}$.  It is precisely this point that will be needed in the proof of the finite rank theorem.  Namely the
finite rank of the space $D({\Tbar}, \iota)$ is shown by using Caratheodory's theorem\index{Caratheodory's theorem} about convex hulls in
${\mathbb R}^n$, which then implies the inequality between the diameter $D$ and the circumradius $R$ of convex sets
in $D({\Tbar}, \iota)$. Without the surjectivity of  the inverse exponential map  ${\rm exp}_p^{-1}$, the
correspondence between the space of geodesics and the space of directions breaks down.  Also one notes that the
geodesic completeness
is understood in the sense that any geodesic segment can be extended to {\it a} geodesic line, and there may be
more than one extension (in fact uncountably many extensions) making ${\rm exp}_p$ multi-valued.
This is once again due to the singular behavior of the \weil metric tensor near the frontier sets, where the sectional
curvature can blow down to $-\infty$, which causes that the behavior
of geodesics resembles that of geodesics in ${\mathbb R}$-trees.

One should also mention that the development $D({\Tbar}, \iota)$ can be seen from  billiard theory.  
This point was raised in the proof of existence of pseudo-Anosov axes by Wolpert \cite{W2} where a conditional sequential compactness of \weil geodesics is established.  A convergence can be guaranteed up to  Dehn twists at the strata the geodesics hit.  Namely a billiard ball is bounced back at each stratum with equal incoming and outgoing angles  at the tangent cone level, but not in the nonlinear level once \weil exponentiated.  In Section 5 of \cite{Y3}, the author has  laid out a detailed comparison between Wolpert's statement and the situation for the development $D({\Tbar}, \iota)$: The billiard ball goes through the wall/stratum only to find on the other side not knowing which directions to go in the Fenchel-Nielsen twist directions. In short, the billiard ball trajectory is deterministic in the tangent cone at the origin of the \weil geodesic, but highly non-deterministic due to the ${\mathbb R}$-tree like structure at the strata.      

\subsection{\weil isometric action and symmetry of $D({\Tbar}, \iota)$}
It was shown (\cite{MW}) that the full isometry group of ${\Tbar}$ is the extended mapping 
class group. Recall that Royden showed that the mapping class group
is the full isometry groups with respect to the \teich distance. Thus the isometry group of $D({\Tbar}, \iota)$ contains a group which is the semi-direct product of the extended mapping class group and the Coxeter group, in which the extended mapping class group is a normal subgroup.  For a Coxeter complex, there is a natural construction of an isometry group of the
Coxeter complex which contains the original Coxeter group as a normal subgroup.
In the current context, the fundamental domain is the \weil completion $\Tbar$ of the \teich
space $\T$, on which the extended  mapping class group $\widehat{{\rm Map}_{\Sigma}}$ acts isometrically.
The extended mapping class group $\widehat{{\rm Map}_{\Sigma}}$ is known (\cite{Iv},\cite{Kor},\cite{Lu}) to be the full automorphism group of the complex of curves $C({\cal S})$.  Using this fact, it has been
shown~\cite{MW} that the extended mapping class group is indeed the full isometry group
of the \weil completed \teich space.  Note that each element
$\gamma$ of the extended mapping class group $\widehat{{\rm Map}_\Sigma}$ preserves the Coxeter matrix, namely
\[
m(\gamma(s), \gamma(t)) = m(s,t)
\]
As the Coxeter group $W$ is generated by ${\cal S}$, and the group $W$ is
completely determined by the Coxeter matrix $[m(s,t)]_{s,t \in {\cal S}}$,
it follows that  each element $\gamma$ in $\widehat{{\rm Map}_{\Sigma}}$ induces an
automorphism of $W$. Such an automorphism of $W$ is called {\it diagram automorphism}\index{diagram automorphism}
\cite{Da}.

The formalism laid out in M.Davis' book gives us a natural action (Proposition 9.1.7~\cite{Da}) of the
semi-direct product $G:=W \rtimes \widehat{{\rm Map}_\Sigma}$ on the development
$D({\Tbar}, \iota)$ as follows: given $u=(g, \gamma) \in G$ and
$[g', y] \in D({\Tbar}, \iota)$,
\[
u \cdot [g', y] := [g \gamma (g'), \gamma y]
\]
where $\gamma(g')$ is the image of $g'$ by the automorphism of $W$ induced by
$\gamma:C({\cal S}) \rightarrow C({\cal S})$.

Clearly the action $G \hookrightarrow D({\Tbar}, \iota)$ is isometric. Thus $G$ is a
subgroup of the isometry group ${\rm Isom}(D({\Tbar}, \iota))$.  It remains an open question
whether this group is indeed the full isometry group, and if not, how much larger the isometry group is.

By applying a result of Davis' book (Lemma~5.1.7~\cite{Da}) we note also the following.

\begin{theorem}
The action of the Coxeter group $W \subset G $
on $D({\Tbar}, \iota)$ is properly discontinuous.
\end{theorem}

For a proof, see \cite{Y7}

\subsection{Embeddings of the Coxeter complex into ${\cal UT}$}

A loosely formulated  guiding philosophy in studying  \teich spaces
is that the geometry of the space is somehow inherited from the geometry of the 
Riemann surfaces it is parameterizing.  In this section, we proceed
along this line of thinking by considering the space of embeddings of the Coxeter
complex into the universal \teich space.

We first note that each simple geodesic in the closed hyperbolic surface $\Sigma$ can 
act as a mirror introducing a reflective ${\bf Z}_2$ symmetry to a doubled cover.  The symmetry is defined by first providing another copy of the surface, then cutting across the simple closed geodesic $c$,
in both the original surface and the new copy, and lastly identifying the four ends by pairs
so that  for each simple closed geodesic $c'$ transverse to $c$, the union $c' \cup c'$
of the original $c'$ and the new copy $c'$  
is either a simple closed geodesic in the new surface $\Sigma \cup \Sigma$, or a pair of simple
closed geodesics. We denote the resulting surface $\Sigma \sqcup_c \Sigma$. 

The distinction here is caused by the nature of the simple closed geodesic $c$.  If $c$ is non-separating,
namely the punctured surface $\Sigma \backslash \{c\}$ consist of a path-connected component, 
then $\Sigma \sqcup_c \Sigma$ is path-connected, and $c' \cup c'$ is a single simple closed geodesic. 
When $c$ is separating, $\Sigma \sqcup_c \Sigma$ consists of two copies of $\Sigma$.   
In the former case, the genus of the surface $\Sigma \sqcup_c \Sigma$ is $2g-1$, in the latter case $2g$. 

In the case of a separating geodesic $c$, we can embed the surface 
$\Sigma \sqcup_c \Sigma \cong \Sigma \coprod \Sigma$ into
a surface of genus $2g-1$ by cutting each of $\Sigma$'s at a simple closed 
non-separating geodesic $c''$ disjoint from $c$ 
and pasting the two surfaces along (the four copies of) $c''$.  We denote the resulting surface by $\Sigma \sqcup_c^{c''} \Sigma$.
Geometrically it is a surface of genus $2g-1$ with a ${\bf Z}_2$-symmetry across the pair of simple closed
geodesics corresponding to $c''$.  

Recall the construction of the Coxeter-\teich complex $D({\Tbar}, \iota)$, a quotient space 
$W \times {\Tbar}/\sim$ whose points are written as $[g, y]$. We identify a point $[g, y]$ with
a point in \teich space of higher genus surface as follows.  One characteristic of the 
Coxter compex is that each element $g$ in $W$ is written as a product of generators $\Pi_{i=1}^N s_i$
of the Coxeter group $W$.  For  $s_1= s_{c_1}$ in the product, we extend the surface $\Sigma$ by introducing an 
unramified double cover $\Sigma \sqcup_{c_1} \Sigma$ which has a symmetry across two copies of $c_1$.  Note the 
unramified cover $\Sigma \sqcup_{c_1} \Sigma$ needs  to be decorated by $c_1''$ in case $c_1$ is a separating geodesic, which we
have suppressed for now.  For $s_2 = s_{c_2}$ in the product, we next  extend the surface $\Sigma \sqcup_{c_1} \Sigma$ 
by introducing an unramified double cover $(\Sigma \sqcup_{c_1} \Sigma) \sqcup_{c_2} (\Sigma \sqcup_{c_1} \Sigma)$ 
which has a symmetry across four copies of $c_2$.  Inductively we can define a tower of double covers over the original
surface $\Sigma$, its largest cover has genus $\phi \circ \cdots \circ \phi(g)$ ($\phi$ composed $N$ times) 
where $\phi(g) = 2g-1$. We denote the resulting surface $\sqcup_w \Sigma$. The fact that the correspondence $w \mapsto \sqcup_w \Sigma$ is well-defined modulo the choices of $c_i''$'s follows from the properties of the Coxeter group \cite{Bo}. 

The \weil metric needs to be normalized by the volume $|\Sigma_g|$, so that the \weil metric on $\Sigma_g$ and \weil metric
$\Sigma_{\phi(g)}$ with a ${\bf Z}_2$ symmetry are compatible. Namely
\[
\langle h_1, h_2 \rangle_{{\rm WP}} = \frac{1}{\chi(\Sigma_g)}\int_{\Sigma_g} \langle h_1(x), h_2(x) 
\rangle_{G(x)} \,\,\, d \mu_G(x). 
\]
 This 
follows from the equalities
\[
|\Sigma(\phi(g))|= \chi(\phi(g)) = 2\phi(g)-2 = 2(2g-1)-2 = 2(2g-2) = 2 \chi(g) = 2 |\chi(g)|.
\]
This simply says that by the doubling procedure via a reflection across
a simple closed geodesic, the volume is doubled, 
which can be normalized by the topological invariant $\chi$ to have a 
well-defined \weil metric.

This construction of higher genus Riemann surfaces 
$\{ \sqcup_w \Sigma \}_{w \in W}$ provides a way of embedding the Coxeter-\teich 
complex $D({\Tbar}, \iota)$ \weil isometrically into the universal \teich 
space ${\cal UT}$\index{\weil isometric embedding the Coxeter-\teich 
complex  into the universal \teich 
space}.  We need to remind 
ourselves that the embedding is not unique for two reasons. First there 
are no canonical 
embeddings of the \teich spaces ${\T}_g$ in ${\cal UT}$ for $g \geq 1$. 
Secondly, we recall that when $g \in W$ is generated by $s_c$ with $c$ a 
separating
simple closed geodesic, we need to enlist an extra parameter $c''$ to obtain 
a path-connected double cover.  

Understanding the space of embeddings is far from complete, and we hope to 
make things better organized in the near future.

\section{\teich Space as a \weil Covex Body\index{\weil convex body}}
\subsection{$\Tbar$ as a convex subset in $D(\Tbar, \iota)$}

The Coxeter complex setting allows 
to view the \teich space as a \weil  convex set in an ambient space $D(\Tbar, \iota)$, bounded by a set of
 complex-codimension one ``supporting hyperplanes"  $\{D(\Tbar_\sigma, \iota) \,\, | \,\, |\sigma|=1\}$ of 
the frontier stratum $\{\Tbar_\sigma\}$ with each $\sigma$ 
representing a {\it single}  node. Every boundary point is contained in at least one of the set of the supporting hyperplanes $\{D(\Tbar_\sigma, \iota) \,\, | \,\, |\sigma|=1 \}$.  In this picture, each  $D(\Tbar_\sigma, \iota)$ is a totally geodesic
set, metrically and geodesically complete, and when $D(\Tbar_{\sigma_1}, \iota)$ and  $D(\Tbar_{\sigma_2}, 
\iota)$ intersect along  $D(\Tbar_{\sigma_1 \cup \sigma_2}, \iota)$, they meet at a right angle.  

One can also look at  the translates of $\{D(\Tbar_\sigma, \iota)\,\, | \,\, |\sigma|=1\}$
by the action of the Coxeter group $W$. They form a right-angled grid structure in $D(\Tbar, \iota)$,
whose lattice points are the orbit image by the Coxeter group $W$ of the set $\{\Tbar_\theta \,\, | \,\, |\theta|=3g-3 \}$ with $\theta$ indexing the maximal set of nodes on the surface.

Under this setting, for each $\sigma$ with $|\sigma| =1$, consider a {\it half-space}, namely the set $H_{\sigma}$, containing $\Tbar$ in the
$D(\Tbar, \iota)$, and bounded by $D(\Tbar_\sigma, \iota)$. 
We note here the fact obtained by Wolpert~\cite{W3} that the \weil metric completion $\Tbar$ is the closure of the convex hull of the
vertex set  $\{\Tbar_\theta \,\, | \,\, |\theta|=3g-3 \}$, which suggests an interpretation of the \teich space 
as a simplex.  

We can summarize the  above discussion as 
\[
\Tbar = \cap_{\sigma \in {\cal S}} H_{\sigma} \,\,\,  \mbox{ with } \,\,\, 
\partial \Tbar \subset \cup_\sigma D(\Tbar_\sigma, \iota).
\]
where every boundary point $b \in \partial \Tbar $ belongs to $D(\Tbar_\sigma, \iota)$ for some $\sigma$ in $\cal S$.  

\subsection{Euclidean convex geometry and Funk metric\index{Funk metric}\index{metric!Funk}}
Suppose that $\Omega$ is an open convex subset in a Euclidean space ${\mathbb R}^d$. In what follows, we set the presentation by Papadopoulos and Troyanov~\cite{PT1} as our reference for Funk 
and Hilbert metrics\index{Hilbert metric}\index{metric!Hilbert}.

First we represent the convex set $\Omega$ as 
\[
\Omega = \cap_{\pi(b) \in {\cal P}} H_{\pi (b)}
\]
where $H_{\pi(b)}$ is the half-space bounded by a supporting
hyperplane $\pi(b)$ of $\Omega$ at the boundary point $b$, containing the convex set $\Omega$.
The index set $\cal P$ is the set of all supporting hyperplanes of $\Omega$. 
That for every boundary point $p$ there exists a supporting hyperplane $\pi(b)$ follows from the 
convexity of $\Omega$.  In general, there can be more than one supporting hyperplane
of $\Omega$ at $p \in \del \Omega$. 
The index set $\cal P$ is identified with the set of unit normal vectors to the supporting hyperplanes.
It is identified with a subset of $S^{d-1}$, is equal to the entire sphere when the convex set is bounded.
We denote by ${\cal P}(b)$ the set of supporting hyperplanes at $b \in \del \Omega$.

\begin{definition}
For a pair of points $x$ and $y$ in $\Omega$, the Funk asymmetric metric~\cite{F} is defined by 
\[
F(x, y) =  \log \frac{d(x, b(x,y))}{d(y,  b(x,y))}. 
\]  
where the point $b(x,y)$ is the intersection of the boundary $\del \Omega$ and the ray $\{x+t \xi_{xy}\,\, | \,\, t>0\}$ 
from $x$ though $y$. Here $\xi_{xy}$ is the unit vector along the ray.
\end{definition}
\noindent{\bf Remark.} In this section only, we use the term {\it metric} on a set  $X$ for 
a function $\delta: X \times X \rightarrow ({\mathbb R}_+ \cup \{ \infty \} $ satisfying:
\begin{enumerate}
\item $\delta(x,x) = 0$ for all $x$ in $X$,
\item $\delta(x, z) \leq \delta(x, y) + \delta(y, z)$ for all $x,y$ and $z$ in $X$.
\end{enumerate} 

Now let $\pi_0$ be a supporting hyperplane at $b(x,y)$, namely $\pi_0 \in {\cal P}(b(x,y))$.
Then  note the similarity of the triangle $\triangle (x, \Pi_{\pi_0}(x), b(x,y))$
and $\triangle (y, \Pi_{\pi_0}(y), b(x,y))$, where $\Pi_{\pi_0}(p)$ is the foot of the point $p$ on the hyperplane $\pi_0$, or put it differently $\Pi_{\pi_0} : {\bf R}^d \rightarrow \pi_0$ is the nearest point projection map. This says that 
\[
\log \frac{d(x, b(x,y))}{d(y,  b(x,y))} = \log \frac{d(x, \pi_0)}{d(y,  \pi_0)}.
\]
Also by the similarity 
argument of triangles, note that the right hand side of the equality is independent of the 
choice of $\pi_0$ in ${\cal P}(b(x, y))$.

Using the convexity of $\Omega$, the quantity $F(x, y)$ can be characterized variationally as follows. 
Define $T(x,\xi, \pi)$ by $\pi \cap \{x+t \xi  | t > 0\}$ with $\pi \in {\cal P}$ where $\xi$ is a unit vector. 
Consider the case $\xi = \xi_{xy}$ where $\xi_{xy}$ is the unit tangent vector at $x$ to the ray from $x$ through $y$.  When the hyperplane supports $\Omega$ at $p$, we have $T(x,\xi_{xy}, \pi) = b(x, y)$.
and otherwise the point $T(x, \xi_{xy}, \pi)$ lies outside $\Omega$.  
When $\pi \notin {\cal P}(b(x,y))$,  by the similarity argument between the triangles
$\triangle (x, F_{\pi}(x), T(x, \xi_{xy}, \pi))$ and $\triangle (y, F_{\pi}(y), T(\xi_{xy}, \pi))$ again we have 
\[
\frac{d(x, \pi)}{d(y, \pi)} =\frac{d(x, T(x,\xi_{xy},b))}{d(y, T(x,\xi_{xy}, b))}.
\]
Now the point $b(x,y) = T(x,\xi_{xy}, \pi)$ is actually the closest point to $x$ along the ray $\{x+t \xi_{xy} : t>0 \}$.
This in turn says that  $\pi$ which supports $\Omega$ at $b(x,y)$ maximizes the quantity $\frac{d(x, T(x,\xi_{xy}, 
\pi))}{d(y, T(x,\xi_{xy}, \pi))}$ among all elements of ${\cal P}$;  

\[
 \log \frac{d(x, \pi(b(x,y)))}{d(y,  \pi(b(x, y)))} = \sup_{\pi \in {\cal P}} \log \frac{d(x, \pi)}{d(y,  \pi)}. 
\]

Hence we have a new characterization of the Funk metric \cite{Y4};  
\begin{theorem} The Funk metric defined as above over a convex subset $\Omega \subset {\mathbb R}^d$ has the following variational formulation\index{variational formulation of Funk metric}\index{Funk metric!variational formulation of}:
\[
F(x, y) = \sup_{\pi \in {\cal P}} \log \frac{d(x, \pi)}{d(y,  \pi)}.
\]    
\end{theorem}

We note that though variational formulations of {\it Hilbert} metric has been known (for example \cite{Li} for polygons), there had not been none available for the Funk metric previous to \cite{Y4}.  

\subsection{\weil Funk metric}
We now transcribe the Euclidean Funk geometry as well as its compatible Finsler structure in the previous section to the \weil setting. 

First note that as each $\Tbar_\sigma$ lies in $\Tbar$ as a complete convex set, for each point $x \in \Tbar$, there 
exists the nearest point projection $\Pi_\sigma (x) \in D(\Tbar_\sigma, \iota)$, and the \weil geodesic $\overline{x 
\pi_\sigma(x)}$ meets  $D(\Tbar_\sigma, \iota)$ perpendicularly, its length uniquely realizing the distance
$\inf_{y \in \Tbar_\sigma} d(x, y) = d(x, \Pi_\sigma (x))$.  We denote this number by $d(x, \Tbar_\sigma)$. 
We also introduce the notation $\nu_\sigma(x)$   for the unit vector
at $x$ along the \weil geodesic  between $x$ and $\Pi_\sigma(x)$.  In particular $- \nu_\sigma(x)$ is the \weil 
gradient vector of the function $d(x, \Tbar_\sigma)$.  

\begin{definition}
We define the Weil-Petersson-Funk metric\index{Weil-Petersson-Funk metric}\index{metric!Weil-Petersson-Funk} $F$ on $\T$ as
\[
F(x, y) = \sup_{\sigma \in {\cal S}} \log \frac{d(x, \Tbar_\sigma)}{d(y, \Tbar_\sigma)}. 
\]
\end{definition}

In order to make the analogy with the Euclidean setting more obvious, and in order to make clearer the viewpoint that \teich space is a convex body within an ambient space, we can instead define the metric,  as
\[
F(x, y) = \sup_{\sigma \in {\cal S}} \log \frac{d(x, D(\Tbar_\sigma, \iota))}{d(y, D(\Tbar_\sigma, \iota))}
\]
We are allowed to replace $\Tbar_\sigma$ by $ 
D(\Tbar_\sigma, \iota)$ in the above definition since for any $z \in \T$, we know that $\Pi_{\sigma}(z)$ is in $\T_{\sigma} \subset D(\iota, \Tbar_{\sigma})$
due to the fact that the frontier sets intersect perpendicularly. \\

The equality follows from the discussion in the paragraph preceding the definition of \weil Funk metric.

Note that the triangle inequality for the \weil Funk metric follows from the following:
\begin{eqnarray*}
F(x, y) + F(y, z) & = & \sup_{\sigma \in {\cal S}} \log \frac{d(x, D(\Tbar_\sigma, \iota))}{d(y, D(\Tbar_\sigma, \iota))} +  \sup_{\sigma \in {\cal S}} \log \frac{d(y, D(\Tbar_\sigma, \iota))}{d(z, D(\Tbar_\sigma, \iota))} \\
 & \geq  & \sup_{\sigma \in {\cal S}} \Big( \log \frac{d(x, D(\Tbar_\sigma, \iota))}{d(y, D(\Tbar_\sigma, \iota))} +   \log \frac{d(y, D(\Tbar_\sigma, \iota))}{d(z, D(\Tbar_\sigma, \iota))} \Big) \\
   & = &  \sup_{\sigma \in {\cal S}} \log \frac{d(x, D(\Tbar_\sigma, \iota))}{d(z, D(\Tbar_\sigma, \iota))} = F(x, z)
\end{eqnarray*}

\subsection{The \teich metric, Thurston's Asymmetric Metric and the Weil-Petersson-Funk metric}
In this section, we make a comparison among three Funk type metrics defined on  \teich spaces.
The first is the \teich metric\index{variational formulation of \teich metric}\index{metric!variational formulation of \teich}, which is defined as:
\begin{definition} 
Let $[G_1]$ and $[G_2]$ be two conformal structures (uniformized by hyperbolic metrics $G_i$) on $\Sigma$.  The \teich distance 
between $[G_1]$ and $[G_2]$ is given by 
\[
d_T ([G_1], [G_2]) = \frac12 \inf_f \log K(f)
\]
where the infimum is taken over all quasi-conformal homeomorphisms $f : (\Sigma, [G_1]) \rightarrow (\Sigma, [G_2])$ that are isotopic to the identity.  
\end{definition}

Kerckhoff \cite{Ke2} showed that the \teich distance can be 
alternatively defined as
\[
d_T([G_1], [G_2]) = \frac12 \sup_{\sigma \in {\cal S}} \log \frac{{\rm Ext}_{[G_1]}(\sigma)}{{\rm Ext}_{[G_2]}(\sigma)}
\]
where ${\rm Ext}_{[G]}(\sigma)$ is the extremal length of the homotopy class of simple closed curves in $\Sigma$.  Recall \cite{PT2}  that the extremal length of $\sigma$ is defined as $1/{\rm Mod}_{\Sigma}(\sigma)$ where  ${\rm Mod}_{\Sigma}(\sigma)$ is the supremum of the moduli of the topological cylinders embedded in $\Sigma$ with core curve in the class $\sigma$. 

Secondly, Thurston's asymmetric metric\index{Thurston's asymmetric metric}\index{metric!Thurston's asymmetric} is defined as 
\begin{definition}
Let $G_1$ and $G_2$ be the two hyperbolic metrics on $\Sigma$.  A distance function can be 
defined as
\[
T(G_1, G_2) = \sup_{\sigma \in {\cal S}} \log \frac{{\ell}_\sigma(G_1)}{{\ell}_\sigma(G_2)}
\]
called Thurston's asymmetric metric.
\end{definition}
This quantity was shown by Thurston to be equal to the following number 
\[
L(G_1, G_2) = \inf_{\phi \sim {\rm Id}_{\Sigma}} {\rm Lip}(\phi)
\]
where the infimum is taken over all diffeomorphisms $\phi$ in the isotopy class of the identity, and $ {\rm Lip}(\phi)$ is the Lipschitz constant of the map $\phi$;
\[
{\rm Lip}(\phi) = \sup_{x \neq y \in \Sigma} \frac{d_{G_2}(\phi(x), \phi(y))}{d_{G_1}(x, y)}. 
\]
For this reason, the quantity $T(G_1, G_2)$ is sometimes called Thurston's Lipschitz metric.  
Thurston in his paper  (\cite{Th1} Chapter 4) emphasizes the underlying convex geometry for the 
metric.  In particular, the space of projective measured laminations is embedded into the cotangent space $T_{G}^* \T$ for a fixed point 
$G$ in $\T$ as the boundary set of a convex body, where the embedding is given by the differential of logarithm of  
geodesic length function of geodesic laminations 
\[
d \log {\rm length} : PL(\Sigma) \rightarrow T_{G}^* \T.
\]
where the map only registers the projective classes of geodesic laminations for one is taking  the 
logarithmic derivative of the length.

Now recall the new Funk-type metrics  we have introduced above;
\begin{definition}
The Weil-Petersson-Funk metric $F$ on $\T$ is defined as
\[
F(x, y) = \sup_{\sigma \in {\cal S}} \log \frac{d(x, \Tbar_\sigma)}{d(y, \Tbar_\sigma)}. 
\]
\end{definition}

Now the analogy among  the \teich metric, the Thurston metric and the Weil-Petersson-Funk metric is clear; namely for each of them we have 
an embedding  
\[
\Theta : \T \rightarrow {\mathbb R}^{\cal S} 
\]
where the target space has a weak metric $d(x, y) = \sup_\sigma \log \frac{x_\sigma}{y_\sigma}$ for 
$x = (x_\sigma)_{\sigma \in {\cal S}}$, and 
each of the three Finsler metrics is the pulled-back metric $\Theta^* d$ defined on $\T \times \T$. 

The author would like to thank H. Miyachi for pointing out this comparison among the three metric structures.



\begin{thebibliography}{1}


\bibitem{Ab} W. Abikoff, Degenerating families of Riemann surfaces. {\it Ann. of Math.}(2) 105 (1977) 29--44.


\bibitem{Ah1} L. Ahlfors, Some remarks on Teichm\"{u}ller's space of Riemann surfaces. {\it Ann. of Math.} (2) 74 (1961), 171--191.

\bibitem{Al} S. Al'bers,  Spaces of mappings into a manifold with negative curvatures.  {\it Soviet Math. Dokl. } 9 (1968), 6--9. 

\bibitem{AB} S. Alexander and R. Bishop, The Hadamard-Cartan theorem in locally convex spaces. {\it Enseign.Math.}
36 (1990), 309--320.

\bibitem{BM} J. Behrstock and Y. Minsky,  Dimension and rank for mapping class groups.  {\it Ann. of Math.} (2)  167 (2008), 1055--1077.

\bibitem{BPT} A. Belkhirat, A. Papadopoulos and M. Troyanov, Thurston's weak metric on the Teichm\"{u}ller space of the torus. {\it Trans. Amer. Math. Soc.}  357 (2005),  3311--3324. 

\bibitem{Ber} L. Bers,  Spaces of degenerating Riemann surfaces.  In {\it Discontinuous groups and Riemann surfaces}
Ann. of Math. Studies 79,   Princeton University Press, Princeton, NJ., 1974, 43--55, 

\bibitem{Be} A. Besse, {\it Einstein Manifolds},   Springer, Berlin Heidelberg 1987.

\bibitem{AB} A. Beurling and L. Ahlfors, The boundary correspondence under quasi-conformal mappings. {\it Acta Math.} 96 (1956), 125--142. 

\bibitem{Bo} N. Bourbaki,  {\it Groupes et Alg\'{e}bres de Lie}, Hermann, Paris 1981.

\bibitem{BH}  M. Bridson and A. Haefliger,  {\it Metric spaces of non-positive curvature},
Springer, Berlin Heidelberg 1999.

\bibitem{Ch} T. Chu, The \weil metric on the moduli space. {\it Chinese J. Math.} 4 (1976), 29--51.

\bibitem{CO} P. Candelas and X. de la Ossa, Moduli space of Calabi-Yau manifolds.
{\it Nuclear Physics} B 355 (1991), 455--481. 

\bibitem{DW} G. Daskalopoulos and R. Wentworth,  Classification of \weil isometries. {\it Amer. J. Math.}  125  (2003),  941--975.

\bibitem{DW2} G. Daskalopoulos and R. Wentworth, Harmonic maps and \teich theory. In  {\it Handbook of Teichm\"{u}ller theory} (A. Papadopoulos ed.), Volume I,  EMS Publishing House, Z\"urich 2007, 33--119. 

\bibitem{Da} M. Davis,  {\it The Geometry and Topology of Coxeter Groups.}
Princeton Univ. Press, Princeton NJ,  2007.

\bibitem{DE} A. Douady and C. Earle, 
Conformally natural extension of homeomorphisms of the circle.
{\it Acta Math.} 157   (1986), 23--48.

\bibitem{EE} C. Earle and J. Eells,  Fiber bundle description of \teich theory. {\it  J. Differential Geom.} 3 (1969), 19--43. 

\bibitem{EL} J. Eells and L. Lemaire, Deformation of metrics and associated harmonic maps.  Patodi Memorial Volume, {\it Geometry and Analysis} Tata Inst. Bombay  1980, 33--45. 

\bibitem{EL2} J. Eells and L. Lemaire, {\it Two reports on harmonic maps.} World Scientific Publishing Co., Inc., River Edge, NJ, 1995.

\bibitem{ES} J. Eells and J. Sampson,
Harmonic Mappings of Riemannian Manifolds
{\it Am. J. of Math.}
86 (1964), 109--160. 

\bibitem{FM} A. Fischer and J. Marsden,  Deformations of the scalar curvature. {\it Duke Math. J.}  43  (1975), 519--547. 

\bibitem{FT1} A. Fischer and A. Tromba, On the \weil metric on the \teich space. {\it Trans. A.M.S. } 42  (1975), 319-335. 

\bibitem{FT2} A. Fischer and A. Tromba,  On a purely ``Riemannian'' proof of the structure and dimension of the unramified moduli space of a compact Riemann surface. {\it Math. Ann.}  267   (1984), 311--345.

\bibitem{FT3}  A.Fischer and A.Tromba,  A new proof that Teichm\"{u}ller space is a cell. {\it  Trans. Amer. Math. Soc.} 303  (1987), 257--262.

\bibitem{F} P. Funk,  \"{U}ber geometrien, bei denen die geraden die k\"{u}rzesten sind. {\it Math. Ann.}  101 (1929), 226--237.

\bibitem{GS} M. Gromov and R. Schoen, Harmonic mappings into singular spaces and p-adic 
super-rigidity for lattices in groups of rank one.  {\it Publ. IHES} 76 (1992), 
165--246.

\bibitem{GT} D. Gilbarg and N. Trudinger, {\it Elliptic Partial Differential Equations of Second Order} Springer, Berlin Heidelberg New York 1983.

\bibitem{Har} P. Hartman, On homotopic harmonic maps. {\it Canad. J. Math.}  19  (1967), 673--687. 

\bibitem{Jo} J. Jost,  {\it Two-dimensional geometric variational problems.}   Wiley--Interscience Publication,  John Wiley \& Sons, Ltd., Chichester, 1991.

\bibitem{JS} J. Jost and R. Schoen, 
On the existence of harmonic diffeomorphisms.
{\it Invent. Math.}  66  (1982), 353--359. 

\bibitem{Ham} R. Hamilton,  The inverse function theorem of Nash and Moser. {\it Bull. Amer. Math. Soc.} (N.S.)  7 (1982),  65--222.

\bibitem{Iv} N. Ivanov,  Automorphisms of complexes of curves and of \teich spaces.
{\it Inter. Math. Res. Notices.} 14 (1997), 651--666.

\bibitem{Ke2} S. Kerckhoff,  The asymptotic geometry of \teich space. {\it Ann. of Math.} (2)  66 (1980), 235--265. 

\bibitem{Ke} S. Kerckhoff, The Nielsen realization problem. {\it Ann. of Math.} (2) 117 (1983), 235--265. 

\bibitem{Ko} N. Koiso, Variation of harmonic mapping caused by a deformation of Riemannian metric. {\it Hokkaido Math. J.} 8   (1979),  199--213.

\bibitem{Kor} M. Korkmaz,   Automorphisms of complex of curves on punctured spheres and on punctured tori. {\it Topology Appl.} 95 (1999), 85-111.

\bibitem{KS1} N. Korevaar and R. Schoen, Sobolev spaces and harmonic maps for metric target spaces.
{\it Comm. Anal. Geom. } 1 (1993), 561--659.

\bibitem{KS2} N. Korevaar and R. Schoen, Global existence theorems for 
harmonic maps to non-locally compact spaces.  {\it Comm. Anal. Geom. }  5 (1997),
333-387.

\bibitem{KS3} N. Korevaar and R. Schoen,  Global existence theorems for harmonic maps: finite rank spaces and an approach to rigidity for smooth actions.  Preprint (1997).

\bibitem{Ku} E. Kuwert, Harmonic maps between flat surfaces with conical singularities. {\it Math.
Z.} 221 (1996), 421-436.



\bibitem{Le} O. Lehto, {\it Univalent functions and \teich spaces.} Graduate Texts in Mathematics 109, Springer-Verlag, New York 1987.

\bibitem{Lei} M. Leite, Harmonic mappings of surfaces with respect to degenerate metrics.
{\it Amer. J. Math.} 110 (1988), 399--412.

\bibitem{Li} B. Lins, A Denjoy-Wolff theorem for Hilbert
metric nonexpansive maps on a polyhedral domain.
{\it Math. Proc. Camb. Phil. Soc.} 143 (2007), 157--164.

\bibitem{LT} P. Li and L.-F. Tam, Uniqueness and regularity of proper harmonic maps.  {\it Ann. of Math.} 126 (1993), 168--203. 

\bibitem{Lu} F. Luo,  Automorphisms of the complex of curves.  {\it Topology} 
39  (2000), 283--298.

\bibitem{Ma} H. Masur,  The extension of the \weil metric to the boundary 
of \teich space.  {\it Duke Math.}  43 (1976), 267--304.

\bibitem{Me} C. Mese. A variational construction of the \teich map. 
{\it Calc. Var. Partial Differential Equations}  21 (2004), 15--46 . 

\bibitem{Mi} Y. Minsky.  Harmonic maps, length, and energy in \teich space. {\it J. Diff. Geom.}  35  (1992), 151--217. 

\bibitem{Mo} C. Morrey. On the solutions of quasi-linear elliptic partial differential equations. {\it Trans. A.M.S.} 43  (1938), 126--166.

\bibitem{MM}  P. Michor and D. Mumford,  An overview of the Riemannian metrics on spaces of curves using the Hamiltonian approach.
{\it Appl. Comput. Harmon. Anal.} 23   (2007), 74--113. 

\bibitem{MW}  H. Masur and M. Wolf,  The \weil isometry group. {\it Geom. Dedicata},  93, 177--190 (2002).

\bibitem{MP} J. McCarthy and A. Papadapoulos,  Dynamics of Thurston's 
sphere of projective measured foliations.  {\it Comm. Math. Helv.}
64 (1989), 133--166.

\bibitem{Na} S. Nag, {\it The complex analytic theory of \teich spaces.} John Wiley \& Sons Inc., New York, 1988. 

\bibitem{NS} S. Nag and D. Sullivan,  Teichm\"{u}ller theory and the universal period mapping via quantum calculus and the $H^{1/2}$ space on the circle. {\it Osaka J. Math.}  32 (1995), 1--34. 

\bibitem{NV} S. Nag and A. Verjovsky,  ${\rm Diff} S^1$ and the \teich 
spaces. {\it Comm. Math. Phys.}  130 (1990), 123-138.

\bibitem{Os} R. Osserman.  {\it A Survey of Minimal Surfaces} Dover Publications, Minora New York, 1986.

\bibitem{PT2} A. Papadopoulos and G. Th\'{e}ret, On Teichm\"{u}ller's metric and Thurston's asymmetric metric on Teichm\"{u}ller space. {\it Handbook of Teichm\"{u}ller theory} Vol. I, IRMA Lect. Math. Theor. Phys.  11, Eur. Math. Soc., Z\"{u}rich  (2007),  111--204.  

\bibitem{PT1} A. Papadopoulos and M. Troyanov,  Weak Finsler structures and the Funk weak metric. {\it Math. Proc. Cambridge Philos. Soc.} 147,  419--437,  (2009). 

\bibitem{Re} Y. Reshetnyak, On the theory of spaces of curvature not 
greater than K.  {\it Mat.Sb.} 52 (1960), 789--798.

\bibitem{Sa} J. Sampson,  Applications of harmonic maps to K\"{a}hler geometry. {\it Complex differential geometry and nonlinear differential equations.}  Brunswick, Maine (1984) 125--134.

\bibitem{Sh} H. Shiga,  On monodromies of holomorphic families of Riemann surfaces and modular transformations. {\it Math. Proc. Cambridge Philos. Soc.}  122 (1997),  541--549.   

\bibitem{ST} G. Shephard and J. Todd,  Finite unitary reflection groups. {\it Canadian J. Math.}   6 (1954), 274--304. 

\bibitem{SY0} R. Schoen and S.-T. Yau,  Existence of incompressible minimal surfaces and the topology of three dimensional manifolds with non-negative scalar curvature.  {\it Ann. of Math.} 110 (1979), 127--142.

\bibitem{SY} R. Schoen and S.-T. Yau, {\it Lectures on Harmonic Maps}. International Press, Boston 1997.

\bibitem{Th1} W. Thurston,  Minimal stretch maps between hyperbolic surfaces. Preprint, 1986;  {\tt arXiv:math/9801039v1} [math.GT]. 

\bibitem{Th} W. Thurston,  On the geometry and dynamics of diffeomorphisms
of surfaces.  {\it Bull. A.M.S.} 19 (1988), 417-431.

\bibitem{Tr1}A. Tromba, {\em Teichm\"{u}ller Theory in Riemannian Geometry}. Birkh\"{a}user, Basel 1992. 

\bibitem{Tr2} A. Tromba,  Dirichlet's energy of \teich moduli space and the Nielsen realization problem. {\it Math. Zeit.} (1996), 315--341.
 
\bibitem{TT} L. Takhtajan and L.-P. Teo,  {\it Weil-Petersson Metric on the Universal Teichmuller Space.}  Mem. Amer. Math. Soc.  183, Amer. Math. Soc., Providence, RI,  2006.

\bibitem{Wo1} M.Wolf,  The Teichm\"{u}ller theory of harmonic maps. {\it J. Diff. Geom.}  29 (1989),  449--479 . 

\bibitem{Wo2} M.Wolf,  The \weil Hessian of length on \teich space. {\it J. Diff. Goem.} 91 (2012), 129--160.
 
\bibitem{W0} S. Wolpert,  Noncompleteness of the Weil-Petersson metric for \teich space. {\it Pacific J. Math.} 61 (1975), 573--577.

\bibitem{W4} S. Wolpert,  Chern forms and the Riemann tensor for the moduli
space of curves. {\it Invent. Math.} 85 (1986), 119--145.

\bibitem{W1}   S. Wolpert,  Geodesic length functions and the Nielsen problem.  {\it J. Diff. Goem.}  25  (1987), 275-296. 
 
\bibitem{W2} S.Wolpert,  Geometry of the Weil-Petersson completion of 
\teich space.  In {\it Surveys in Differential Geometry}
VIII: {\it Papers in Honor of Calabi, Lawson, Siu and Uhlenbeck.}  Intl.Press, 
Cambridge, MA, 2003.

\bibitem{W3} S.Wolpert,  Behavior of geodesic-length functions on 
\teich space. {\it J. Diff. Geom.}  79 (2008), 277--334.

\bibitem{Wohb} S.Wolpert,    The Weil-Petersson metric geometry.  In  {\it Handbook of Teichm\"{u}ller theory} (A. Papadopoulos ed.), Volume II,  EMS Publishing House, Z\"urich 2009,   47--64. 

\bibitem{Wobook} S.Wolpert,  {\it Families of Riemann surfaces and \weil Geometry.} CBMS series 113 Amer. Math. Soc., Providence, RI, 2010.   


\bibitem{Y1} S.Yamada,  Weil-Petersson convexity of the energy functional on classical and universal \teich spaces,
{\it J. Diff. Geom.}  51 (1999), 35--96.

\bibitem{Y5} S.Yamada,  Weil-Petersson Completion of Teichmueller Spaces and Mapping Class Group Actions. Preprint, 2001;  {\tt 	arXiv:math/0112001v1}[math.DG]. 

\bibitem{Y2} S.Yamada,   On the geometry of Weil-Petersson completion of \teich spaces.  {\it Math. Res. Let.}  11 (2004),  327--344.

\bibitem{Y7} S.Yamada,   On Weil-Petersson symmetry of moduli Spaces of Riemann surfaces. In {\it Proc. of the 16th OCU Intl. Acad. Symp. } (Y. Ohnita, M. Guest, R. Miyaoka, and W. Rossman ed.), OCAMI Studies 3, Osaka Municipal Universities Press 2009, 67--78.

\bibitem{Y3} S.Yamada, \weil geometry of Teichm\"{u}ller-Coxeter complex and its finite 
rank property. {\it Geom. Dedicata} 145(2010), 43--63.

\bibitem{Y6} S.Yamada,  Some Aspects of Weil-Petersson Geometry of \teich Spaces. {\it Surveys in Geometric Analysis and Relativity}  Advanced Lectures in Mathematics 20   (2011), 529--544.

\bibitem{Y4} S.Yamada,  Convex Bodies in Euclidean and \weil geometries.   {\it Proc. Amer. Math. Soc.} 142 (2014), 603--616.





\end{thebibliography}
\end{document}